\documentclass[11pt,a4paper]{article}
\usepackage{soul}
\usepackage{amssymb}
\usepackage{dsfont}
\usepackage{graphicx} 
\usepackage{amsmath}
\usepackage{amsthm}
\usepackage{enumitem}   
\usepackage[dvipsnames]{xcolor}
\usepackage{hyperref}
\hypersetup{
    colorlinks=true,
    linkcolor=Green,
    citecolor=Green,
    urlcolor=Cyan
    }
\usepackage{soul}
\usepackage{diagbox}
\usepackage{multirow}
\usepackage{subcaption}
\usepackage[titletoc,title]{appendix}
\usepackage{titlesec}
\usepackage{pgfplotstable}
\usepackage{pgfplots}
\usepgfplotslibrary{fillbetween}
\pgfplotsset{compat = newest}

\titleformat{\paragraph}
{\normalfont\normalsize\bfseries}{\theparagraph}{1em}{}
\titlespacing*{\paragraph}
{0pt}{3.25ex plus 1ex minus .2ex}{1.5ex plus .2ex}

\def\num#1{\numx#1}\def\numx#1e#2{{#1}\mathrm{e}{#2}}
\newcommand*{\collauthor}[2]{{#1}$^{#2}$}

\newcommand*{\affiliation}[2]{$\mbox{}^{{#2}}${#1}}
\newcommand*{\colltitle}[1]{\textbf{#1}}

\newenvironment{keywords}[1]{\vspace{1cm}\\{\bf \slshape{Keywords}}\quad\slshape{#1}}{}
\DeclareMathOperator*{\argmin}{arg\,min}
\newtheorem{example}{Example}[]
\newtheorem{lemma}{Lemma}[section]
\newtheorem{theorem}{Theorem}[section]

\newtheorem{assumption}{}
\newenvironment{assump}[1]{%
  \renewcommand\theassumption{#1}%
  \assumption
}{\endassumption}

\DeclareMathOperator{\Tr}{Tr}
\DeclareMathOperator{\Hessian}{Hess}

\newcommand{\bfrac}{\mathfrak{b}}
\newcommand{\ifrac}{\mathfrak{i}}

\newcommand{\Kf}{\mathfrak{K}}
\newcommand{\lfrac}{\mathfrak{l}}

\newcommand{\qfrac}{\mathfrak{q}}

\newcommand{\pfrac}{\mathfrak{p}}

\def\num#1{\numx#1}\def\numx#1e#2{{#1}\mathrm{e}{#2}}

\begin{document}
\begin{center}
\begin{Large}
  \colltitle{A backward differential deep learning-based algorithm for solving high-dimensional nonlinear backward stochastic differential equations}
\end{Large}
\vspace*{1.5ex}

\begin{sc}
\begin{large}
\collauthor{Lorenc Kapllani}{}
and
\collauthor{Long Teng}{}
\end{large}
\end{sc}
\vspace{1.5ex}

\affiliation{Chair of Applied and Computational Mathematics,\\
Faculty of Mathematics and Natural Sciences,\\
University of Wuppertal,\\
Gau{\ss}str. 20, 42119 Wuppertal, Germany\linebreak }{} \\
\end{center}

\section*{Abstract}
In this work, we propose a novel backward differential deep learning-based algorithm for solving high-dimensional nonlinear backward stochastic differential equations (BSDEs), where the deep neural network (DNN) models are trained not only on the inputs and labels but also the differentials of the corresponding labels. This is motivated by the fact that differential deep learning can provide an efficient approximation of the labels and their derivatives with respect to inputs. The BSDEs are reformulated as differential deep learning problems by using Malliavin calculus. The Malliavin derivatives of solution to a BSDE satisfy themselves another BSDE, resulting thus in a system of BSDEs. Such formulation requires the estimation of the solution, its gradient, and the Hessian matrix, represented by the triple of processes $\left(Y, Z, \Gamma\right).$ All the integrals within this system are discretized by using the Euler-Maruyama method. Subsequently, DNNs are employed to approximate the triple of these unknown processes. The DNN parameters are backwardly optimized at each time step by minimizing a differential learning type loss function, which is defined as a weighted sum of the dynamics of the discretized BSDE system, with the first term providing the dynamics of the process $Y$ and the other the process $Z$. An error analysis is carried out to show the convergence of the proposed algorithm. Various numerical experiments up to $50$ dimensions are provided to demonstrate the high efficiency. Both theoretically and numerically, it is demonstrated that our proposed scheme is more efficient compared to other contemporary deep learning-based methodologies, especially in the computation of the process $\Gamma$. 
\begin{keywords}
backward stochastic differential equations, high-dimensional problems, deep neural networks, differential deep learning, malliavin calculus, nonlinear option pricing and hedging
\end{keywords}

\section{Introduction}
\label{sec1}
In this paper, we are concerned with the numerical solution of the decoupled forward-backward stochastic differential equation (FBSDE) of the form
\begin{equation}
    \begin{split}
        \left\{
            \begin{array}{rcl}
                X_t & = & x_0 + \int_{0}^{t} a \left(s, X_s\right)\,ds + \int_{0}^{t} b \left(s, X_s\right)\,dW_s,\\
   	   		    Y_t & = & g\left(X_T\right) + \int_{t}^{T} f\left(s, X_s, Y_s, Z_s\right)\,ds -\int_{t}^{T}Z_s\,dW_s,
            \end{array} \forall\, t \in [0, T]
        \right. 
    \end{split}
\label{eq1}
\end{equation}
where $W_t = \left( W_t^1, \ldots, W_t^d \right)^\top$ is a $d$-dimensional Brownian motion, $a: [0, T] \times \mathbb{R}^{d} \to \mathbb{R}^{d}$, $b: [0, T] \times \mathbb{R}^{d} \to \mathbb{R}^{d \times d}$, $f:\left[0,T\right]\times\mathbb{R}^{d}\times\mathbb{R}\times\mathbb{R}^{1\times d} \to\mathbb{R}$ is the driver function and $g: \mathbb{R}^{d} \to \mathbb{R}$ is the terminal condition which depends on the final value $X_T$ of the forward stochastic differential equation (SDE). Hence, the randomness in the backward stochastic differential equation (BSDE) is driven by the forward SDE. Usually, the coupled FBSDE is referred to as a FBSDE. Hence, to avoid confusion, we refer to the decoupled FBSDE~\eqref{eq1} as a BSDE. We shall work under the standard well-posedness assumptions of~\cite{Pardoux1990} to ensure the existence of a unique solution pair of~\eqref{eq1}.

The main motivation for studying BSDEs lies in their significance as essential tools for modeling problems across various scientific domains, including finance, economics, physics, etc., due to their connection to partial differential equations (PDEs) through the well-known (nonlinear) Feynman-Kac formula. As an illustrative example of their applications in finance, it was demonstrated in~\cite{El1997} that the price and delta hedging of an option can be represented by a BSDE. Such an approach via a BSDE has a couple of advantages when compared with the usual one of considering the associated PDE. Firstly, the delta hedging strategy is inclusive in the BSDE solution. Secondly, many market models can be presented in terms of BSDEs, ranging from the Black-Scholes model to more advanced ones such as local volatility models~\cite{labart2011parallel}, stochastic volatility models~\cite{Fahim_2011}, jump-diffusion models~\cite{Eyraud_Loisel_2005}, defaultable options~\cite{ANKIRCHNER_2010}, and many others. Thirdly, BSDEs can also be used in incomplete markets~\cite{El1997}. Furthermore, using BSDEs  eliminates the need to switch to the so-called risk-neutral measure. Therefore, BSDEs represent a more intuitive and understandable approach for option pricing and hedging.

Under the Black-Scholes framework, such a BSDE is linear and the solution is given in a closed form. However, in most practical scenarios, BSDEs cannot be explicitly solved. For instance, the Black-Scholes model under different interest rates for lending and borrowing~\cite{bergman1995option} leads to a nonlinear BSDE for which finding an analytical solution becomes challenging. Hence, advanced numerical techniques to approximate their solutions become desired. In recent years, various numerical methods have been proposed for solving BSDEs, e.g., \cite{bouchard2004discrete,zhang2004numerical,gobet2005regression,lemor2006rate,zhao2006new,bender2008,ma2008numerical,zhao2010stable,gobet2010solving,crisan2012solving,zhao2014new,ruijter2015fourier,ruijter2016fourier,teng2020multi,teng2021high} and many others. However, most of them are not suitable for tackling high-dimensional BSDEs due to the well-recognized challenge known as the ``curse of dimensionality". The computational cost associated with solving high-dimensional BSDEs grows exponentially with the increase in dimensionality. Some of the most important equations are naturally formulated in high dimensions. For instance, the Black-Scholes equation for option pricing exhibits the dimensionality of the BSDE with the number of underlying financial assets under consideration. Some techniques such as parallel computing using GPU computing~\cite{gobet2016stratified, kapllani2022multistep} or sparse grid methods~\cite{zhang2013sparse,fu2017efficient,chassagneux2023learning} have proven effective in solving only moderately dimensional BSDEs within reasonable computation time.

In recent years, machine learning models have demonstrated remarkable success in the field of artificial intelligence, inspiring applications in other domains where the curse of dimensionality has been a persistent challenge. Consequently, different approaches using machine learning have been proposed to solve high-dimensional BSDEs: the deep learning-based methods using deep neural networks (DNNs) and the regression tree-based methods~\cite{teng2021review,teng2022gradient}. The first deep learning-based scheme called the deep BSDE (we refer to it as the DBSDE scheme), was introduced in~\cite{weinan2017deep,han2018solving}. The authors conducted numerical experiments with various examples, demonstrating the effectiveness of their proposed algorithm in high-dimensional settings. It proved proficient in delivering both accurate approximations of the solution and computational efficiency. Therefore, the method opened the door to solving BSDEs in hundreds of dimensions in a reasonable amount of time. Several articles have been published after the original publication of the DBSDE method, some adjusting, reformulating, or extending the algorithm~\cite{wang2018deep,fujii2019asymptotic,ji2020three,hure2020deep,kremsner2020deep,pereira2019learning,beck2021deep,chen2021deep,liang2021deep,ji2021novel,pham2021neural,takahashi2022new,germain2022approximation,gnoatto2022deep,ji2022deep,kapllani2022effect,abbas2022pathwise,andersson2023convergence,gnoatto2023deep,raissi2024forward,kapllani2024deep,negyesi2024one}, while others focused on error analysis~\cite{han2020convergence,jiang2021convergence,negyesi2024generalized} and uncertainty quantification~\cite{kapllaniuncertainty}. It has been pointed out in the literature that the DBSDE method suffers from different issues such as convergence to an approximation far from the solution or even divergence when the problem has a complex structure and a long terminal time. To tackle these drawbacks, many alternative methods have been proposed, we refer to, e.g.,~\cite{hure2020deep,chassagneux2023learning,teng2022gradient,andersson2023convergence,kapllani2024deep}. High-accurate gradient approximations are of great significance, especially in financial applications, where the process $Z$ represents the hedging strategy for an option contract. Except the works in~\cite{kapllani2024deep,negyesi2024one}, other deep learning-based schemes does not discuss in detail the approximations for $Z$ in high-dimensional spaces, as it is generally more challenging than approximating $Y$ for BSDEs. In this work, we develop a novel algorithm that ensures high accuracy not only for the process $Y$ but also for the process $Z$.

The authors in~\cite{hure2020deep} approximate the unknown solution pair of~\eqref{eq1} using DNNs. The network parameters are optimized at each time step through the minimization of loss functions defined recursively via backward induction. More precisely, the loss is formulated from the Euler-Maruyama discretization of the BSDE at each time interval. Such formulation gives an implicit approximation of the process $Z$. Hence, the stochastic gradient descent (SGD) algorithm lacks explicit information about $Z$, which impacts its approximation accuracy. To address this, we enhance the SGD algorithm by providing it with additional information to achieve accurate approximations of $Z$. We make use of differential deep learning~\cite{huge2020differential}, a general extension of supervised deep learning. In this framework, the DNN model is trained not only on inputs and labels but also on differentials of labels with respect to (w.r.t.) inputs. Differential deep learning offers an efficient approximation not only of the labels but also of their derivatives when compared to traditional supervised deep learning. We use Malliavin calculus to formulate the BSDE problem as a differential deep learning problem. By applying the Malliavin derivative to a BSDE, the Malliavin derivatives of the solution pair $(Y, Z)$ of the BSDE satisfy themselves another BSDE, resulting thus in a system of BSDEs. This formulation also requires estimating the Hessian matrix of the solution. In the context of option pricing, this matrix corresponds to $\Gamma$ sensitivity, which can be used to indicate a potential acceleration in changes in the option's value. 

Our method works as follows. Firstly, we discretize the system of BSDEs using the Euler-Maruyama method. Subsequently, we utilize DNNs to approximate the unknown solution of these BSDEs, requiring the estimation of the triple of the processes $\left(Y, Z, \Gamma\right)$. The network parameters are optimized backwardly at each time step by minimizing a loss function defined as a weighted sum of the dynamics of the discretized BSDE system. Through this way, SGD is equipped with explicit information about the dynamics of the process $Z$. As a result, our method can yield more accurate approximations than the scheme proposed in~\cite{hure2020deep} not only for the process $Z$, but also for the process $\Gamma$. Moreover, the computation time of our scheme is shorter in comparison to~\cite{hure2020deep} when including the computation of $\Gamma$ at each optimization step. This is because the latter requires the use of automatic differentiation (AD) to estimate the process $\Gamma$, whereas in our method, we approximate it using a DNN, which is more efficient in terms of time consumption, as we demonstrate in our numerical experiments. Note that the authors in~\cite{negyesi2024one} also used the Malliavin derivative to improve the accuracy of $Z$. However, their method significantly differs from ours, as they only employ supervised deep learning. Their approach requires training the BSDE system separately, which can be expected to have a higher computational cost compared to our method. Furthermore, our approach using differential deep learning can be straightforwardly extended not only to~\cite{hure2020deep}, which operates backward in time through local optimization at each discrete time step, but also to other deep learning-based schemes~\cite{weinan2017deep,raissi2024forward,kapllani2024deep} formulated forward in time as a global optimization problem (this is part of our ongoing research). In contrast, the scheme presented in~\cite{negyesi2024one} cannot be integrated into such methodologies, as it cannot be formulated as a global optimization problem. To the best of our knowledge, only~\cite{lefebvre2023differential} applies differential deep learning to solve high-dimensional PDEs, where the authors consider the associated dual stochastic control problem instead of working with BSDEs.

The outline of the paper is organized as follows. In the next section, we recall some of the well-known results concerning BSDEs. In Section~\ref{sec3}, DNNs and differential deep learning techniques are described. Our backward differential deep learning-based algorithm is presented in Section~\ref{sec4}. Section~\ref{sec5} is devoted to the convergence analysis of our algorithm. The numerical experiments presented in Section~\ref{sec6} confirm the theoretical results and show high accuracy of the solution, its gradient, and the Hessian matrix of the solution over different option pricing problems. Finally, Section~\ref{sec7} concludes this work.

\section{Preliminaries}
\label{sec2}
\subsection{Spaces and notation}
\label{subsec21}
Let $\left(\Omega,\mathcal{F},\mathbb{P},\{\mathcal{F}_t\}_{0\le t \le T}\right)$ be a complete, filtered probability space. In this space a standard $d$-dimensional Brownian motion $\{W_t\}_{\leq t \leq T}$ is defined, such that the filtration $\{\mathcal{F}_t\}_{0\le t\le T}$ is the natural filtration of $W_t.$ As usual, we identify random variables that are equal $\mathbb{P}$-a.s. and accordingly, understand equalities and inequalities between them in the $\mathbb{P}$-a.s. sense. For the expectation, we omit the superscript $\mathbb{P}$ if it is meant under probability measure $\mathbb{P}$ (unless stated otherwise). We denote further
\begin{itemize}
    \item $x \in \mathbb{R}^d$ as a column vector. $x \in \mathbb{R}^{1\times d}$ as a row vector. 
    \item $| x |$ for the Frobenius norm of any $x \in \mathbb{R}^{d \times \qfrac}$. In the case of scalar and vector inputs, these coincide with the standard Euclidian norm. 
    \item $\mathbb{S}^2\left([0, T] \times \Omega; \mathbb{R}^{d\times \qfrac} \right)$ for the space of continuous and progressively measurable stochastic processes $X: [0, T] \times \Omega \to \mathbb{R}^{d\times \qfrac} $ such that $\mathbb{E}\bigl[\sup_{ 0 \leq t\leq T}\left|X_t\right|^2\bigr] < \infty$.
    \item $\mathbb{H}^2\left([0, T] \times \Omega; \mathbb{R}^{d\times \qfrac} \right)$ for the space of progressively measurable stochastic processes $Z: [0, T] \times \Omega \to \mathbb{R}^{d\times \qfrac} $ such that $\mathbb{E}\left[ \int_0^T \left|Z_t\right|^2 \, dt \right] < \infty$.
    \item $\mathbb{L}^2_{\mathcal{F}_t}\left(\Omega; \mathbb{R}^{d \times \qfrac} \right)$ for the space of $\mathcal{F}_t$-measurable random variable $\xi: \Omega \to \mathbb{R}^{d \times \qfrac} $ such that $\mathbb{E}\bigl[\left|\xi\right|^2 \bigr] < \infty$.
    \item $H^2\left( [0, T] ; \mathbb{R}^{\qfrac}\right)$ for the Hilbert space of deterministic functions $h: [0, T] \to \mathbb{R}^{\qfrac}$ such that $\int_{0}^T \left| h\left(t\right) \right|^2 dt < \infty$.
    \item $\nabla_x f := \left( \frac{\partial f}{\partial x_1}, \ldots, \frac{\partial f}{\partial x_d} \right) \in \mathbb{R}^{1 \times d}$ for the gradient of scalar-valued multivariate function $f\left(t, x, y, z\right)$ w.r.t. $x \in \mathbb{R}^{d}$, and analogously for $\nabla_y f \in \mathbb{R}$ and $\nabla_z f \in \mathbb{R}^{1 \times d}$ w.r.t. $y \in \mathbb{R}$ and $z \in \mathbb{R}^{1 \times d}$, respectively. Similarly, we denote the Jacobian matrix of a vector-valued function $u: \mathbb{R}^d \to \mathbb{R}^{\qfrac}$ by $\nabla_x u \in \mathbb{R}^{\qfrac \times d}$.
    \item $\Hessian_x u \in \mathbb{R}^{d \times d}$ the Hessian matrix of a function $u: \mathbb{R}^d \to \mathbb{R}$.
    \item $C^{\lfrac}_{\bfrac}\left( \mathbb{R}^{d}; \mathbb{R}^{\qfrac} \right)$ and $C^{\lfrac}_{\pfrac}\left( \mathbb{R}^{d}; \mathbb{R}^{\qfrac} \right)$ for the set of $\lfrac$-times continuously differentiable functions $\varphi: \mathbb{R}^d \to \mathbb{R}^{\qfrac}$ such that all partial derivatives up to order $\lfrac$ are bounded or have polynomial growth, respectively.
    \item $\Delta = \{t_0, t_1, \ldots, t_N\}$ is the time discretization of $[0, T]$ with $t_0 = 0 < t_1 < \ldots < t_N = T$, $\Delta t_n = t_{n+1} - t_n$ and $| \Delta | := \max_{ 0\leq n \leq N-1 } t_{n+1} - t_n$.
    \item $\mathbb{E}_n\left[ Y \right]:=\mathbb{E}\left[ Y | \mathcal{F}_{t_n} \right]$ for the conditional expectation w.r.t. the natural filtration, given the time partition $\Delta$.
    \item $x^{\top}\in \mathbb{R}^{d \times \qfrac}$ for the transpose of any $x \in \mathbb{R}^{d \times \qfrac}$.
    \item $\Tr\left[x\right]$ for the trace of any $x \in \mathbb{R}^{d \times d}$.
    \item $\mathbf{0}_{d,d}$, $\mathbf{1}_{d,d}$ for $\mathbb{R}^{d \times d}$ matrices of all zeros and ones, respectively.
\end{itemize}
\subsection{Malliavin calculus}
\label{subsec22}
We shall use techniques of the stochastic calculus of variations. To this end, we use the following notation. For more details, we refer the reader to~\cite{nualart2006malliavin}. Let $\mathcal{S}$ be the space of smooth random variables of the form 
$$\xi = \varphi \left( \int_0^T h_1(t) dW_t, \ldots, \int_0^T h_d(t) dW_t\right)$$
where $\varphi \in C^{\infty}_{\pfrac}\left(\mathbb{R}^d;\mathbb{R}\right)$, $h_1, \ldots, h_d \in H^2\left( [0, T] ; \mathbb{R}^{\qfrac}\right)$. The Malliavin derivative of smooth random variable $\xi \in \mathcal{S}$ is the $\mathbb{R}^{1 \times \qfrac}$-valued stochastic process given by
$$
D_s \xi : = \sum_{k=1}^d \frac{\partial \varphi}{x_k}\left(  \int_0^T h_1(t) dW_t, \ldots, \int_0^T h_d(t) dW_t \right) h_k(s).
$$
We define the domain of $D$ in $\mathbb{L}^2_{\mathcal{F}_T}$ as $\mathbb{D}^{1,2}\left( \Omega;\mathbb{R} \right)$, meaning that $\mathbb{D}^{1,2}$ is the closure of the class of smooth random variables $\mathcal{S}$  w.r.t. the norm
$$
\| \xi  \|_{\mathbb{D}^{1,2}}:= \left( \mathbb{E}\left[  \left| \xi \right|^2 + \int_{0}^T \left| D_s \xi \right|^2 ds \right] \right)^{\frac{1}{2}}.
$$
Note that in case of vector valued Malliavin differentiable random variables $\xi = \left( \xi_1, \ldots, \xi_{\qfrac} \right)$, $\xi \in \mathbb{D}^{1,2}\left(\Omega; \mathbb{R}^{\qfrac} \right)$, its Malliavin derivative $D_s \xi \in \mathbb{R}^{\qfrac \times \qfrac}$ is the matrix-valued stochastic process. 

The following lemma represents the Malliavin chain rule, which can be extended to Lipschitz continuous functions.
\begin{lemma} 
(Malliavin chain rule~\cite{nualart2006malliavin})\\
Let $F \in C^{1}_{\bfrac}\left( \mathbb{R}^d; \mathbb{R}^{\qfrac} \right)$. Suppose that $\xi \in \mathbb{D}^{1, 2}\left( \Omega; \mathbb{R}^d \right)$. Then $F\left(\xi\right) \in \mathbb{D}^{1,2}\left( \Omega; \mathbb{R}^{\qfrac} \right)$, and for each $0 \leq s \leq T$
\begin{equation*}    
    D_s F (\xi) = \nabla_x F (\xi) D_s \xi.
\end{equation*}
\label{lemma1}
\end{lemma}

\subsection{Some results on BSDEs}
\label{subsec23}
We recall some results on BSDE known from the literature that are relevant for this work. For the functions in BSDE~\eqref{eq1}, we hierarchically structure the properties that they are assumed to fulfill.
\begin{assump}{AX1}
    The initial condition $x_0 \in \mathbb{L}^2_{\mathcal{F}_0}\left(\Omega; \mathbb{R}^{d}\right)$ and $a, b$ satisfy a linear growth condition in $x$, i.e.,
    $$ \left| a\left(t, x\right) \right| + \left| b\left(t, x\right) \right| \leq C \left( 1+\left| x\right|\right),$$
    $\forall \,  t \in [0, T], x \in \mathbb{R}^{d}$ and some constant $C>0$. Furthermore, $a, b$ are uniformly Lipschitz continuous in the spatial variable, i.e.,
    \begin{equation*}
        \left| a\left(t, x_1\right)  - a\left(t, x_2\right) \right| + \left| b\left(t, x_1\right)  - b\left(t, x_2\right) \right| \leq L_{a,b} \left| x_1 - x_2 \right|
    \end{equation*}
    $\forall \,  t \in [0, T], x_1, x_2 \in \mathbb{R}^{d}$, for some constant $L_{a,b}>0$.
    \label{AX1}
\end{assump}
\begin{assump}{AX2}
    Assumption~\ref{AX1} holds. Moreover, $a(t, 0)$, $b(t, 0)$ are uniformly bounded $\forall$ $0 \leq t \leq T$ and $a \in C_{\bfrac}^{0, 1}\left( [0, T] \times \mathbb{R}^{d}; \mathbb{R}^{d} \right)$, $b \in C_{\bfrac}^{0, 1}\left( [0, T] \times \mathbb{R}^{d}; \mathbb{R}^{d \times d} \right)$.
    \label{AX2}
\end{assump}
\begin{assump}{AX3}
    Assumption~\ref{AX2} holds. Moreover, $a \in C_{\bfrac}^{0, 2}\left( [0, T] \times \mathbb{R}^{d}; \mathbb{R}^{d} \right)$, $b \in C_{\bfrac}^{0, 2}\left( [0, T] \times \mathbb{R}^{d}; \mathbb{R}^{d \times d} \right)$ and there exist and positive constant $C>0$ such that
     $$v^{\top} b(t, x) b(t, x)^{\top} v \geq C |v|^2, \quad x, v \in \mathbb{R}^d, t \in [0, T].$$    
    \label{AX3}
\end{assump}
\begin{assump}{AY1}
    The function $f(t,x,y,z)$ is uniformly Lipschitz continuous w.r.t $y$ and $z$, i.e.
    \begin{equation*}    
        \left| f\left(t, x, y_1, z_1\right) - f\left(t, x, y_2, z_2\right)\right| \leq L_{f}\left( \left| y_1 - y_2 \right| + \left| z_1 - z_2 \right| \right),
    \end{equation*}
    $\forall \,  (t, x, y_1, z_1)$ and $(t, x, y_2, z_2) \in [0, T] \times \mathbb{R}^d \times \mathbb{R} \times \mathbb{R}^{1\times d}$, for some constant $L_{f}>0$. Moreover, $f, g$ satisfy a polynomial growth condition in $x$,i.e.,
    \begin{equation*}    
        \left| f\left(t, x, y, z\right) \right| + \left| g\left(x\right) \right| \leq C\left( 1 + \left|x\right|^2 \right),
    \end{equation*}
    $\forall \,  (t, x, y, z) \in [0, T] \times \mathbb{R}^d \times \mathbb{R} \times \mathbb{R}^{1\times d}$ for some constant $C>0$.
    \label{AY1}
\end{assump}
\begin{assump}{AY2}
    Assumption~\ref{AY1} holds. Moreover, $f \in C_{\bfrac}^{0,1,1,1}\left( [0, T] \times \mathbb{R}^{d} \times \mathbb{R} \times \mathbb{R}^{1 \times d}; \mathbb{R}\right)$ and $g \in C_{\bfrac}^1\left( \mathbb{R}^{d}; \mathbb{R} \right)$.
    \label{AY2}
\end{assump}
\begin{assump}{AY3}
    Assumption~\ref{AY2} holds. Moreover, $f \in C_{\bfrac}^{0,2,2,2}\left( [0, T] \times \mathbb{R}^{d} \times \mathbb{R} \times \mathbb{R}^{1 \times d}; \mathbb{R}\right)$ and $g \in C_{\bfrac}^2\left( \mathbb{R}^{d}; \mathbb{R} \right)$.
    \label{AY3}
\end{assump}

In the following theorem, we state the well-known result on SDEs.
\begin{theorem}(Moment Estimates for SDEs~\cite{kloeden2013numerical})\\
    Assume that Assumption~\ref{AX1} holds. Then the SDE in~\eqref{eq1} has a unique strong solution $\{ X_t \}_{0\leq t \leq T} \in \mathbb{S}^2\left([0, T] \times \Omega; \mathbb{R}^{d} \right)$ and the following moment estimates hold:
    \begin{equation*}
        \mathbb{E}\left[ \sup_{0 \leq t \leq T} \left| X_t \right|^2 \right] \leq C, \quad \mathbb{E}\left[ \sup_{s \leq r \leq t}\left| X_r - X_s \right|^2 \right] \leq C \left| t - s \right|,
    \end{equation*}
    where constant $C$ depends only on $T, d$. 
    \label{theorem1}
\end{theorem}
The well-posedness of the BSDE~\eqref{eq1} is guaranted by Assumption~\ref{AY1}. The following theorem guarantes the existence of a unique solution triple of~\eqref{eq1}.
\begin{theorem}(Properties of BSDEs~\cite{El1997})\\
    Assume that Assumptions~\ref{AX1} and~\ref{AY1} holds. Then the BSDE~\eqref{eq1} admits a unique solution triple $\{X_t, Y_t, Z_t \}_{0\leq t \leq T} \in \mathbb{S}^2\left([0, T]\times \Omega; \mathbb{R}^{d} \right) \times \mathbb{S}^2\left( [0, T]\times \Omega;\mathbb{R} \right) \times \mathbb{H}^2\left([0, T]\times \Omega; \mathbb{R}^{1 \times d} \right)$.
    \label{theorem2}
\end{theorem}
An important property of BSDEs is that they provide a probabilistic representation for the solution of a specific class of PDEs given by the nonlinear Feynman–Kac formula. Consider the semi-linear parabolic PDE
\begin{equation}
    \begin{aligned}        
    \frac{\partial u(t,x)}{\partial t} + \nabla_x u(t, x)\,a(t, x) + \frac{1}{2} \Tr\left[b b^{\top} \Hessian_x u (t, x)\right] + f\left(t, x, u, \nabla_x u \,  b\right)(t, x) = 0,
    \end{aligned}
    \label{eq2}
\end{equation}
for all $(t, x) \in ([0, T]\times \mathbb{R}^d)$ and the terminal condition $u(T,x)=g(x)$. Assume that~\eqref{eq2} has a classical solution $u \in C^{1,2}_{\bfrac}\left([0, T] \times \mathbb{R}^{d}; \mathbb{R}\right)$ and the aforementioned standard Lipschitz assumptions of~\eqref{eq1} are satisfied. 
Then the solution of~\eqref{eq1} can be represented $\mathbb{P}$-a.s. by 
\begin{equation}
	Y_t = u\left(t, X_t\right), \quad Z_t= \nabla_x u\left(t, X_t\right)b\left(t, X_t\right) \quad \forall \, t \in \left[0,T\right).
\label{eq3}
\end{equation}

Next, we collect some Malliavin differentiability results on BSDEs, as we are interested on BSDEs such that their solution triple $\{X_t, Y_t, Z_t \}_{0\leq t \leq T}$ is differentiable
in the Malliavin sense. The results are stated in the following theorems.
\begin{theorem}
(Malliavin differentiability of SDEs~\cite{nualart2006malliavin})\\
Assume that Assumption~\ref{AX2} holds. Then $\forall \, t \in [0, T]$, $X_t \in \mathbb{D}^{1,2}\left( \Omega;\mathbb{R}^{d} \right)$ and its Malliavin derivative admits a continuous version $\{ D_s X_t \}_{ 0\leq s, t \leq T} \in \mathbb{S}^2\left([0, T] \times \Omega; 
 \mathbb{R}^{d \times d}\right)$ satisfying for $ 0\leq s \leq t \leq T$ the SDE
 \begin{equation*}
    D_s X_t = \mathds{1}_{s \leq t}\Biggl\{ b\left(s, X_s\right)  + \int_s^t \nabla_x a\left( r, X_r \right) D_s X_r dr +  \int_s^t \nabla_x b\left( r, X_r \right) D_s X_r dW_r\Biggr\},
 \end{equation*}
 where $\nabla_x b$ denotes a $\mathbb{R}^{d \times d \times d}$-valued tensor. Moreover, there exists a constant $C > 0$ such that
 \begin{equation*}
    \sup_{s \in [0, T]}\mathbb{E}\left[ \sup_{t \in [s, T]} \left| D_s X_t \right|^2 \right] \leq C, \quad \mathbb{E}\left[ \left| D_s X_r - D_s X_t \right|^2 \right] \leq C \left| r - t \right|.
\end{equation*}
\label{theorem3}
\end{theorem}

\begin{theorem}
(Malliavin differentiability of BSDEs~\cite{El1997})\\
Assume that Assumptions~\ref{AX2} and~\ref{AY2} holds. Then  the solution triple $\{X_t, Y_t, Z_t \}_{0\leq t \leq T}$ of~\eqref{eq1} verify that $\forall \, t \in [0, T]$ $Y_t \in \mathbb{D}^{1,2}\left( \Omega;\mathbb{R} \right)$, $Z_t \in \mathbb{D}^{1,2}\left( \Omega;\mathbb{R}^{1 \times d} \right)$, $X$ satisfies the statement of Theorem~\ref{theorem3}, and a version of $\{ D_s Y_t \}_{ 0\leq s, t \leq T} \in \mathbb{S}^2\left( [0, T] \times \Omega;
 \mathbb{R}^{1 \times d}\right)$, $\{ D_s Z_t \}_{ 0\leq s, t \leq T} \in \mathbb{H}^2\left( [0, T] \times \Omega; 
 \mathbb{R}^{d \times d}\right)$ satisfy the BSDE
 \begin{equation*}
     \begin{split}     
         D_s Y_t & = \mathbf{0}_d, \quad D_s Z_t = \mathbf{0}_{d,d}, \quad 0 \leq t < s \leq T,\\
        D_s Y_t & = \nabla_x g\left(X_T\right) D_s X_T+ \int_t^T \bigl( \nabla_x f\left( r, X_r, Y_r, Z_r \right) D_s X_r + \nabla_y f\left( r, X_r, Y_r, Z_r \right) D_s X_r \bigr. \\ 
         & \quad  + \bigl. \nabla_z f\left( r, X_r, Y_r, Z_r \right) D_s X_r \bigr)dr -  \int_t^T D_s Z_r dW_r, \quad 0 \leq s \leq t \leq T.\\
    \end{split}
\end{equation*}
Moreover, $D_t Y_t$ defined by the above equation is a version of $Z_t$ $\mathbb{P}$-a.s. $\forall \, t \in [0, T]$.
\label{theorem4}
\end{theorem}
The final important result that is relevant for this work is the path regularity result of the processes $Y$ and $Z$, which we state in the following theorem.
\begin{theorem}
(Path regularity~\cite{imkeller2010path})\\
Under Assumptions~\ref{AX2} and~\ref{AY2} the BSDE~\eqref{eq1} admits a unique solution triple $\{X_t, Y_t, Z_t \}_{0\leq t \leq T} \in \mathbb{S}^2\left([0, T]\times \Omega; \mathbb{R}^{d} \right) \times \mathbb{S}^2\left( [0, T]\times \Omega;\mathbb{R} \right) \times \mathbb{H}^2\left([0, T]\times \Omega; \mathbb{R}^{1 \times d} \right)$. Moreover, the following holds true:
\begin{itemize}
    \item[(i)] There exist a constant $C>0$ such that $\forall \, 0\leq s\leq t\leq T$
    \begin{equation*}
        \mathbb{E}\left[ \sup_{s \leq r \leq t}\left| Y_r - Y_s \right|^2 \right] \leq C \left| t - s \right|
    \end{equation*}
    \item[(ii)] There exist a constant $C>0$ such that for any partition $\Delta$ of $[0, T]$
    \begin{equation*}
        \sum_{n=1}^{N-1}\mathbb{E}\left[ \int_{t_n}^{t_{n+1}} \left| Z_t - Z_{t_n} \right|^2 dt  \right] \leq C \left| \Delta \right|.    
    \end{equation*}
    \item[(iii)] Under Assumptions~\ref{AX3} and~\ref{AY3} we further have that there exist a constant $C>0$ such that $\forall \, 0\leq s\leq t\leq T$
    \begin{equation*}
        \mathbb{E}\left[ \sup_{s \leq r \leq t}\left| Z_r - Z_s \right|^2 \right] \leq C \left| t - s \right|
    \end{equation*}
    In particular, there exists a continuous modification of the process $Z$. 
\end{itemize}
\label{theorem5}
\end{theorem}

\section{Differential deep learning}
\label{sec3}
In this section, we discuss differential machine learning in the context of DNNs, specifically differential deep learning, which plays a crucial role in formulating our algorithm.  We start by describing DNNs, that are designed to approximate unknown or a large class of functions.
\subsection{Deep neural networks}
\label{subsec31}
Let $d_0, d_1\in \mathbb{N}$ be the input and output dimensions, respectively. We fix the global number of layers as $L+2$, $L \in \mathbb{N}$ the number of hidden layers each with $\eta \in \mathbb{N}$ neurons. The first layer is the input layer with $d_0$ neurons and the last layer is the output layer with $d_1$ neurons. A DNN is a function $\phi(\cdot; \theta): \mathbb{R}^{d_0} \to \mathbb{R}^{d_1}$ composed of a sequence of simple functions, which therefore can be collected in the following form
\begin{equation*}
    x \in \mathbb{R}^{d_0} \longmapsto A_{L+1}(\cdot;\theta(L+1)) \circ \varrho \circ A_{L}(\cdot;\theta(L)) \circ \varrho \circ \ldots \circ \varrho \circ A_1(x;\theta(1)) \in \mathbb{R}^{d_1},
\end{equation*}
where $\theta:=\left( \theta(1), \ldots, \theta(L+1) \right) \in \mathbb{R}^{P}$ and $P$ is the total number of network parameters, $x \in \mathbb{R}^{d_0}$ is called an input vector. Moreover, $A_l(\cdot; \theta(l)), l = 1, 2, \ldots, L+1$ are affine transformations: $A_1(\cdot;\theta(1)): \mathbb{R}^{d_0} \to \mathbb{R}^{\eta}$, $A_l(\cdot;\theta(l)),  l = 2, \ldots, L: \mathbb{R}^{\eta} \to \mathbb{R}^{\eta}$ and $A_{L+1}(\cdot;\theta(L+1)): \mathbb{R}^{\eta} \to \mathbb{R}^{d_1}$, represented by
\begin{equation*}
    A_l(v;\theta(l)) = \mathcal{W}_l v + \mathcal{B}_l, \quad v \in \mathbb{R}^{\eta_{l-1}},
\end{equation*}
where $\theta(l):=\left(\mathcal{W}_l, \mathcal{B}_l\right)$, $\mathcal{W}_l \in \mathbb{R}^{\eta_{l} \times \eta_{l-1}}$ is the weight matrix and $\mathcal{B}_l \in \mathbb{R}^{\eta_{l}}$ is the bias vector with $\eta_0 = d_0, \eta_{L+1} = d_1, \eta_l = \eta$ for $l = 1, \ldots, L$ and $\varrho: \mathbb{R} \to \mathbb{R}$ is a nonlinear function (called the activation function), and applied component-wise on the outputs of $A_l(\cdot;\theta(l))$. Common choices are $\tanh(\cdot), \sin(\cdot), \max(0,\cdot)$ etc. All these matrices $\mathcal{W}_l$ and vectors $\mathcal{B}_l$ form the parameters $\theta$ of the DNN and can be collected as 
$$P = \sum_{l=1}^{L+1}\eta_{l}(\eta_{l-1}+1) = \eta(d_0+1) + \eta(\eta+1)(L-1) + d_1(\eta+1),$$
for fixed $d_0, d_1, L$ and $\eta$. We denote by $\Theta$ the set of possible parameters for the DNN $\phi(\cdot; \theta)$ with $\theta \in \Theta$. The Universal Approximation Theorem (UAT)~\cite{hornik1989multilayer,cybenko1989approximation} justifies the use of DNNs as function approximators.

\subsection{Training of deep neural networks using supervised deep learning}
\label{subsec32}
Once the DNN architecture is defined, what determines the mapping of a certain input to an output are the parameters $\theta$ incorporated in the DNN model. These parameters need to be optimized such that the DNN approximates the unknown function which is called the training of the DNN. The loss function acts as the objective function to be minimized during the training procedure, in which the DNNs optimal set of parameters is searched. 

Consider the training data sampled from some (unknown) multivariate joint distribution $(\mathcal{X}, \mathcal{Y}) \sim \mathcal{P}$, where the random variable $\mathcal{X} \in \mathbb{R}^d$ is referred as the input and the random variable $\mathcal{Y} \in \mathbb{R}$ as the label. The goal (in a regression setting) is then to approximate the deterministic function $F(x) := \mathbb{E}^{\mathcal{P}} \left[ \mathcal{Y} |\mathcal{X} = x \right]$ by DNN $\phi(\mathcal{X};\theta)$ using $(\mathcal{X}, \mathcal{Y}) \sim \mathcal{P}$. The loss function measures how well the current approximation of the DNN is compared to the label. A common choice is the expected squared error, which is given as  
\begin{equation}
    \mathbf{L}( \theta ):= \mathbb{E}^{\mathcal{P}}\left[ \left| \phi\left( \mathcal{X}; \theta \right) - \mathcal{Y} \right|^2\right].
    \label{eq4}
\end{equation}
Then the optimal parameters $\theta^{*}$ in~\eqref{eq4} are given as 
\begin{equation*}
    \theta^{*} \in \argmin_{\theta \in \Theta} \mathbf{L}(\theta),
\end{equation*}
which can be estimated by using SGD-type algorithms. 

\subsection{Training of deep neural networks using differential deep learning}
\label{subsec33}
One of the biggest challenges w.r.t. finding the optimal parameter set of the DNN is to avoid learning training data-specific patterns, namely overfitting, and rather enforcing better generalization of the fitted models. Hence, regularization approaches have been developed for DNNs to avoid overfitting and thus improve the performance of the model. Such approaches penalize certain norms of the parameters $\theta$, expressing a {\it preference} for $\theta$. Differential deep learning~\cite{huge2020differential} has the same motivation as regularization, namely to improve the accuracy of the model. This is achieved by not expressing a {\it preference}, but {\it correctness}, in particular enforcing differential {\it correctness}. It assumes that the derivative of the label w.r.t. input is known. Let us consider the function $F_x(x)=\nabla_{x} F(x)$ and the random variable $\mathcal{Z}:= F_x(\mathcal{X}) \in \mathbb{R}^{1\times d}$. The goal in differential deep learning is to approximate the label function $F(x)$ by DNN $\phi(\mathcal{X};\theta)$ using data $( \mathcal{X}, \mathcal{Y}, \mathcal{Z}) \sim \mathcal{P}$ and minimizing the extended loss function~\eqref{eq4} given as
\begin{equation}
        \mathbf{L}( \theta )= \mathbb{E}^{\mathcal{P}}\left[ \left| \phi\left( \mathcal{X}; \theta \right) - \mathcal{Y} \right|^2\right] + \lambda \mathbb{E}^{\mathcal{P}}\left[ \left| \nabla_{x}\phi\left( \mathcal{X}; \theta \right) - \mathcal{Z} \right|^2\right],
\label{eq5}
\end{equation}
where $\nabla_{x}\phi$ is calculated using AD and $\lambda \in \mathbb{R}_{+}$. For our purpose, we consider a slightly different formulation of differential deep learning compared to~\cite{huge2020differential}. We use one DNN $\phi^y\left( \mathcal{X} ; \theta^y\right)$ to approximate the function $F(x)$ and another $\phi^z\left( \mathcal{X} ; \theta^z\right)$ for $F_x(x)$, and rewrite the loss function~\eqref{eq5} as
\begin{equation}
    \mathbf{L}( \theta )= \omega_1 \mathbb{E}^{\mathcal{P}}\left[ \left| \phi^y\left( \mathcal{X}; \theta^y \right) - \mathcal{Y} \right|^2\right] + \omega_2 \mathbb{E}^{\mathcal{P}}\left[ \left| \phi^z\left( \mathcal{X}; \theta^z \right) - \mathcal{Z} \right|^2\right],
\label{eq6}
\end{equation}
where $\theta = \left( \theta^y, \theta^z \right)$, $\omega_1, \omega_2 \in [0, 1]$ and $\omega_1 +\omega_2 = 1$. Then the optimal parameters $\theta^{*}$ in~\eqref{eq6} are given as 
\begin{equation*}
    \theta^{*} \in \argmin_{\theta \in \Theta} \mathbf{L}(\theta),
\end{equation*}
estimated using a SGD method. 

\section{A backward differential deep learning-based scheme for BSDEs}
\label{sec4}
In this section, we introduce the proposed backward differential deep learning-based method. In order to formulate BSDE as a differential learning problem, we firstly discretize the integrals in the resulting BSDE system given as
\begin{align}
        X_t & = x_0 + \int_{0}^{t} a \left(s, X_s\right)\,ds + \int_{0}^{t} b \left(s, X_s\right)\,dW_s, \label{eq7}\\
      Y_t & = g\left(X_T\right) + \int_{t}^{T} f\left(s, \mathbf{X}_s\right)\,ds -\int_{t}^{T}Z_s\,dW_s, \label{eq8}\\
        D_s X_t & = \mathds{1}_{s \leq t}\Bigl[ b\left(s, X_s\right)  + \int_s^t \nabla_x a\left( r, X_r \right) D_s X_r dr +  \int_s^t \nabla_x b\left( r, X_r \right) D_s X_r dW_r\Bigr], \label{eq9}\\
    D_s Y_t & = \mathds{1}_{s \leq t}\Bigl[\nabla_x g\left(X_T\right) D_s X_T+ \int_t^T f_D\left( r, \mathbf{X}_r, \mathbf{D}_s \mathbf{X}_r\right) dr -  \int_t^T D_s Z_r dW_r \Bigl],  \label{eq10}
\end{align}
where we introduced the notations $\mathbf{X}_t := \left( X_t, Y_t, Z_t\right)$, $\mathbf{D}_s\mathbf{X}_t := \left( D_s X_t, D_s Y_t, D_s Z_t\right)$ and $f_D\left(t, \mathbf{X}_t, \mathbf{D}_s\mathbf{X}_t \right):= \nabla_x f\left( t, \mathbf{X}_t \right) D_s X_t + \nabla_y f\left( t, \mathbf{X}_t\right) D_s Y_t + f\left( t, \mathbf{X}_t\right) D_s Z_t$ $\forall \, 0 \leq s, t \leq T$. Note that the solution of the above BSDE system is a pair of triples of stochastic processes $\{\left(X_t, Y_t, Z_t\right)\}_{0\leq t \leq T}$ and $\{\left(D_s X_t, D_s Y_t, D_s Z_t\right)\}_{0\leq s, t \leq T}$ such that~\eqref{eq7}-\eqref{eq10} holds $\mathbb{P}$-a.s.

Let us consider the time discretization $\Delta$. For notational convenience we write $\Delta W_n = W_{t_{n+1}} - W_{t_n}$, $(X_n, Y_n, Z_n) = (X_{t_n}, Y_{t_n}, Z_{t_n})$, $(D_n X_m, D_n Y_m, D_n Z_m) = (D_{t_n} X_{t_m}, D_{t_n} Y_{t_m}, D_{t_n} Z_{t_m})$, and $\left(X^{\Delta}_n, Y^{\Delta}_n, Z^{\Delta}_n\right)$, $\left(D_nX^{\Delta}_m, D_nY^{\Delta}_m, D_n Z^{\Delta}_m\right)$ for the approximations, where $0 \leq n, m \leq N$. The forward SDE~\eqref{eq7} is approximated by the Euler-Maruyama scheme, i.e.,
\begin{equation}
     X^{\Delta}_{n+1} = X^{\Delta}_n + a\left(t_n, X^{\Delta}_n\right) \Delta t_n + b\left(t_n, X^{\Delta}_n\right) \Delta W_n,
     \label{eq11}
\end{equation}
for $n = 0, 1, \ldots, N-1$, where $X^{\Delta}_0 = x_0$.

Next we apply the Euler-Maruyama scheme to~\eqref{eq8}. For the time interval $\left[t_n, t_{n+1}\right]$ we have
\begin{equation}
     Y_n = Y_{n+1} + \int_{t_{n}}^{t_{n+1}} f\left(s, \mathbf{X}_s\right)\,ds -\int_{t_{n}}^{t_{n+1}} Z_s\,dW_s.
     \label{eq12}
\end{equation}
Applying the Euler-Maruyama scheme in~\eqref{eq12} one obtains
\begin{equation}
     Y^{\Delta}_n = Y^{\Delta}_{n+1} + f\left(t_n, \mathbf{X}_n^{\Delta}\right) \Delta t_n -  Z^{\Delta}_n \Delta W_n,
     \label{eq13}
\end{equation}
for $n = N-1, N-2, \ldots, 0,$ where $\mathbf{X}_n^{\Delta} := \left( X_n^{\Delta}, Y_n^{\Delta}, Z_n^{\Delta}\right)$ and $Y^{\Delta}_N = g\left(X^{\Delta}_N\right)$. 

Next, we discretize the BSDE for the Malliavin derivatives, i.e.~\eqref{eq9}-\eqref{eq10} in a similar manner. The Malliavin derivative~\eqref{eq9} approximated by the Euler-Maruyama method gives the estimates
\begin{equation}
    D_n X_m^{\Delta} = 
    \begin{cases}
        \mathds{1}_{n = m}b\left( t_n, X_n^{\Delta}\right), \quad  0 \leq m \leq n \leq N,\\
   	D_n X^{\Delta}_{m-1} + \nabla_x a\left(t_{m-1}, X^{\Delta}_{m-1}\right) D_n X_{m-1}^{\Delta} \Delta t_{m-1}\\
    + \nabla_xb\left(t_{m-1}, X^{\Delta}_{m-1}\right) D_n X_{m-1}^{\Delta} \Delta W_{m-1},  \quad  0 \leq n < m \leq N.
    \end{cases}
     \label{eq14}
\end{equation}
From $\left[t_n, t_{n+1}\right]$~\eqref{eq10} is given as
\begin{equation}
    \begin{split}
    D_n Y_n & = D_n Y_{n+1} + \int_{t_{n}}^{t_{n+1}} f_D\left(s, \mathbf{X}_s, \mathbf{D}_n \mathbf{X}_s\right)\,ds - \int_{t_{n}}^{t_{n+1}} D_n Z_s\,dW_s.
    \end{split}
    \label{eq15}
\end{equation}
Using Euler-Maruyama scheme in~\eqref{eq15}, we get
\begin{equation}
    \begin{split}
    D_n Y_n^{\Delta} & = D_n Y_{n+1}^{\Delta} + f_D\left(t_n, \mathbf{X}_n^{\Delta}, \mathbf{D}_n \mathbf{X}_n^{\Delta}\right)\,\Delta t_n - D_n Z_n^{\Delta} \Delta W_n,
    \end{split}
    \label{eq16}
\end{equation}
with $\mathbf{D}_n\mathbf{X}_n^{\Delta} := \left( D_n X_n^{\Delta}, D_n Y_n^{\Delta}, D_n Z_n^{\Delta}\right)$. Given the Markov property of the underlying processes, the Malliavin chain rule (Lemma~\ref{lemma1}) implies that
\begin{equation}
  D_n Y_m = \nabla_x y\left(t_m, X_m\right) D_n X_m, \quad  D_n Z_m = \nabla_x z\left(t_m, X_m\right) D_n X_m =: \gamma\left(t_m, X_m\right) D_n X_m,  
  \label{eq17}
\end{equation}
for some deterministic functions $y: [0, T] \times \mathbb{R}^{d} \to \mathbb{R}$ and $z: [0, T] \times \mathbb{R}^{d} \to \mathbb{R}^{1 \times d}$, where $\gamma: [0, T] \times \mathbb{R}^{d} \to \mathbb{R}^{d \times d} $ is the Jacobian matrix of $z\left(t_m, X_m\right)$. Note that from the Feynman-Kac relation~\eqref{eq3} we have that $z\left(t_m, X_m\right) = \nabla_x y\left(t_m, X_m\right) b\left(t_m, X_m\right)$. Hence, one can write that $D_n Y_m = z\left(t_m, X_m\right) b^{-1}\left(t_m, X_m\right) D_n X_m$. Using Theorem~\ref{theorem4}, we have that~\eqref{eq17} becomes
\begin{equation}
    Z_n^{\Delta} = Z_{n+1}^{\Delta} b^{-1}\left(t_{n+1}, X_{n+1}^{\Delta}\right) D_n X_{n+1}^{\Delta} + f_D\left(t_n, \mathbf{X}_n^{\Delta}, \mathbf{D}_n \mathbf{X}_n^{\Delta}\right) \Delta t_n - \Gamma_N^{\Delta} D_n X_n^{\Delta} \Delta W_n,
    \label{eq18}
\end{equation}
where due to the aforementioned relations $f_D\left(t, \mathbf{X}_n^{\Delta}, \mathbf{D}_n \mathbf{X}_n^{\Delta} \right)= \nabla_x f\left( t_n, \mathbf{X}_n^{\Delta} \right) D_n X_n^{\Delta} + \nabla_y f\left( t_n, \mathbf{X}_n^{\Delta}\right) Z_n^{\Delta} + \nabla_z f\left( t_n, \mathbf{X}_n^{\Delta}\right) \Gamma_n^{\Delta} D_n X_n^{\Delta}.$ 

After discretizing the integrals, our scheme is made fully implementable at each discrete time point $t_n$ by an appropriate function approximator to estimate the discrete unknown processes $\left(Y^{\Delta}_n, Z^{\Delta}_n, \Gamma^{\Delta}_n\right)$ in~\eqref{eq13} and~\eqref{eq18}. We estimate these unknown processes using DNNs and propose the following scheme:
\begin{itemize}
    \item Generate estimates $X^{\Delta}_{n+1}$ for $n = 0, 1, \ldots, N-1$ of SDE~\eqref{eq7} via~\eqref{eq11} and its discrete Malliavin derivative $D_n X_n^{\Delta}$, $D_n X_{n+1}^{\Delta}$ using~\eqref{eq14}.
    \item Set $$Y_N^{\Delta, \hat{\theta}} := g(X_N^{\Delta}), \quad Z_N^{\Delta, \hat{\theta}} := \nabla_x g(X_N^{\Delta}) b(t_N, X_N^{\Delta}), \quad \Gamma_N^{\Delta, \hat{\theta}} := \left[\nabla_x (\nabla_x g\, b)\right](t_N, X_N^{\Delta}).$$
    \item For each discrete time point $t_n$, $n = N-1, N-2, \ldots, 0$, use three independent DNNs $\phi^y_n(\cdot; \theta^y_n): \mathbb{R}^{d} \to \mathbb{R}$, $\phi^z_n(\cdot; \theta^z_n): \mathbb{R}^{d} \to \mathbb{R}^{1 \times d}$ and $\phi^{\gamma}_n(\cdot; \theta^{\gamma}_n): \mathbb{R}^{d} \to \mathbb{R}^{d \times d}$ to approximate the
    discrete processes $\left(Y_n^{\Delta}, Z_n^{\Delta}, \Gamma_n^{\Delta}\right)$, respectively. Train the parameter set $\theta_n = \left( \theta^y_n, \theta^z_n, \theta^{\gamma}_n\right)$ using the differential learning approach by constructing a loss function - as in~\eqref{eq6} - such that the dynamics of the discretized process $Y$ and $Z$ given by~\eqref{eq13} and~\eqref{eq18} are fulfilled, namely
    \begin{equation}
    \begin{aligned}
        \mathbf{L}_n^{\Delta}\left( \theta_n \right) & := \omega_1 \mathbf{L}^{y,\Delta}_n\left( \theta_n \right) + \omega_2  \mathbf{L}^{z,\Delta}_n\left( \theta_n \right),\\
        \mathbf{L}^{y,\Delta}_n\left( \theta_n \right)& :=\mathbb{E}\left[ \left\vert Y^{\Delta, \hat{\theta}}_{n+1} - \phi^y_n\left( X^{\Delta}_n; \theta^y_n \right) + f\left(t_n, \mathbf{X}^{\Delta, \theta}_n\right) \Delta t_n -  \phi^z_n\left( X^{\Delta}_n; \theta^z_n \right) \Delta W_n  \right\vert^2 \right],\\
        \mathbf{L}^{z,\Delta}_n\left( \theta_n \right) & := \mathbb{E}\left[ \left\vert Z^{\Delta, \hat{\theta}}_{n+1} b^{-1}\left( t_{n+1}, X_{n+1}^{\Delta} \right) D_n X^{\Delta}_{n+1} - \phi^z_n\left( X^{\Delta}_n; \theta^z_n \right) \right. \right.\\
        & \quad \left. \left. + f_D\left(t_n, \mathbf{X}^{\Delta, \theta}_n, \mathbf{D}_n\mathbf{X}_n^{\Delta, \theta}\right)\Delta t_n - \phi^{\gamma}_n\left( X_n^{\Delta}; \theta^{\gamma}_n\right) D_n X_n^{\Delta} \Delta W_n
        \right\vert^2\right],
    \end{aligned}
    \label{eq19}
    \end{equation}
    where for notational convinience $\mathbf{X}^{\Delta, \theta}_n:=\left( X_n^{\Delta}, \phi^y_n\left( X^{\Delta}_n; \theta^y_n \right), \phi^z_n\left( X^{\Delta}_n; \theta^z_n \right) \right)$ and $\mathbf{D}_n\mathbf{X}^{\Delta, \theta}_n:=\left( D_n X_n^{\Delta}, \phi^z_n\left( X^{\Delta}_n; \theta^y_n \right), \phi^{\gamma}_n\left( X^{\Delta}_n; \theta^z_n \right) D_n X_n^{\Delta} \right)$. Approximate the optimal parameters $\theta^*_n \in \argmin_{\theta_n \in \Theta_{n}} \mathbf{L}^{\Delta}_n\left( \theta_n \right)$ using a SGD method and receive the estimated parameters $\hat{\theta}_n = \left( \hat{\theta}^y_n, \hat{\theta}^z_n, \hat{\theta}^{\gamma}_n \right)$. Then, define
    $$Y_n^{\Delta, \hat{\theta}} := \phi^y_n\left( X^{\Delta}_n; \hat{\theta}^y_n \right), \quad Z_n^{\Delta, \hat{\theta}} := \phi^z_n\left( X^{\Delta}_n; \hat{\theta}^z_n \right), \quad \Gamma_n^{\Delta, \hat{\theta}} := \phi^{\gamma}_n \left( X^{\Delta}_n; \hat{\theta}^{\gamma}_n \right).$$ 
\end{itemize}
Note that the scheme in~\cite{hure2020deep} (the deep backward dynamic programming (DBDP) scheme) can be considered as a special case of our scheme by choosing $\omega_1 = 1$ and $\omega_2 = 0$. We refer to our scheme as DLBDP (differential learning backward dynamic programming) scheme, where $\omega_1 = \frac{1}{d+1}$ and $\omega_2 = \frac{d}{d+1}$ is considered due to dimensionality of the processes $Y$ and $Z$.

Our scheme offers several advantages over the DBDP scheme and other well-known deep learning-based approaches~\cite{weinan2017deep,germain2022approximation,raissi2024forward,kapllani2024deep}:
\begin{itemize}
    \item[(i)] By explicitly incorporating the dynamics of the process $Z$ via the BSDE~\eqref{eq10} in the loss function~\eqref{eq19}, we enhance the accuracy of $Z$ approximations through the SGD method.
    \item[(ii)] Additionally, the inclusion of the process $\Gamma$ in the loss function through BSDE~\eqref{eq10} allows for better estimation of $\Gamma$ within the DLBDP scheme compared to the deep learning-based schemes, where AD is required for approximation of $\Gamma$.
\end{itemize}
In comparison to the approach presented in~\cite{negyesi2024one}, which utilizes deep learning and the Malliavin derivative, our scheme demonstrates the following advantages:
\begin{itemize}
    \item[(i)] Since the scheme in~\cite{negyesi2024one} employs supervised deep learning, it requires solving two optimization problems per time step - one for BSDE~\eqref{eq8} and another for BSDE~\eqref{eq10} - to approximate the unknown processes $\left(Y, Z, \Gamma\right)$. Consequently, the computational cost of the scheme in~\cite{negyesi2024one} is expected to be higher than that of our scheme.
    \item[(ii)] Our scheme can be seamlessly extended not only to DBDP scheme, which is formulated backward in time through local optimizations at each discrete time step, but also to other supervised deep learning-based approaches, such as~\cite{weinan2017deep,raissi2024forward,kapllani2024deep}, which are formulated forward in time as a global optimization problem. This is part of our ongoing research. The approach proposed in~\cite{negyesi2024one} cannot be integrated to such schemes, as it cannot be formulated as a global optimization problem.
\end{itemize}

\section{Convergence analysis}
\label{sec5}
The main goal of this section is to prove the convergence of the DLBDP scheme towards the solution $\left(Y, Z, \Gamma\right)$ of the BSDE system~\eqref{eq7}-\eqref{eq10}, and provide a rate of convergence that depends on the discretization error from the Euler-Maruyama scheme and the approximation or model error by the DNNs. 

For the functions figuring in the BSDE system~\eqref{eq7}-\eqref{eq10}, the following assumptions are in place.
\begin{assump}{AX4}
    Assumption~\ref{AX3} holds, with the function $b(t,x)$ depends only on $t$. The functions $a(t,x)$ and $b(t,x)$ are $1/2$-Hölder continuous in time.
    \label{AX4}
\end{assump}    
\begin{assump}{AY4}
    Assumption~\ref{AY3} holds. Moreover, $g \in C^{2+\lfrac}_{\bfrac}\left( \mathbb{R}^d; \mathbb{R} \right)$, $\lfrac >0$. The function $f(t,x,y,z)$ and its partial derivatives $\nabla_x f$, $\nabla_y f$ and $\nabla_z f$ are all $1/2$-Hölder continuous in time.
    \label{AY4}
\end{assump}
We emphasise that our assumption for the SDE is less restrictive than that in~\cite{negyesi2024one}, where arithmetic Brownian motion (ABM) is assumed. When pricing and hedging options, usually the stock dynamics are modeled by the geometric Brownian motion (GBM). To ensure the applicability of our convergence analysis in such cases, we consider in the numerical section the ln-transformation of stock prices. Consequently, we obtain a drift and diffusion function that satisfy Assumption~\ref{AX4}, thereby ensuring that our theoretical analysis holds in the numerical experiments.

The following lemma is a consequence of the considered assumptions.
\begin{lemma}
Under Assumptions~\ref{AX4} and~\ref{AY4}, the Malliavin derivatives $\left( D_s X_t, D_s Y_t, D_s Z_t  \right)$ are bounded $\mathbb{P}$-a.s. for $0\leq s \leq t \leq T$.
\label{lemma2}
\end{lemma}
\begin{proof}
Due to Assumption~\ref{AX4}, we have that $\left\vert D_s X_t \right\vert \leq C$ $\mathbb{P}$-a.s. for $0\leq s \leq t \leq T$ (see~\cite{cheridito2014bsdes}, Lemma 4.2.). Moreover, the parabolic PDE~\eqref{eq3} has a classical solution $u \in C^{1,2}_{\bfrac}\left([0, T] \times \mathbb{R}^{d}; \mathbb{R}\right)$ (see~\cite{delarue2006forward}, Theorem 2.1). The boundedness of $\left(D_s Y_t, D_s Z_t  \right)$ follows after using the relations~\eqref{eq17}.
\end{proof}
From the mean-value theorem for $f \in C^{0,2,2,2}_{\bfrac}\left( [0, T]\times \mathbb{R}^d \times \mathbb{R} \times \mathbb{R}^{1 \times d}; \mathbb{R} \right)$, we have that $f$ and all its first-order derivatives in $(x, y, z)$ are Lipschitz continuous. Therefore, the following holds (using also Assumption~\ref{AY4} and Lemma~\ref{lemma2})
\begin{equation}    
\begin{split}
    | f(t_1, \mathbf{x}_1) - f(t_2, \mathbf{x}_2) | & \leq L_f \left( | t_1 - t_2 |^\frac{1}{2} + | x_1 - x_2 | + | y_1 - y_2 | + | z_1 - z_2 |\right),\\
    |f_D(t_1, \mathbf{x}_1, \mathbf{Dx}_1) - f_D (t_2, \mathbf{x}_2, \mathbf{Dx}_2) | & \leq L_{f_D} \left( | t_1 - t_2 |^\frac{1}{2} + | x_1 - x_2 | + | y_1 - y_2 | + | z_1 - z_2 | \right.\\
    &\quad \left. + | Dx_1 - Dx_2 | + | Dy_1 - Dy_2 | + | Dz_1 - Dz_2 | \vphantom{| t_1 - t_2 |^\frac{1}{2}} \right),\\
\end{split}
\label{eq20}
\end{equation}
with $\mathbf{x}_{\ifrac} = \left( x_{\ifrac}, y_{\ifrac}, z_{\ifrac} \right)$, $\mathbf{Dx}_{\ifrac} = \left( Dx_{\ifrac}, Dy_{\ifrac}, Dz_{\ifrac} \right)$ and $t_{\ifrac} \in [0,T]$, $x_{\ifrac} \in \mathbb{R}^d$, $y_{\ifrac} \in \mathbb{R}$, $z_{\ifrac}, Dy_{\ifrac} \in \mathbb{R}^{1 \times d}$, $Dx_{\ifrac}, Dz_{\ifrac} \in \mathbb{R}^{d \times d}$, where $L_f, L_{f_D} >0$ and ${\ifrac} = 1, 2$.  

Under Assumptions~\ref{AX4} and~\ref{AY4}, Theorems~\ref{theorem1},~\ref{theorem3} and~\ref{theorem5}, we have that the processes $(X, Y, Z, DX, DY )$ are all mean-squared continuous in time, i.e., there exists some constant $C>0$ such that $\forall\, s, r, t \in [0, T]$ 
\begin{equation}    
\begin{split}
    \mathbb{E}\left[ \left| X_t - X_s \right|^2 \right] & \leq C \left| t - s \right|, \quad \mathbb{E}\left[ \left| D_s X_t - D_sX_r \right|^2 \right] \leq C \left| t - r \right|,\\    
    \mathbb{E}\left[ \left| Y_t - Y_s \right|^2 \right] & \leq C \left| t - s \right|, \quad \mathbb{E}\left[ \left| Z_t - Z_s \right|^2 \right] \leq C \left| t - s \right|, \quad \mathbb{E}\left[ \left| D_sY_t - D_sY_r \right|^2 \right] \leq C \left| t - r \right|.    
\end{split}
\label{eq21}
\end{equation}
From Assumptions~\ref{AX4},~\ref{AY4} and Lemma~\ref{lemma2}, we also see for $0\leq s \leq t \leq T$ that 
\begin{equation}
    \mathbb{E}\left[ \int_{0}^T \left\vert f_D\left( t, \mathbf{X}_t, \mathbf{D}_s\mathbf{X}_t\right) \right\vert^2 dt \right] < \infty.
    \label{eq22}
\end{equation}
Moreover, we have the well-known error estimate that the Euler-Maruyama approximations in~\eqref{eq11} admit to
\begin{equation}
        \max_{0\leq n \leq N-1}\mathbb{E}\left[  \sup_{ t_n\leq t \leq t_{n+1} }\left| X_t - X_{t_n}^{\Delta} \right|^2\right] = \mathcal{O}\left( \left| \Delta \right| \right),
    \label{eq23}
\end{equation}
under Assumption~\ref{AX1} and the Hölder continuity assumption in~\ref{AX4} (see~\cite{zhang2017backward}, Theorem 5.3.1), where the notation $\mathcal{O}\left( \left| \Delta \right| \right)$ means that $\limsup_{\left| \Delta \right| \to 0 } \left| \Delta \right|^{-1} \mathcal{O}\left( \left| \Delta \right| \right) < \infty$. Note that under Assumption~\ref{AX2} and the Hölder continuity assumption in~\ref{AX4}, it can be showed that the Euler-Maruyama Malliavin derivative approximations $D_n X_{n+1}^{\Delta}$ in~\eqref{eq14} admit to similar error estimates as in~\eqref{eq23}
\begin{equation}
    \mathbb{E}\left[\left| D_n X_{n+1} - D_n X_{n+1}^{\Delta} \right|^2\right] = \mathcal{O}\left( \left| \Delta \right| \right).
    \label{eq24}
\end{equation}
Let us introduce the $\mathbb{L}^2$-regularity of $DZ$:
\begin{equation}
    \varepsilon^{DZ}\left( |\Delta| \right):= \mathbb{E}\left[ \sum_{n=0}^{N-1} \int_{t_n}^{t_{n+1}} \left\vert D_{n} Z_r - \hat{DZ}_n \right\vert^2 dr \right], \quad \hat{DZ}_n:=\frac{1}{\Delta t_n} \mathbb{E}_n\left[ \int_{t_n}^{t_{n+1}} D_{n} Z_s ds \right].
    \label{eq25}
\end{equation}
Since $\hat{DZ}_n$ is an $\mathbb{L}^2$-projection of $DZ$, we have that $\varepsilon^{DZ}\left( |\Delta| \right)$ converges to zero for $|\Delta|$ going to zero. Moreover, as mentioned in~\cite{negyesi2024one}, since the terminal condition of the Malliavin BSDE~\eqref{eq10} is also Lipschitz continuous due to Assumption~\ref{AY4}, we have that
\begin{equation*}
    \varepsilon^{DZ}\left( |\Delta| \right) = \mathcal{O}\left( | \Delta | \right),
\end{equation*}
after applying~\cite{zhang2004numerical} (Theorem 3.1).

We now define 
\begin{equation}
    \begin{cases}
        \hat{Y}_n^{\Delta} & := \mathbb{E}_n\left[ Y_{n+1}^{\Delta, \hat{\theta}} \right] + f\left(t_n, \hat{\mathbf{X}}_n^{\Delta}\right) \Delta t_n,\\
        \hat{Z}_n^{\Delta} & := \mathbb{E}_n\left[ Z_{n+1}^{\Delta, \hat{\theta}} b^{-1}\left( t_{n+1}, X_{n+1}^{\Delta} \right) D_n X_{n+1}^{\Delta} \right] + f_D\left(t_n,\hat{\mathbf{X}}_n^{\Delta}, \mathbf{D}_n \hat{\mathbf{X}}_n^{\Delta}\right) \Delta t_n,\\
        \hat{\Gamma}_n^{\Delta} & := \frac{1}{\Delta t_n}\mathbb{E}_n\left[ \Delta W_n Z_{n+1}^{\Delta, \hat{\theta}} b^{-1}\left( t_{n+1}, X_{n+1}^{\Delta} \right) D_n X_{n+1}^{\Delta} \right] b^{-1}\left( t_{n}, X_{n}^{\Delta} \right),
    \end{cases}
    \label{eq26}
\end{equation}
for $n=0, \ldots, N-1$, where $\hat{\mathbf{X}}_n := \left( X_n^{\Delta}, \hat{Y}_n^{\Delta}, \hat{Z}_n^{\Delta}\right)$ and $\mathbf{D}_n \hat{\mathbf{X}}_n := \left( D_n X_n^{\Delta}, \hat{Z}_n^{\Delta}, \hat{\Gamma}_n^{\Delta} b(t_n, X_n^{\Delta})\right)$. Note that  $\hat{Y}_n^{\Delta}$ and $\hat{Z}_n^{\Delta}$ in~\eqref{eq26} are calculated by taking $\mathbb{E}_n[\cdot]$ in~\eqref{eq13} and~\eqref{eq18}, where $\mathbb{E}_n\left[\hat{Z}_n^{\Delta} \Delta W_n\right] = 0$ and $\mathbb{E}_n\left[\hat{\Gamma}_n^{\Delta} b(t_n, X_n^{\Delta})\Delta W_n\right] = 0$. Moreover, $\hat{\Gamma}_n^{\Delta}$ in~\eqref{eq26} is calculated by multiplying both sides of~\eqref{eq18} with $\Delta W_n$, where $\mathbb{E}_n\left[\Delta W_n f_D\left(t_n,\hat{\mathbf{X}}_n^{\Delta}, \mathbf{D}_n \hat{\mathbf{X}}_n^{\Delta}\right)\right] = 0$. Finally, applying the It\^o isometry gives $\hat{\Gamma}_n^{\Delta}$ in~\eqref{eq26}.

By the Markov property of the underlying processes, there exist some deterministic functions $\hat{y}_n$, $\hat{z}_n$ and $\hat{\gamma}_n$ such that
\begin{equation}
    \hat{Y}_n^{\Delta} = \hat{y}_n\left( X_n^{\Delta} \right), \quad \hat{Z}_n^{\Delta} = \hat{z}_n\left( X_n^{\Delta} \right), \quad \hat{\Gamma}_n^{\Delta} = \hat{\gamma}_n\left( X_n^{\Delta} \right), \quad n = 0, \ldots, N-1. 
\label{eq27}
\end{equation}
Moreover, by the martingale representation theorem, there exists an $\mathbb{R}^{d \times d}-$valued square integrable process $D_n \hat{Z}_t$ such that
\begin{equation}
     D_n Y_{n+1}^{\Delta, \hat{\theta}} = \hat{Z}_n^{\Delta} - f_D\left(t_n, \hat{\mathbf{X}}_n^{\Delta}, \mathbf{D}_n \hat{\mathbf{X}}_n^{\Delta}\right) \Delta t_n + \int_{t_n}^{t_{n+1}} D_n \hat{Z}_s dW_s,
\label{eq28}
\end{equation}
and by It\^o isometry, we have
\begin{equation*}
    D_n \hat{Z}_n^{\Delta} = \hat{\Gamma}_n^{\Delta} b(t_n, X_n^{\Delta})  =\frac{1}{\Delta t_n} \mathbb{E}_n\left[ \int_{t_n}^{t_{n+1}} D_n \hat{Z}_s ds\right], \quad n=0, \ldots, N-1.
\end{equation*}
Hence, $D \hat{Z}^{\Delta}$ is an $\mathbb{L}^2-$projection of $D \hat{Z}$. Moreover, $\hat{Z}^{\Delta}$ is an $\mathbb{L}^2-$projection of $\hat{Z}$ such that
\begin{equation}
    Y_{n+1}^{\Delta, \hat{\theta}} = \hat{Y}_n^{\Delta} - f\left(t_n, \hat{\mathbf{X}}_n^{\Delta}\right) \Delta t_n + \int_{t_n}^{t_{n+1}} \hat{Z}_s dW_s.
\label{eq29}
\end{equation}
Finally, we define the approximation errors of $\hat{y}_n$, $\hat{z}_n$ and $\hat{\gamma}_n$ by the DNNs $\phi^y_n$, $\phi^z_n$ and $\phi^{\gamma}_n$ defined as
\begin{equation}
    \begin{split}
        \varepsilon_n^{y}&:= \inf_{\theta^y_n \in \Theta^y_n} \mathbb{E}\left[ \left\vert \hat{y}_n\left( X_n^{\Delta} \right) - \phi^{y}_n\left(X_n^{\Delta}; \theta_n^y\right)\right\vert^2 \right],\\
        \varepsilon_n^{z}&:= \inf_{\theta^z_n \in \Theta^z_n} \mathbb{E}\left[ \left\vert \hat{z}_n\left( X_n^{\Delta} \right) - \phi^{z}_n\left(X_n^{\Delta}; \theta_n^z\right)\right\vert^2 \right],\\ \varepsilon_n^{\gamma}&:= \inf_{\theta^{\gamma}_n \in \Theta^{\gamma}_n} \mathbb{E}\left[ \left\vert \left(\hat{\gamma}_n\left( X_n^{\Delta} \right) - \phi^{\gamma}_n\left(X_n^{\Delta}; \theta_n^{\gamma}\right)\right) b(t_n, X_n^{\Delta})\right\vert^2 \right],
    \end{split}
    \label{eq30}
\end{equation}
for $n = 0, \ldots, N-1.$ The goal is now to find an upper bound of the total approximation error of the DLBDP scheme defined as 
\begin{equation*}
    \begin{split}        
    \mathcal{E}\left[\left(Y, Z, \Gamma\right), \left(Y^{\Delta, \hat{\theta}}, Z^{\Delta, \hat{\theta}}, \Gamma^{\Delta, \hat{\theta}}\right) \right]&:= \max_{0\leq n\leq N} \mathbb{E}\left[ \left\vert Y_n - Y^{\Delta, \hat{\theta}}_n\right\vert^2 \right] + \max_{0\leq n\leq N} \mathbb{E}\left[\left\vert Z_n - Z^{\Delta, \hat{\theta}}_n\right\vert^2 \right]\\
    & \qquad + \mathbb{E}\left[ \sum_{n=0}^{N-1} \int_{t_n}^{t_{n+1}} \left\vert D_n Z_s - D_n Z^{\Delta, \hat{\theta}}_n \right\vert^2 ds \right],
    \end{split}
\end{equation*}
in terms of the discretization error (from the Euler-Maruyama scheme) and the approximation errors~\eqref{eq30} by the DNNs, where $D_n Z_s - D_n Z^{\Delta, \hat{\theta}}_n = \Gamma_s b(s) - \Gamma_n^{\Delta, \hat{\theta}} b(t_n)$ due to relations~\eqref{eq17} and Assumption~\ref{AX4}.
\begin{theorem}
(Consistency of DLBDP scheme). Under Assumptions~\ref{AX4} and~\ref{AY4}, there exist a constant $C>0$ independent of the time partition $\Delta$, such that the total approximation error of the DLBDP scheme admits to
\begin{align*}
    \mathcal{E}\left[\left(Y, Z, \Gamma\right), \left(Y^{\Delta, \hat{\theta}}, Z^{\Delta, \hat{\theta}}, \Gamma^{\Delta, \hat{\theta}}\right) \right]
    & \leq C \left\{ \vphantom{\sum_{n=0}^{N-1}} \mathbb{E}\left[ \left\vert g(X_T) - g\left(X_N^{\Delta}\right) \right\vert^2\right] \right.\\
    &\quad \left. + \mathbb{E}\left[ \left\vert \nabla_x g(X_T) - \nabla_x g\left(X_N^{\Delta}\right) \right\vert^2\right] + |\Delta| + \varepsilon^{DZ}\left( |\Delta| \right) \right.\\
    & \quad \left. + N \sum_{n=0}^{N-1} \left( \omega_1 \varepsilon_n^{y} + \omega_2  \varepsilon_n^{z} \right)\right.\\
    & \quad \left. + \sum_{n=0}^{N-1}\left( \omega_2 \varepsilon_n^{y} + \omega_1  \varepsilon_n^{z} + \omega_2 \varepsilon_n^{\gamma}\right)\right\}.
\end{align*}
\label{theorem6}
\end{theorem}
\begin{proof}
In the following, $C$ denotes a positive generic constant independent of $\Delta$, which may take different values from line to line.

{\it Step 1.} Let us fix $n \in \{0, 1, \ldots, N-1\}$. After taking $\mathbb{E}_n\left[ \cdot \right]$ in~\eqref{eq12} and using the relation for $\hat{Y}_n^{\Delta}$ in~\eqref{eq26}, we get
\begin{equation*}
   Y_n - \hat{Y}_n^{\Delta} = \mathbb{E}_n\left[Y_{n+1} - Y_{n+1}^{\Delta, \hat{\theta}}\right] + \mathbb{E}_n\left[ \int_{t_{n}}^{t_{n+1}} f\left(s, \mathbf{X}_s\right) - f\left(t_n, \hat{\mathbf{X}}_n^{\Delta}\right)\,ds \right].
\end{equation*} 
Using the Jensen inequality for the second term above and then the $H^2\left([0, T]; \mathbb{R} \right)$ Cauchy-Schwarz inequality, we have
\begin{align*}
    \left\vert Y_n - \hat{Y}_n^{\Delta} \right\vert & \leq  \left\vert \mathbb{E}_n\left[ Y_{n+1} - Y_{n+1}^{\Delta, \hat{\theta}} \right]  \right\vert + \left(\Delta t_n\right)^\frac{1}{2}\left( \mathbb{E}_n\left[  \int_{t_{n}}^{t_{n+1}} \left\vert f\left(s, \mathbf{X}_s\right)  - f\left(t_n, \hat{\mathbf{X}}_n^{\Delta}\right)\right\vert^2 ds  \right] \right)^\frac{1}{2}.
\end{align*}
Using again the Jensen inequality for the first term above and the Young inequality of the form
\begin{equation}
  \left( c_1 + c_2\right)^2 \leq \left(1 + \nu \Delta t_n \right)c_1^2 + \left( 1 + \frac{1}{\nu \Delta t_n} \right) c_2^2, \quad \nu >0, 
  \label{eq31}
\end{equation}
we get (after taking $\mathbb{E}\left[ \cdot\right]$) that
\begin{align*}
    \mathbb{E}\left[\left\vert Y_n - \hat{Y}_n^{\Delta} \right\vert^2\right] & \leq   \left(1 + \nu \Delta t_n \right) \mathbb{E}\left[\left\vert Y_{n+1} - Y_{n+1}^{\Delta, \hat{\theta}} \right\vert^2\right] \\
    & \quad + \frac{1}{\nu}\left( 1 +  \nu \Delta t_n \right) \mathbb{E}\left[  \int_{t_{n}}^{t_{n+1}} \left\vert f\left(s, \mathbf{X}_s\right) - f\left(t_n, \hat{\mathbf{X}}_n^{\Delta}\right)\right\vert^2 ds  \right].
\end{align*}
The Lipschitz and Hölder continuity of $f$ in~\eqref{eq20} and the inequality $\left(\sum_{\ifrac=1}^4 c_{\ifrac}\right)^2 \leq 4 \sum_{\ifrac=1}^4 c_{\ifrac}^2$ yields
\begin{equation}
    \begin{aligned}
    & \mathbb{E}\left[ \int_{t_{n}}^{t_{n+1}} \left\vert f\left(s, \mathbf{X}_s\right) -  f\left(t_n, \hat{\mathbf{X}}_n^{\Delta}\right) \right\vert^2 ds \right]\\
    & \leq 4 L_{f}^2 \left( \int_{t_{n}}^{t_{n+1}} \left\vert s - t_n \right\vert ds + \int_{t_{n}}^{t_{n+1}} \mathbb{E}\left[\left\vert X_s - X_n^{\Delta} \right\vert^2\right] ds + \int_{t_{n}}^{t_{n+1}} \mathbb{E}\left[\left\vert Y_s - \hat{Y}_n^{\Delta} \right\vert^2\right] ds \right.\\
    & \left. \quad + \int_{t_{n}}^{t_{n+1}} \mathbb{E}\left[\left\vert Z_s - \hat{Z}_n^{\Delta} \right\vert^2\right] ds\right)
\end{aligned}
\label{eq32}
\end{equation}
Due to the mean squared continuities~\eqref{eq21} and the inequality $\left( c_1 + c_2 \right)^2 \leq 2\left( c_1^2 + c_2^2 \right)$, we have
\begin{align*}
    \int_{t_{n}}^{t_{n+1}} \mathbb{E}\left[\left\vert X_s - X_n^{\Delta} \right\vert^2\right] ds &= \int_{t_{n}}^{t_{n+1}} \mathbb{E}\left[\left\vert \left(X_s - X_n\right) + \left( X_n - X_n^{\Delta}\right) \right\vert^2\right] ds,\\
    & \leq \int_{t_{n}}^{t_{n+1}} \mathbb{E}\left[ \left( \left\vert X_s - X_n \right\vert  + \left\vert X_n - X_n^{\Delta}\right\vert \right)^2\right] ds,\\        
    & \leq C |\Delta|^2 + 2 \Delta t_n \mathbb{E}\left[ \left\vert  X_n - X_n^{\Delta} \right\vert^2 \right].
\end{align*}
Performing similar calculations for other terms in~\eqref{eq32}, we gather
\begin{equation}
\begin{aligned}
    \mathbb{E}\left[\left\vert Y_n - \hat{Y}_n^{\Delta} \right\vert^2\right] & \leq \left( 1 +  \nu \Delta t_n \right) \mathbb{E}\left[ \left\vert Y_{n+1} - Y_{n+1}^{\Delta, \hat{\theta}} \right\vert^2\right]\\
    & \quad + \frac{4 L_{f}^2}{\nu} \left( 1 +  \nu \Delta t_n \right) \left\{ \vphantom{\mathbb{E}\left[ \left\vert \Delta Z_n^{\Delta} \right\vert^2 \right]} C |\Delta|^2 + 2\Delta t_n \left(\mathbb{E}\left[ \left\vert Y_n - \hat{Y}_n^{\Delta} \right\vert^2 \right] +\mathbb{E}\left[ \left\vert Z_n - \hat{Z}_n^{\Delta} \right\vert^2 \right] \right) \right\},
\end{aligned}
\label{eq33}
\end{equation}
where we also used~\eqref{eq23}.

{\it Step 2.} By taking $\mathbb{E}_n\left[ \cdot \right]$ in~\eqref{eq15} and using the relation for $\hat{Z}_n^{\Delta}$ in~\eqref{eq26}, we have
\begin{equation*}
   Z_n - \hat{Z}_n^{\Delta} = \mathbb{E}_n\left[ D_n Y_{n+1} - D_n Y_{n+1}^{\Delta, \hat{\theta}} \right] + \mathbb{E}_n\left[ \int_{t_{n}}^{t_{n+1}} f_D\left(s, \mathbf{X}_s, \mathbf{D}_n \mathbf{X}_s\right) - f_D\left(t_n, \hat{\mathbf{X}}_n^{\Delta}, \mathbf{D}_n \hat{\mathbf{X}}_n^{\Delta}\right)\,ds \right],
\end{equation*} 
where $Z_n = D_n Y_n$ due to Theorem~\ref{theorem4}. Similarly as in the previous step, namely applying Jensen inequality for the second term above, using the $H^2\left([0, T]; \mathbb{R}^d \right)$ Cauchy-Schwarz inequality and the Young inequality~\eqref{eq31}, we have
\begin{align*}
    \mathbb{E}\left[\left\vert Z_n - \hat{Z}_n^{\Delta} \right\vert^2\right] & \leq   \left(1 + \nu \Delta t_n \right) \mathbb{E}\left[\left\vert \mathbb{E}_n\left[ D_n Y_{n+1} - D_n Y_{n+1}^{\Delta, \hat{\theta}} \right]  \right\vert^2\right] \\
    & \quad + \frac{1}{\nu}\left( 1 +  \nu \Delta t_n \right) \mathbb{E}\left[  \int_{t_{n}}^{t_{n+1}} \left\vert f_D\left(s, \mathbf{X}_s, \mathbf{D}_n \mathbf{X}_s\right) - f_D\left(t_n, \hat{\mathbf{X}}_n^{\Delta}, \mathbf{D}_n \hat{\mathbf{X}}_n^{\Delta}\right)\right\vert^2 ds  \right].
\end{align*}
From the Lipschitz and Hölder continuity of $f_D$ in~\eqref{eq20} and the inequality $\left(\sum_{\ifrac=1}^7 c_{\ifrac}\right)^2 \leq 8 \sum_{\ifrac=1}^7 c_{\ifrac}^2$, we get
\begin{equation*}
    \begin{aligned}
    & \mathbb{E}\left[ \int_{t_{n}}^{t_{n+1}} \left\vert f_D\left(s, \mathbf{X}_s, \mathbf{D}_n \mathbf{X}_s\right) -  f_D\left(t_n, \hat{\mathbf{X}}_n^{\Delta}, \mathbf{D}_n \hat{\mathbf{X}}_n^{\Delta}\right) \right\vert^2 ds \right]\\
    & \leq 8 L_{f_D}^2 \left( \int_{t_{n}}^{t_{n+1}} \left\vert s - t_n \right\vert ds + \int_{t_{n}}^{t_{n+1}} \mathbb{E}\left[\left\vert X_s - X_n^{\Delta} \right\vert^2\right] ds + \int_{t_{n}}^{t_{n+1}} \mathbb{E}\left[\left\vert Y_s - \hat{Y}_n^{\Delta} \right\vert^2\right] ds \right.\\
    & \left. \quad + \int_{t_{n}}^{t_{n+1}} \mathbb{E}\left[\left\vert Z_s - \hat{Z}_n^{\Delta} \right\vert^2\right] ds + \int_{t_{n}}^{t_{n+1}} \mathbb{E}\left[\left\vert D_n X_s - D_n X_n^{\Delta} \right\vert^2\right] ds\right.\\
    & \left. \quad + \int_{t_{n}}^{t_{n+1}} \mathbb{E}\left[\left\vert D_n Y_s - D_n \hat{Y}_n^{\Delta} \right\vert^2\right] ds + \int_{t_{n}}^{t_{n+1}} \mathbb{E}\left[\left\vert D_n Z_s - D_n \hat{Z}_n^{\Delta} \right\vert^2\right] ds \vphantom{\int_{t_{n}}^{t_{n+1}}} \right).
\end{aligned}
\end{equation*}
The mean squared continuities~\eqref{eq21} and the inequality $\left( c_1 + c_2 \right)^2 \leq 2\left( c_1^2 + c_2^2 \right)$ yields
\begin{align*}
    \mathbb{E}\left[\left\vert Z_n - \hat{Z}_n^{\Delta} \right\vert^2\right] & \leq   \left(1 + \nu \Delta t_n \right) \mathbb{E}\left[\left\vert \mathbb{E}_n\left[ D_n Y_{n+1} - D_n Y_{n+1}^{\Delta, \hat{\theta}} \right]  \right\vert^2\right] \\
    & \quad + \frac{8 L_{f_D}^2}{\nu} \left( 1 +  \nu \Delta t_n \right) \left\{ \vphantom{\mathbb{E}\left[\int_{t_{n}}^{t_{n+1}}\right]} C |\Delta|^2 + 2 \Delta t_n\left( \mathbb{E}\left[ \left\vert X_n - X_n^{\Delta} \right\vert^2 \right] + \mathbb{E}\left[ \left\vert Y_n - \hat{Y}_n^{\Delta} \right\vert^2 \right] \right. \right. \\
    & \left. \left. \quad + \mathbb{E}\left[ \left\vert Z_n - \hat{Z}_n^{\Delta} \right\vert^2 \right]  \right)  + 2 \Delta t_n\left( \mathbb{E}\left[ \left\vert D_n X_n - D_n X_n^{\Delta} \right\vert^2 \right] + \mathbb{E}\left[ \left\vert D_n Y_n - D_n \hat{Y}_n^{\Delta} \right\vert^2 \right]\right) \right.\\
    & \left. \quad + \mathbb{E}\left[\int_{t_{n}}^{t_{n+1}} \left\vert D_n Z_s - D_n \hat{Z}_n^{\Delta} \right\vert^2 ds\right] \right\}.
\end{align*}
By using~\eqref{eq23}, $D_n Y_n - D_n \hat{Y}_n^{\Delta} = Z_n - \hat{Z}_n^{\Delta}$ and $\mathbb{E}\left[ \left\vert D_n X_n - D_n X_n^{\Delta} \right\vert^2 \right] = 0$ (due to Assumption~\ref{AX4}), we have
\begin{align*}
    \mathbb{E}\left[\left\vert Z_n - \hat{Z}_n^{\Delta} \right\vert^2\right] & \leq   \left(1 + \nu \Delta t_n \right) \mathbb{E}\left[\left\vert \mathbb{E}_n\left[ D_n Y_{n+1} - D_n Y_{n+1}^{\Delta, \hat{\theta}} \right]  \right\vert^2\right] \\
    & \quad + \frac{8 L_{f_D}^2}{\nu} \left( 1 +  \nu \Delta t_n \right) \left\{ \vphantom{\mathbb{E}\left[\int_{t_{n}}^{t_{n+1}}\right]} C |\Delta|^2 + 2 \Delta t_n\left(\mathbb{E}\left[ \left\vert Y_n - \hat{Y}_n^{\Delta} \right\vert^2 \right]+ 2\mathbb{E}\left[ \left\vert Z_n - \hat{Z}_n^{\Delta} \right\vert^2 \right]  \right) \right. \\
    & \left. \quad + \mathbb{E}\left[\int_{t_{n}}^{t_{n+1}} \left\vert D_n Z_s - D_n \hat{Z}_n^{\Delta} \right\vert^2 ds\right] \right\}.
\end{align*}
By the definition of $\hat{DZ}_n$ in~\eqref{eq25}, the last term above is given as
\begin{equation}
        \mathbb{E}\left[\int_{t_{n}}^{t_{n+1}} \left\vert D_n Z_s - D_n \hat{Z}_n^{\Delta} \right\vert^2 ds\right] =\mathbb{E}\left[\int_{t_{n}}^{t_{n+1}} \left\vert D_n Z_s - \hat{DZ}_n \right\vert^2 ds\right] + \Delta t_n \mathbb{E}\left[\left\vert \hat{DZ}_n - D_n \hat{Z}_n^{\Delta} \right\vert^2\right].
\label{eq34}
\end{equation}
Hence, we get
\begin{equation}
\begin{aligned}
    \mathbb{E}\left[\left\vert Z_n - \hat{Z}_n^{\Delta} \right\vert^2\right] & \leq   \left(1 + \nu \Delta t_n \right) \mathbb{E}\left[\left\vert \mathbb{E}_n\left[ D_n Y_{n+1} - D_n Y_{n+1}^{\Delta, \hat{\theta}} \right]  \right\vert^2\right] \\
    & \quad + \frac{8 L_{f_D}^2}{\nu} \left( 1 +  \nu \Delta t_n \right) \left\{ \vphantom{\mathbb{E}\left[\int_{t_{n}}^{t_{n+1}}\right]} C |\Delta|^2 + 2 \Delta t_n\left(\mathbb{E}\left[ \left\vert Y_n - \hat{Y}_n^{\Delta} \right\vert^2 \right]+ 2\mathbb{E}\left[ \left\vert  Z_n - \hat{Z}_n^{\Delta} \right\vert^2 \right]  \right) \right. \\
    & \left. \quad + \mathbb{E}\left[\int_{t_{n}}^{t_{n+1}} \left\vert D_n Z_s - \hat{DZ}_n \right\vert^2 ds\right] + \Delta t_n \mathbb{E}\left[\left\vert \hat{DZ}_n - D_n \hat{Z}_n^{\Delta} \right\vert^2\right] \right\}.
\end{aligned}
\label{eq35}
\end{equation}
Next, we find an upper bound for $\Delta t_n \mathbb{E}\left[\left\vert \hat{DZ}_n - D_n \hat{Z}_n^{\Delta} \right\vert^2\right]$ in~\eqref{eq35}. By multiplying both sides of~\eqref{eq15} with $\Delta W_n$, taking $\mathbb{E}_n[\cdot]$, using the It\^o's isometry and~\eqref{eq25}, we have together with~\eqref{eq26}
\begin{equation*} 
    \begin{aligned}
    \Delta t_n \left(\hat{DZ}_n - D_n \hat{Z}_n^{\Delta} \right)& = \mathbb{E}_n\left[  \Delta W_n \left(D_n Y_{n+1} - D_n Y_{n+1}^{\Delta, \hat{\theta}} - \mathbb{E}_n\left[   D_n Y_{n+1} - D_n Y_{n+1}^{\Delta, \hat{\theta}} \right] \right) \right]\\
    & \quad + \mathbb{E}_n\left[ \Delta W_n \int_{t_{n}}^{t_{n+1}} f_D\left(s, \mathbf{X}_s, \mathbf{D}_n \mathbf{X}_s\right) ds \right],
    \end{aligned} 
\end{equation*} 
after rewriting the relation for $\hat{\Gamma}_n^{\Delta}$ in~\eqref{eq26} as $\Delta t_n D_n \hat{Z}_n^{\Delta} = \mathbb{E}_n\left[ \Delta W_n D_n Y_{n+1}^{\Delta, \hat{\theta}} \right]$ and that $\mathbb{E}_n\left[  \Delta W_n \mathbb{E}_n\left[ D_n Y_{n+1} - D_n Y_{n+1}^{\Delta, \hat{\theta}}\right] \right] = 0$. The conditional $\mathbb{L}^2\left( \Omega; \mathbb{R}^d \right)$ Cauchy-Schwarz inequality in the Frobenius norm and the independence of Brownian motion increments implies
\begin{align*}
    \Delta t_n \left\vert \hat{DZ}_n - D_n \hat{Z}_n^{\Delta} \right\vert & \leq \left( d \Delta t_n \right)^\frac{1}{2} \left( \mathbb{E}_n\left[ \left\vert D_n Y_{n+1} - D_n Y_{n+1}^{\Delta, \hat{\theta}} - \mathbb{E}_n\left[ D_n Y_{n+1} - D_n Y_{n+1}^{\Delta, \hat{\theta}} \right] \right\vert^2 \right]  \right)^\frac{1}{2}\\
    & \quad + \left( d \Delta t_n \right)^\frac{1}{2}\left(\mathbb{E}_n\left[ \left\vert \int_{t_{n}}^{t_{n+1}} f_D\left(s, \mathbf{X}_s, \mathbf{D}_n \mathbf{X}_s\right)ds \right \vert^2 \right] \right)^\frac{1}{2}.
\end{align*} 
By applying the $H^2\left([0, T];\mathbb{R}^d\right)$ Cauchy-Schwarz inequality for the last term above, using the inequality $\left( c_1 + c_2 \right)^2 \leq 2\left( c_1^2 + c_2^2 \right)$ and the law of total expectation yields
\begin{equation}    
\begin{aligned}
    \Delta t_n \mathbb{E}\left[\left\vert \hat{DZ}_n - D_n \hat{Z}_n^{\Delta} \right\vert^2 \right] & \leq 2d\left( \mathbb{E}\left[ \left\vert D_n Y_{n+1} - D_n Y_{n+1}^{\Delta, \hat{\theta}} \right\vert^2\right] - \mathbb{E}\left[\left\vert \mathbb{E}_n\left[   D_n Y_{n+1} - D_n Y_{n+1}^{\Delta, \hat{\theta}} \right] \right\vert^2 \right]  \right)\\
    & \quad + 2d \Delta t_n \mathbb{E}\left[ \int_{t_{n}}^{t_{n+1}} \left\vert f_D\left(s, \mathbf{X}_s, \mathbf{D}_n \mathbf{X}_s\right) \right\vert^2 ds \right].
\end{aligned}
\label{eq36}
\end{equation}
Using the upper bound~\eqref{eq36} in~\eqref{eq35} and choosing $\nu \equiv \check{\nu} = 16 d L_{f_D}^2$, this implies
\begin{equation}    
\begin{aligned}
    \mathbb{E}\left[\left\vert Z_n - \hat{Z}_n^{\Delta} \right\vert^2\right] & \leq \left( 1 +  \check{\nu} \Delta t_n \right) \mathbb{E}\left[ \left\vert D_n Y_{n+1} - D_n Y_{n+1}^{\Delta, \hat{\theta}} \right\vert^2\right]\\
    & \quad + \frac{1}{2d} \left( 1 +  \check{\nu} \Delta t_n \right) \left\{ \vphantom{\mathbb{E}\left[\int_{t_{n}}^{t_{n+1}}\right]} C |\Delta|^2 + 2\Delta t_n\left(\mathbb{E}\left[ \left\vert Y_n - \hat{Y}_n^{\Delta} \right\vert^2 \right] + 2 \mathbb{E}\left[ \left\vert Z_n - \hat{Z}_n^{\Delta} \right\vert^2 \right]\right)\right. \\ 
    &\quad \left. + \mathbb{E}\left[\int_{t_{n}}^{t_{n+1}} \left\vert D_n Z_s - \hat{DZ}_n \right\vert^2 ds\right] \right\}\\
    &\quad + \left( 1 +  \check{\nu} \Delta t_n \right) \Delta t_n \mathbb{E}\left[ \int_{t_{n}}^{t_{n+1}} \left\vert f_D\left(s, \mathbf{X}_s, \mathbf{D}_n \mathbf{X}_s\right) \right\vert^2 ds \right].
\end{aligned}
\label{eq37}
\end{equation}

{\it Step 3.} Combining~\eqref{eq33} and~\eqref{eq37} gives
\begin{align*}
    & \mathbb{E}\left[\left\vert Y_n - \hat{Y}_n^{\Delta} \right\vert^2\right] + \mathbb{E}\left[\left\vert Z_n - \hat{Z}_n^{\Delta} \right\vert^2\right] \\
    & \leq \left( 1 +  \nu \Delta t_n \right) \mathbb{E}\left[ \left\vert Y_{n+1} - Y_{n+1}^{\Delta, \hat{\theta}} \right\vert^2\right] +  \left( 1 +  \check{\nu} \Delta t_n \right) \mathbb{E}\left[ \left\vert D_n Y_{n+1} - D_n Y_{n+1}^{\Delta, \hat{\theta}} \right\vert^2\right]\\
    &\quad + \frac{4 L_{f}^2}{\nu} \left( 1 +  \nu \Delta t_n \right) \left\{ C |\Delta|^2 + \Delta t_n \left(\mathbb{E}\left[ \left\vert Y_n - \hat{Y}_n^{\Delta} \right\vert^2 \right] +\mathbb{E}\left[ \left\vert Z_n - \hat{Z}_n^{\Delta} \right\vert^2 \right] \right) \right\}\\
    &\quad + \frac{1}{2d} \left( 1 +  \check{\nu} \Delta t_n \right)\left\{ C |\Delta|^2 +4\Delta t_n \left(\mathbb{E}\left[ \left\vert Y_n - \hat{Y}_n^{\Delta} \right\vert^2 \right] + \mathbb{E}\left[ \left\vert Z_n - \hat{Z}_n^{\Delta} \right\vert^2 \right]\right) \right.\\
    &\quad \left. + \mathbb{E}\left[\int_{t_{n}}^{t_{n+1}} \left\vert D_n Z_s - \hat{DZ}_n \right\vert^2 ds\right]\right\}\\
    &\quad + \left( 1 +  \check{\nu} \Delta t_n \right) \Delta t_n\mathbb{E}\left[ \int_{t_{n}}^{t_{n+1}} \left\vert f_D\left(s, \mathbf{X}_s, \mathbf{D}_n \mathbf{X}_s\right) \right\vert^2 ds \right].
\end{align*}
Moreover,
\begin{align*}
    & \left(1 - \left( \frac{4 L_{f}^2}{\nu} \left( 1 +  \nu \Delta t_n \right) + \frac{2}{d} \left( 1 +  \check{\nu} \Delta t_n \right) \right)\Delta t_n \right)\left(\mathbb{E}\left[\left\vert Y_n - \hat{Y}_n^{\Delta} \right\vert^2\right] + \mathbb{E}\left[\left\vert Z_n - \hat{Z}_n^{\Delta} \right\vert^2\right] \right)\\
    & \leq \left( 1 +  \nu \Delta t_n \right) \mathbb{E}\left[ \left\vert Y_{n+1} - Y_{n+1}^{\Delta, \hat{\theta}} \right\vert^2\right] +  \left( 1 +  \check{\nu} \Delta t_n \right) \mathbb{E}\left[ \left\vert D_n Y_{n+1} - D_n Y_{n+1}^{\Delta, \hat{\theta}} \right\vert^2\right]\\
    & \quad + \frac{4 L_{f}^2}{\nu} \left( 1 +  \nu \Delta t_n \right) C |\Delta|^2 + \frac{1}{2d}\left( 1 +  \check{\nu} \Delta t_n \right) \left\{ C |\Delta|^2 + \mathbb{E}\left[\int_{t_{n}}^{t_{n+1}} \left\vert D_n Z_s - \hat{DZ}_n \right\vert^2 ds\right]\right\}\\
    &\quad + \left( 1 +  \check{\nu} \Delta t_n \right) \Delta t_n\mathbb{E}\left[ \int_{t_{n}}^{t_{n+1}} \left\vert f_D\left(s, \mathbf{X}_s, \mathbf{D}_n \mathbf{X}_s\right) \right\vert^2 ds \right].
\end{align*}
Then, for any given $\nu >0$ and $|\Delta|$ sufficiently small:
\begin{equation}
\begin{aligned}
    & \mathbb{E}\left[\left\vert Y_n - \hat{Y}_n^{\Delta} \right\vert^2\right] + \mathbb{E}\left[\left\vert Z_n - \hat{Z}_n^{\Delta} \right\vert^2\right]\\
    & \leq \left( 1 +  C |\Delta| \right) \mathbb{E}\left[ \left\vert Y_{n+1} - Y_{n+1}^{\Delta, \hat{\theta}} \right\vert^2\right] + \left( 1 +  C |\Delta| \right)\mathbb{E}\left[ \left\vert D_n Y_{n+1} - D_n Y_{n+1}^{\Delta, \hat{\theta}} \right\vert^2\right]+ C|\Delta|^2\\
     & \quad + C\mathbb{E}\left[\int_{t_{n}}^{t_{n+1}} \left\vert D_n Z_s - \hat{DZ}_n \right\vert^2 ds\right]+ C |\Delta| \mathbb{E}\left[ \int_{t_{n}}^{t_{n+1}} \left\vert f_D\left(s, \mathbf{X}_s, \mathbf{D}_n \mathbf{X}_s\right) \right\vert^2 ds \right].
\end{aligned}
\label{eq38}
\end{equation}
Using the Young inequality in the form
\begin{equation}
  \left(1 - \nu \right)c_1^2 - \frac{1}{\nu} c_2^2 \leq \left(1 - \nu \right)c_1^2 + \left( 1 - \frac{1}{\nu} \right) c_2^2 \leq \left( c_1 + c_2\right)^2,\quad \nu>0,
  \label{eq39}
\end{equation}
we have for $\nu = |\Delta|$ that
\begin{align*}
    \left(1 - |\Delta| \right) \mathbb{E}\left[\left\vert Y_n - Y_n^{\Delta, \hat{\theta}} \right\vert^2 \right] - \frac{1}{|\Delta|} \mathbb{E}\left[\left\vert Y_n^{\Delta, \hat{\theta}} - \hat{Y}_n^{\Delta} \right\vert^2 \right] & \leq \mathbb{E}\left[\left\vert Y_n - \hat{Y}_n^{\Delta} \right\vert^2 \right],\\
    \left(1 - |\Delta| \right) \mathbb{E}\left[\left\vert Z_n - Z_n^{\Delta, \hat{\theta}} \right\vert^2 \right] - \frac{1}{|\Delta|} \mathbb{E}\left[\left\vert Z_n^{\Delta, \hat{\theta}} - \hat{Z}_n^{\Delta} \right\vert^2 \right] & \leq \mathbb{E}\left[\left\vert Z_n - \hat{Z}_n^{\Delta} \right\vert^2 \right].    
\end{align*}
Then, for small enough $|\Delta|$,~\eqref{eq38} becomes
\begin{align*}
    & \mathbb{E}\left[\left\vert Y_n - Y_n^{\Delta, \hat{\theta}} \right\vert^2\right] + \mathbb{E}\left[\left\vert Z_n - Z_n^{\Delta, \hat{\theta}} \right\vert^2\right]\\
    & \leq \left( 1 +  C|\Delta| \right) \mathbb{E}\left[ \left\vert Y_{n+1} - Y_{n+1}^{\Delta, \hat{\theta}} \right\vert^2\right] + \left( 1 +  C |\Delta| \right)\mathbb{E}\left[ \left\vert D_n Y_{n+1} - D_n Y_{n+1}^{\Delta, \hat{\theta}} \right\vert^2\right]+ C |\Delta|^2\\
    &\quad + C \mathbb{E}\left[\int_{t_{n}}^{t_{n+1}} \left\vert D_n Z_s - \hat{DZ}_n \right\vert^2 ds\right]+ C |\Delta|\mathbb{E}\left[ \int_{t_{n}}^{t_{n+1}} \left\vert f_D\left(s, \mathbf{X}_s, \mathbf{D}_n \mathbf{X}_s\right) \right\vert^2 ds \right]\\
     & \quad + C N \mathbb{E}\left[\left\vert \hat{Y}_n^{\Delta} - Y_n^{\Delta, \hat{\theta}} \right\vert^2 \right] + C N \mathbb{E}\left[\left\vert \hat{Z}_n^{\Delta} - Z_n^{\Delta, \hat{\theta}} \right\vert^2 \right].
\end{align*}
Using the discrete Grönwall lemma in the above equation, we get
\begin{align*}    
    & \max_{0 \leq n \leq N}\mathbb{E}\left[\left\vert Y_n - Y_n^{\Delta, \hat{\theta}} \right\vert^2\right] + \max_{0 \leq n \leq N} \mathbb{E}\left[\left\vert Z_n - Z_n^{\Delta, \hat{\theta}} \right\vert^2\right]\\
    & \leq C\mathbb{E}\left[ \left\vert g(X_T) - g\left(X_N^{\Delta}\right) \right\vert^2\right] + C \mathbb{E}\left[ \left\vert \nabla_x g(X_T) D_{N-1} X_N - \nabla_x g\left(X_N^{\Delta}\right) D_{N-1} X_N^{\Delta} \right\vert^2\right]\\
     & \quad + C|\Delta| + C\sum_{n=0}^{N-1} \mathbb{E}\left[\int_{t_{n}}^{t_{n+1}} \left\vert D_n Z_s - \hat{DZ}_n \right\vert^2 ds\right]\\
    &\quad + C |\Delta| \sum_{n=0}^{N-1} \mathbb{E}\left[ \int_{t_{n}}^{t_{n+1}} \left\vert f_D\left(s, \mathbf{X}_s, \mathbf{D}_n \mathbf{X}_s\right) \right\vert^2 ds \right]+ C N \sum_{n=0}^{N-1}\mathbb{E}\left[\left\vert \hat{Y}_n^{\Delta} - Y_n^{\Delta, \hat{\theta}} \right\vert^2 \right]\\
     & \quad  + C N \sum_{n=0}^{N-1} \mathbb{E}\left[\left\vert \hat{Z}_n^{\Delta} - Z_n^{\Delta, \hat{\theta}} \right\vert^2 \right].
\end{align*}
The second term in the above inequality can be written as
\begin{align*}    
    & \mathbb{E}\left[ \left\vert \nabla_x g(X_T) D_{N-1} X_N - \nabla_x g\left(X_N^{\Delta}\right) D_{N-1} X_N^{\Delta} \right\vert^2\right]\\
    & = \mathbb{E}\left[ \left\vert \left(\nabla_x g(X_T) - \nabla_x g\left(X_N^{\Delta}\right)\right)D_{N-1} X_N + \nabla_x g\left(X_N^{\Delta}\right)\left( D_{N-1} X_N - D_{N-1} X_N^{\Delta}\right)\right\vert^2\right],\\
    &\leq 2 \mathbb{E}\left[ \left\vert \nabla_x g(X_T) - \nabla_x g\left(X_N^{\Delta}\right)\right\vert^2 \left\vert D_{N-1} X_N\right\vert^2\right] + 2 \mathbb{E}\left[\left\vert \nabla_x g\left(X_N^{\Delta}\right)\right\vert^2\left \vert D_{N-1} X_N - D_{N-1} X_N^{\Delta}\right\vert^2\right],\\
    &\leq C \mathbb{E}\left[ \left\vert \nabla_x g(X_T) - \nabla_x g\left(X_N^{\Delta}\right)\right\vert^2 \right] + C \mathbb{E}\left[\left \vert D_{N-1} X_N - D_{N-1} X_N^{\Delta}\right\vert^2\right]
\end{align*}
where we used the submultiplicative property of the Frobenius norm and $\left( c_1 + c_2 \right)^2 \leq 2\left( c_1^2 + c_2^2 \right)$ for the first inequality, the boundedness of $DX$ and $\nabla_x g(X)$ for the second inequality. Therefore, we get from~\eqref{eq24} that 
\begin{equation}
    \begin{aligned}
    \mathbb{E}\left[ \left\vert \nabla_x g(X_T) D_{N-1} X_N - \nabla_x g\left(X_N^{\Delta}\right) D_{N-1} X_N^{\Delta} \right\vert^2\right] \leq C \mathbb{E}\left[ \left\vert \nabla_x g(X_T) - \nabla_x g\left(X_N^{\Delta}\right)\right\vert^2\right] + C |\Delta|.
    \end{aligned}
\label{eq40}
\end{equation}       
From~\eqref{eq22}, the $\mathbb{L}^2$-regularity of $DZ$~\eqref{eq25} and the inequality~\eqref{eq40}, we have
\begin{equation}    
\begin{aligned}    
    & \max_{0 \leq n \leq N}\mathbb{E}\left[\left\vert Y_n - Y_n^{\Delta, \hat{\theta}} \right\vert^2\right] + \max_{0 \leq n \leq N} \mathbb{E}\left[\left\vert Z_n - Z_n^{\Delta, \hat{\theta}} \right\vert^2\right]\\
    & \leq C \mathbb{E}\left[ \left\vert g(X_T) - g\left(X_N^{\Delta}\right) \right\vert^2\right] + C \mathbb{E}\left[ \left\vert \nabla_x g(X_T) - \nabla_x g\left(X_N^{\Delta}\right) \right\vert^2\right] +  C |\Delta| +  C \varepsilon^{DZ}\left( |\Delta| \right)\\
     & \quad 
     + C N \sum_{n=0}^{N-1} \mathbb{E}\left[\left\vert \hat{Y}_n^{\Delta} - Y_n^{\Delta, \hat{\theta}} \right\vert^2 \right] +  C N \sum_{n=0}^{N-1} \mathbb{E}\left[\left\vert \hat{Z}_n^{\Delta} - Z_n^{\Delta, \hat{\theta}} \right\vert^2 \right].
\end{aligned}
\label{eq41}
\end{equation}

{\it Step 4.} Let us fix $n \in \{0, 1, \ldots, N-1\}$. By using relations~\eqref{eq28} and~\eqref{eq29}, and recalling the definition of $\left( \hat{Z}_n^{\Delta}, D_n \hat{Z}_n^{\Delta}\right)$ as an $\mathbb{L}^2-$projection of $\left(\hat{Z}_t, D_n \hat{Z}_t\right)$, we can rewrite the loss function~\eqref{eq19} as
\begin{equation}    
    \begin{aligned}
        \mathbf{L}_n^{\Delta}\left( \theta_n \right) & = \omega_1\left(\hat{\mathbf{L}}^{y,\Delta}_n\left( \theta_n\right) + \mathbb{E}\left[ \int_{t_n}^{t_{n+1}} \left \vert \hat{Z}_s - \hat{Z}_n^{\Delta} \right \vert^2 ds \right]\right)\\
        &\quad + \omega_2\left( 
        \hat{\mathbf{L}}^{z,\Delta}_n\left( \theta_n\right) + \mathbb{E}\left[ \int_{t_n}^{t_{n+1}} \left \vert D_n \hat{Z}_s - D_n \hat{Z}_n^{\Delta} \right \vert^2 ds \right]\right),  
    \end{aligned}
\label{eq42}
\end{equation}
where
\begin{equation}    
    \begin{aligned}
        \hat{\mathbf{L}}^{y,\Delta}_n\left( \theta_n\right) & := \mathbb{E}\left[  \left \vert \hat{Y}_n^{\Delta} - \phi^y_n\left( X_n^{\Delta} ; \theta^y_n\right) + \left( f\left(t_n, \mathbf{X}_n^{\Delta, \theta}\right) - f\left(t_n, \hat{\mathbf{X}}_n^{\Delta}\right)\right)\Delta t_n \right \vert^2 \right]\\
        & \quad + \Delta t_n \mathbb{E}\left[  \left \vert \hat{Z}_n^{\Delta} - \phi^z_n\left( X_n^{\Delta} ; \theta^z_n\right) \right \vert^2 \right],
    \end{aligned}
\label{eq43}
\end{equation}
and 
\begin{equation}    
    \begin{aligned}
        \hat{\mathbf{L}}^{z,\Delta}_n\left( \theta_n\right) & := \mathbb{E}\left[  \left \vert \hat{Z}_n^{\Delta} - \phi^z_n\left( X_n^{\Delta} ; \theta^z_n\right) + \left( f_D\left(t_n, \mathbf{X}_n^{\Delta, \theta}, \mathbf{D}_n \mathbf{X}_n^{\Delta, \theta} \right) \right. \right. \right. \\ 
        &  \qquad \left. \left. \left. -  f_D\left(t_n, \hat{\mathbf{X}}_n^{\Delta}, \mathbf{D}_n \hat{\mathbf{X}}_n^{\Delta}\right)\right)\Delta t_n \right \vert^2 \right]\\
        & \quad + \Delta t_n \mathbb{E}\left[  \left \vert \left( \hat{\Gamma}_n^{\Delta} - \phi^{\gamma}_n\left( X_n^{\Delta} ; \theta^{\gamma}_n\right)\right) b\left(t_n, X_n^{\Delta}\right) \right \vert^2 \right].
    \end{aligned}
\label{eq44}
\end{equation}
By using the Young inequality~\eqref{eq31}, the Lipschitz and Hölder continuity of $f$~\eqref{eq20} and the inequality $\left( c_1 + c_2 \right)^2 \leq 2 \left( c_1^2 + c_2^2\right)$, we have that~\eqref{eq43} is bounded by
\begin{equation}    
    \hat{\mathbf{L}}^{y,\Delta}_n\left( \theta_n\right) \leq \left(1+C\Delta t_n \right) \mathbb{E}\left[  \left \vert \hat{Y}_n^{\Delta} - \phi^y_n\left( X_n^{\Delta} ; \theta^y_n\right) \right \vert^2 \right] + C \Delta t_n \mathbb{E}\left[  \left \vert \hat{Z}_n^{\Delta} - \phi^z_n\left( X_n^{\Delta} ; \theta^z_n\right) \right \vert^2 \right].
\label{eq45}
\end{equation}
By performing similar steps for~\eqref{eq44}, we get 
\begin{equation}
    \begin{aligned}        
    \hat{\mathbf{L}}^{z,\Delta}_n\left( \theta_n\right) & \leq C\Delta t_n \mathbb{E}\left[  \left \vert \hat{Y}_n^{\Delta} - \phi^y_n\left( X_n^{\Delta} ; \theta^y_n\right) \right \vert^2 \right] + \left(1+C\Delta t_n \right) \mathbb{E}\left[  \left \vert \hat{Z}_n^{\Delta} - \phi^z_n\left( X_n^{\Delta} ; \theta^z_n\right) \right \vert^2 \right]\\
    & \quad + C \Delta t_n \mathbb{E}\left[  \left \vert \left(\hat{\Gamma}_n^{\Delta} - \phi^{\gamma}_n\left( X_n^{\Delta} ; \theta^{\gamma}_n\right)\right) b(t_n, X_n^{\Delta}) \right \vert^2 \right].
    \end{aligned}
\label{eq46}
\end{equation}
We define $\hat{\mathbf{L}}_n^{\Delta} := \omega_1\hat{\mathbf{L}}^{y,\Delta}_n + \omega_2\hat{\mathbf{L}}^{z,\Delta}_n.$ Using the bounds in~\eqref{eq45} and~\eqref{eq46} yields
\begin{equation}
    \begin{aligned}        
    \hat{\mathbf{L}}_n^{\Delta}\left( \theta_n\right) & \leq \omega_1\left(1+C\Delta t_n \right) \mathbb{E}\left[  \left \vert \hat{Y}_n^{\Delta} - \phi^y_n\left( X_n^{\Delta} ; \theta^y_n\right) \right \vert^2 \right] + \omega_2C\Delta t_n  \mathbb{E}\left[  \left \vert \hat{Y}_n^{\Delta} - \phi^y_n\left( X_n^{\Delta} ; \theta^y_n\right) \right \vert^2 \right]\\
    &\quad+ \omega_2\left(1+C\Delta t_n \right)\mathbb{E}\left[  \left \vert \hat{Z}_n^{\Delta} - \phi^z_n\left( X_n^{\Delta} ; \theta^z_n\right) \right \vert^2 \right] + \omega_1C\Delta t_n \mathbb{E}\left[  \left \vert \hat{Z}_n^{\Delta} - \phi^z_n\left( X_n^{\Delta} ; \theta^z_n\right) \right \vert^2 \right]\\
    & \quad + \omega_2 C \Delta t_n \mathbb{E}\left[  \left \vert \left(\hat{\Gamma}_n^{\Delta} - \phi^{\gamma}_n\left( X_n^{\Delta} ; \theta^{\gamma}_n\right)\right) b(t_n, X_n^{\Delta}) \right \vert^2 \right].
    \end{aligned}
\label{eq47}
\end{equation}
Similarly as above, we get
\begin{equation}    
    \hat{\mathbf{L}}^{y,\Delta}_n\left( \theta_n\right) \geq \left(1- C \Delta t_n \right) \mathbb{E}\left[  \left \vert \hat{Y}_n^{\Delta} - \phi^y_n\left( X_n^{\Delta} ; \theta^y_n\right) \right \vert^2 \right]  + \frac{\Delta t_n}{2} \mathbb{E}\left[  \left \vert \hat{Z}_n^{\Delta} - \phi^z_n\left( X_n^{\Delta} ; \theta^z_n\right) \right \vert^2 \right],
\label{eq48}
\end{equation}
and
\begin{equation}    
    \begin{aligned}        
    \hat{\mathbf{L}}^{z,\Delta}_n\left( \theta_n\right) & \geq \left(1- C \Delta t_n \right) \mathbb{E}\left[  \left \vert \hat{Z}_n^{\Delta} - \phi^z_n\left( X_n^{\Delta} ; \theta^z_n\right) \right \vert^2 \right] - \frac{\Delta t_n}{2}\mathbb{E}\left[  \left \vert \hat{Y}_n^{\Delta} - \phi^y_n\left( X_n^{\Delta} ; \theta^y_n\right) \right \vert^2 \right]\\
    &\quad + \frac{\Delta t_n}{2} \mathbb{E}\left[  \left \vert \left(\hat{\Gamma}_n^{\Delta} - \phi^{\gamma}_n\left( X_n^{\Delta} ; \theta^{\gamma}_n\right)\right) b(t_n, X_n^{\Delta}) \right \vert^2 \right].
    \end{aligned}
\label{eq49}
\end{equation}
By using the bounds in~\eqref{eq48} and~\eqref{eq49}, we have for $\hat{\mathbf{L}}_n^{\Delta}$ that
\begin{equation}
    \begin{aligned}        
    \hat{\mathbf{L}}_n^{\Delta}\left( \theta_n\right) & \geq \omega_1\left(1- C \Delta t_n \right) \mathbb{E}\left[  \left \vert \hat{Y}_n^{\Delta} - \phi^y_n\left( X_n^{\Delta} ; \theta^y_n\right) \right \vert^2 \right] - \frac{ \omega_2}{2}\Delta t_n\mathbb{E}\left[  \left \vert \hat{Y}_n^{\Delta} - \phi^y_n\left( X_n^{\Delta} ; \theta^y_n\right) \right \vert^2 \right]\\
    &\quad + \omega_2\left(1 - C \Delta t_n\right)\mathbb{E}\left[  \left \vert \hat{Z}_n^{\Delta} - \phi^z_n\left( X_n^{\Delta} ; \theta^z_n\right) \right \vert^2 \right]+ \frac{\omega_1}{2}\Delta t_n\mathbb{E}\left[  \left \vert \hat{Z}_n^{\Delta} - \phi^z_n\left( X_n^{\Delta} ; \theta^z_n\right) \right \vert^2 \right]\\
    & \quad + \omega_2 \frac{\Delta t_n}{2} \mathbb{E}\left[  \left \vert \left(\hat{\Gamma}_n^{\Delta} - \phi^{\gamma}_n\left( X_n^{\Delta} ; \theta^{\gamma}_n\right)\right) b(t_n, X_n^{\Delta}) \right \vert^2 \right],\\
    & \geq \omega_1\left(1- C \Delta t_n \right) \mathbb{E}\left[  \left \vert \hat{Y}_n^{\Delta} - \phi^y_n\left( X_n^{\Delta} ; \theta^y_n\right) \right \vert^2 \right] + \omega_2\left(1- C \Delta t_n \right) \mathbb{E}\left[  \left \vert \hat{Z}_n^{\Delta} - \phi^z_n\left( X_n^{\Delta} ; \theta^z_n\right) \right \vert^2 \right] \\
    & \quad + \omega_2 \frac{\Delta t_n}{2} \mathbb{E}\left[  \left \vert \left(\hat{\Gamma}_n^{\Delta} - \phi^{\gamma}_n\left( X_n^{\Delta} ; \theta^{\gamma}_n\right)\right) b(t_n, X_n^{\Delta}) \right \vert^2 \right].
    \end{aligned}
\label{eq50}
\end{equation}

{\it Step 5.} Let us fix $n \in \{0, 1, \ldots, N-1\}$. We assume that $\hat{\theta}_n$ is a perfect approximation of the optimal parameters $\theta_n^{\star} = \left( \theta_n^{y, \star}, \theta_n^{z, \star}, \theta_n^{\gamma, \star} \right) \in \argmin_{\theta_n \in \Theta_n} \mathbf{L}_n^{\Delta}\left(\theta_n\right)$ so that $Y_n^{\Delta, \hat{\theta}} = \phi^y_n\left(\cdot; \theta_n^{y, \star} \right)$, $Z_n^{\Delta, \hat{\theta}} = \phi^z_n\left(\cdot; \theta_n^{z, \star} \right)$ and $\Gamma_n^{\Delta, \hat{\theta}} = \phi^{\gamma}_n\left(\cdot; \theta_n^{\gamma, \star} \right)$. In other words, we assume that the SGD method is not trapped in a local minimum, and we neglect the estimation error resulting from minimizing an empirical loss function. We also have that $\theta_n^{\star} \in \argmin_{\theta_n \in \Theta_n} \hat{\mathbf{L}}_n^{\Delta}\left(\theta_n\right)$ from~\eqref{eq42}. Hence, $\hat{\mathbf{L}}_n^{\Delta}\left(\theta_n^{\star}\right) \leq \hat{\mathbf{L}}_n^{\Delta}\left(\theta_n\right)$ for any $\theta_n$. By using
 the upper bound~\eqref{eq47} and the lower bound~\eqref{eq50} of $\hat{\mathbf{L}}_n^{\Delta}\left(\theta_n\right)$, we then have for all $\theta_n$
\begin{align*}
    &\omega_1 \left(1- C \Delta t_n\right) \mathbb{E}\left[  \left \vert \hat{Y}_n^{\Delta} - Y_n^{\Delta, \hat{\theta}} \right \vert^2 \right] +\omega_2 \left(1-C \Delta t_n \right)\mathbb{E}\left[  \left \vert \hat{Z}_n^{\Delta} - Z_n^{\Delta, \hat{\theta}} \right \vert^2 \right] \\
    &+\omega_2 \frac{\Delta t_n}{2} \mathbb{E}\left[  \left \vert \left(\hat{\Gamma}_n^{\Delta} - \Gamma_n^{\Delta, \hat{\theta}}\right) b(t_n, X_n^{\Delta}) \right \vert^2 \right] \leq \hat{\mathbf{L}}_n^{\Delta}\left( \theta_n^{\star}\right) \leq \hat{\mathbf{L}}_n^{\Delta}\left( \theta_n\right) \leq \\ 
    &\omega_1\left(1+C\Delta t_n \right) \mathbb{E}\left[  \left \vert \hat{Y}_n^{\Delta} - \phi^y_n\left( X_n^{\Delta} ; \theta^y_n\right) \right \vert^2 \right] + \omega_2 C\Delta t_n\mathbb{E}\left[  \left \vert \hat{Y}_n^{\Delta} - \phi^y_n\left( X_n^{\Delta} ; \theta^y_n\right) \right \vert^2 \right]\\
    &+\omega_2\left(1+C\Delta t_n \right)\mathbb{E}\left[  \left \vert \hat{Z}_n^{\Delta} - \phi^z_n\left( X_n^{\Delta} ; \theta^z_n\right) \right \vert^2 \right] + \omega_1C\Delta t_n\mathbb{E}\left[  \left \vert \hat{Z}_n^{\Delta} - \phi^z_n\left( X_n^{\Delta} ; \theta^z_n\right) \right \vert^2 \right]\\
    &+\omega_2 C \Delta t_n \mathbb{E}\left[  \left \vert \left(\hat{\Gamma}_n^{\Delta} - \phi^{\gamma}_n\left( X_n^{\Delta} ; \theta^{\gamma}_n\right)\right) b(t_n, X_n^{\Delta}) \right \vert^2 \right].   
\end{align*}
For $\Delta t_n$ sufficiently small satisfying $C \Delta t_n \leq \frac{1}{2}$ and using~\eqref{eq27}, we have
\begin{equation}    
\begin{aligned}
    & \omega_1 \mathbb{E}\left[  \left \vert \hat{Y}_n^{\Delta} - Y_n^{\Delta, \hat{\theta}} \right \vert^2 \right] + \omega_2\mathbb{E}\left[  \left \vert \hat{Z}_n^{\Delta} - Z_n^{\Delta, \hat{\theta}} \right \vert^2 \right] + \omega_2 \Delta t_n \mathbb{E}\left[  \left \vert \left(\hat{\Gamma}_n^{\Delta} - \Gamma_n^{\Delta, \hat{\theta}}\right) b(t_n, X_n^{\Delta}) \right \vert^2 \right]\\
    & \leq \omega_1 C  \varepsilon_n^{y}+ \omega_2 C \Delta t_n  \varepsilon_n^{y}+ \omega_2 C  \varepsilon_n^{z} + \omega_1 C \Delta t_n  \varepsilon_n^{z} + \omega_2 C \Delta t_n  \varepsilon_n^{\gamma}.
\end{aligned}
\label{eq51}
\end{equation}
After inserting the last inequality into~\eqref{eq41}, we obtain
\begin{equation}
    \begin{aligned}
        & \max_{0 \leq n \leq N}\mathbb{E}\left[\left\vert Y_n - Y_n^{\Delta, \hat{\theta}} \right\vert^2\right] + \max_{0 \leq n \leq N} \mathbb{E}\left[\left\vert Z_n - Z_n^{\Delta, \hat{\theta}} \right\vert^2\right]\\
        & \leq C \mathbb{E}\left[ \left\vert g(X_T) - g\left(X_N^{\Delta}\right) \right\vert^2\right] + C \mathbb{E}\left[ \left\vert \nabla_x g(X_T) - \nabla_x g\left(X_N^{\Delta}\right) \right\vert^2\right] +  C |\Delta| +  C \varepsilon^{DZ}\left( |\Delta| \right)\\
        & \quad 
        + C N \omega_1 \sum_{n=0}^{N-1} \varepsilon_n^{y} + C \omega_2 \sum_{n=0}^{N-1} \varepsilon_n^{y} + C N \omega_2 \sum_{n=0}^{N-1} \varepsilon_n^{z} + C \omega_1 \sum_{n=0}^{N-1} \varepsilon_n^{z} + C \omega_2 \sum_{n=0}^{N-1} \varepsilon_n^{\gamma}.
    \end{aligned}
\label{eq52}
\end{equation}
This completes the proof of the consistency of processes $Y$ and $Z$ in Theorem~\ref{theorem6}.

{\it Step 6.} It now remains to prove the consistency for the process $\Gamma$. Let us fix $n \in \{0, 1, \ldots, N-1\}$. Using~\eqref{eq36} in~\eqref{eq34}, we get
\begin{align*}
   &\mathbb{E}\left[\int_{t_{n}}^{t_{n+1}} \left\vert D_n Z_s - D_n \hat{Z}_n^{\Delta} \right\vert^2 ds\right]\\
   &\leq \mathbb{E}\left[\int_{t_{n}}^{t_{n+1}} \left\vert D_n Z_s - \hat{DZ}_n \right\vert^2 ds\right]+ 2d\left( \mathbb{E}\left[ \left\vert D_n Y_{n+1} - D_n Y_{n+1}^{\Delta, \hat{\theta}} \right\vert^2\right] \right.\\ 
   & \quad \left.- \mathbb{E}\left[\left\vert \mathbb{E}_n\left[   D_n Y_{n+1} - D_n Y_{n+1}^{\Delta, \hat{\theta}} \right] \right\vert^2 \right]  \right)+ 2d \Delta t_n \mathbb{E}\left[\int_{t_{n}}^{t_{n+1}} \left\vert f_D\left(s, \mathbf{X}_s, \mathbf{D}_n \mathbf{X}_s\right) \right\vert^2 ds\right].
\end{align*}
Summing from $n=0,\ldots,N-1$, using~\eqref{eq22} and~\eqref{eq25} gives
\begin{equation}    
\begin{aligned}
   &\mathbb{E}\left[\sum_{n=0}^{N-1} \int_{t_{n}}^{t_{n+1}} \left\vert D_n Z_s - D_n \hat{Z}_n^{\Delta} \right\vert^2 ds\right]\\
   & \leq \varepsilon^{DZ}\left( |\Delta| \right) + C |\Delta| +  2d\mathbb{E}\left[ \left\vert D_{N-1} Y_{N} - D_{N-1} Y_{N}^{\Delta, \hat{\theta}} \right\vert^2\right]\\
   &\quad + 2d\sum_{n=1}^{N-1}\left( \mathbb{E}\left[ \left\vert D_{n-1} Y_{n} - D_{n-1} Y_{n}^{\Delta, \hat{\theta}} \right\vert^2\right] - \mathbb{E}\left[\left\vert \mathbb{E}_n\left[   D_n Y_{n+1} - D_n Y_{n+1}^{\Delta, \hat{\theta}} \right] \right\vert^2 \right]  \right),\\
   &\leq \varepsilon^{DZ}\left( |\Delta| \right) + C |\Delta| +  C\mathbb{E}\left[ \left\vert\nabla_x g\left(X_{N}\right) - \nabla_x g\left(X_{N}^{\Delta}\right) \right\vert^2\right]\\
   &\quad + 2d\sum_{n=1}^{N-1}\left( \mathbb{E}\left[ \left\vert D_{n-1} Y_{n} - D_{n-1} Y_{n}^{\Delta, \hat{\theta}} \right\vert^2\right] - \mathbb{E}\left[\left\vert \mathbb{E}_n\left[    D_n Y_{n+1} - D_n Y_{n+1}^{\Delta, \hat{\theta}} \right] \right\vert^2 \right]  \right),\\
\end{aligned}
\label{eq53}
\end{equation}
where the summation index is changed for the last summation in the first inequality, and~\eqref{eq40} is used in the second inequality. 

Using similar steps as for~\eqref{eq40}, we have that
\begin{align*}
    & \mathbb{E}\left[ \left\vert D_{n-1} Y_{n} - D_{n-1} Y_{n}^{\Delta, \hat{\theta}} \right\vert^2\right] - \mathbb{E}\left[\left\vert \mathbb{E}_n\left[   D_{n} Y_{n+1} - D_{n} Y_{n+1}^{\Delta, \hat{\theta}} \right] \right\vert^2 \right]\\
    & \leq C \mathbb{E}\left[ \left\vert Z_n - Z_n^{\Delta, \hat{\theta}} \right\vert^2\right] + C | \Delta | - \mathbb{E}\left[\left\vert \mathbb{E}_n\left[ D_n Y_{n+1} - D_n Y_{n+1}^{\Delta, \hat{\theta}}\right] \right\vert^2 \right].
\end{align*}
Moreover, using
\begin{equation*}    
    \left(1 - |\Delta| \right) \mathbb{E}\left[\left\vert Z_n - Z_n^{\Delta, \hat{\theta}} \right\vert^2 \right] - \frac{1}{|\Delta|} \mathbb{E}\left[\left\vert Z_n^{\Delta, \hat{\theta}} - \hat{Z}_n^{\Delta}\right\vert^2 \right] \leq \mathbb{E}\left[\left\vert Z_n - \hat{Z}_n^{\Delta} \right\vert^2 \right]
\end{equation*}
and~\eqref{eq37}, we have for $|\Delta|$ small enough:
\begin{align*}
    & \mathbb{E}\left[ \left\vert D_{n-1} Y_{n} - D_{n-1} Y_{n}^{\Delta, \hat{\theta}} \right\vert^2\right] - \mathbb{E}\left[\left\vert \mathbb{E}_n\left[   D_{n} Y_{n+1} - D_{n} Y_{n+1}^{\Delta, \hat{\theta}} \right] \right\vert^2 \right]\\
    &\leq C \mathbb{E}\left[ \left\vert D_n Y_{n+1} - D_n Y_{n+1}^{\Delta, \hat{\theta}} \right\vert^2\right] + C |\Delta| + C |\Delta| \mathbb{E}\left[ \left\vert Y_n - \hat{Y}_n^{\Delta} \right\vert^2 \right] +  C |\Delta| \mathbb{E}\left[ \left\vert Z_n - \hat{Z}_n^{\Delta} \right\vert^2 \right]\\
    & \quad + C\mathbb{E}\left[\int_{t_{n}}^{t_{n+1}} \left\vert D_n Z_s - \hat{DZ}_n \right\vert^2 ds\right] + C |\Delta|\mathbb{E}\left[ \int_{t_{n}}^{t_{n+1}} \left\vert f_D\left(s, \mathbf{X}_s, \mathbf{D}_n \mathbf{X}_s\right) \right\vert^2 ds \right]\\
    & \quad + C N \mathbb{E}\left[\left\vert \hat{Z}_n^{\Delta} - Z_n^{\Delta, \hat{\theta}} \right\vert^2 \right].
\end{align*}
Hence,~\eqref{eq53} becomes 
\begin{align*}
   &\mathbb{E}\left[\sum_{n=0}^{N-1} \int_{t_{n}}^{t_{n+1}} \left\vert D_n Z_s - D_n \hat{Z}_n^{\Delta} \right\vert^2 ds\right]\\
   & \leq C \varepsilon^{DZ}\left( |\Delta| \right) + C |\Delta| +  C\mathbb{E}\left[ \left\vert\nabla_x g\left(X_{N}\right) - \nabla_x g\left(X_{N}^{\Delta}\right) \right\vert^2\right]\\
   & \quad + C \sum_{n=1}^{N-1} \mathbb{E}\left[ \left\vert D_n Y_{n+1} - D_n Y_{n+1}^{\Delta, \hat{\theta}} \right\vert^2\right] + C |\Delta| \sum_{n=1}^{N-1}\mathbb{E}\left[ \left\vert Y_n - \hat{Y}_n^{\Delta} \right\vert^2 \right]\\
   & \quad + C |\Delta| \sum_{n=1}^{N-1}\mathbb{E}\left[ \left\vert Z_n - \hat{Z}_n^{\Delta} \right\vert^2 \right] + C N \sum_{n=1}^{N-1} \mathbb{E}\left[\left\vert \hat{Z}_n^{\Delta} - Z_n^{\Delta, \hat{\theta}} \right\vert^2 \right],
\end{align*}
where we used~\eqref{eq22} and~\eqref{eq25}. From~\eqref{eq38} and~\eqref{eq41}, we have that
\begin{equation}    
\begin{aligned}
   &\mathbb{E}\left[\sum_{n=0}^{N-1} \int_{t_{n}}^{t_{n+1}} \left\vert D_n Z_s - D_n \hat{Z}_n^{\Delta} \right\vert^2 ds\right]\\
   & \leq C \mathbb{E}\left[ \left\vert g(X_T) - g\left(X_N^{\Delta}\right) \right\vert^2\right] + C \mathbb{E}\left[ \left\vert \nabla_x g(X_T) - \nabla_x g\left(X_N^{\Delta}\right) \right\vert^2\right] +  C |\Delta| +  C \varepsilon^{DZ}\left( |\Delta| \right)\\
     & \quad 
     + C N \sum_{n=0}^{N-1} \mathbb{E}\left[\left\vert \hat{Y}_n^{\Delta} - Y_n^{\Delta, \hat{\theta}} \right\vert^2 \right] +  C N \sum_{n=0}^{N-1} \mathbb{E}\left[\left\vert \hat{Z}_n^{\Delta} - Z_n^{\Delta, \hat{\theta}} \right\vert^2 \right].
\end{aligned}
\label{eq54}
\end{equation}
Finally, using
\begin{align*}
    \mathbb{E}\left[ \int_{t_{n}}^{t_{n+1}} \left\vert D_n Z_s - D_n Z_n^{\Delta, \hat{\theta}} \right\vert^2 ds\right]  & \leq 2 \mathbb{E}\left[ \int_{t_{n}}^{t_{n+1}} \left\vert D_n Z_s - D_n \hat{Z}_n^{\Delta} \right\vert^2 ds\right]\\
    & \quad + 2 \Delta t_n \mathbb{E}\left[ \left\vert D_n \hat{Z}_n^{\Delta} - D_n Z_n^{\Delta, \hat{\theta}} \right\vert^2 ds\right], 
\end{align*}
summing over $n=0,\ldots,N-1$ and applying the inequalities~\eqref{eq51} and~\eqref{eq54}, we derive the proof of the consistency of process $\Gamma$ as expressed in~\eqref{eq52}. This concludes the proof of Theorem~\ref{theorem6}.
\end{proof}

According to Theorem~\ref{theorem6}, the total approximation error of the DLBDP scheme consists of six terms. The first two terms correspond to the strong approximation of the terminal condition and its gradient, depending on the Euler-Maruyama scheme and the functions $\left(g(x), \nabla_x g(x)\right)$. The third term represents the strong approximation of the Euler-Maruyama scheme and the path regularity of the processes $\left(Y, Z\right)$, see Theorem~\ref{theorem5}. The fourth term represents the $\mathbb{L}^2$-regularity of $DZ$. All the afformentioned terms converge to zero as $|\Delta|$ goes to zero, with a rate of $|\Delta|$ when Assumptions~\ref{AX4} and~\ref{AY4} are satisfied. For the last two terms, the better the DNNs are able to estimate the functions~\eqref{eq27}, the smaller is their contribution in the total approximation error. Note that from the UAT~\cite{hornik1989multilayer,cybenko1989approximation}, the approximation error from the DNNs can be made arbitrarily small for a sufficiently large number of hidden neurons. It is crucial noting that, in contrast to both the DBDP scheme and the method outlined in~\cite{negyesi2024one}, the DLBDP scheme provides a means to manage the impact of the DNN's approximation error. This is accomplished by selecting the values of $\omega_1$ and $\omega_2$, resulting in improved accuracy for the processes $\left(Y, Z, \Gamma\right)$, as we demonstrate in the next section.

\section{Numerical results}
\label{sec6}
In this section, we illustrate the improved performance of the DLBDP scheme compared to the DBDP scheme not only when approximating the solution, but also its gradient and the Hessian matrix. As high-accurate gradient approximations are of great importance in finance, we consider linear and nonlinear option pricing examples. All the experiments below were run in PYTHON using TensorFlow on the PLEIADES cluster (no parallelization), which consists of 268 workernodes and additionally 5 GPU nodes with 8 NVidia HGX A100 GPUs (128 cores each, 2 TB memory, 16 GB per thread). We run the algorithms on the GPU nodes. For more information, see~PLEIADES documentation\footnote{\url{https://pleiadesbuw.github.io/PleiadesUserDocumentation/}}.

\subsection{Experimental setup}
\label{subsec61}
In all the following examples, we consider the same hyperparameters for our scheme and the DBDP scheme for a fair comparison. For the DNNs, we choose $L = 2$ hidden layers and $\eta= 100 + d$ neurons per hidden layer. The input is normalized based on the true moments. The input is not normalized at discrete time point $t_0,$ as the standard deviation is zero. A hyperbolic tangent activation $\tanh(\cdot)$ is applied on each hidden layer. It is crucial to mention that one can't apply batch normalization for the hidden layers as AD is required to approximate the process $\Gamma$ in the DBDP scheme. This is because using batch normalization creates dependence for the gradients in the batch, since it normalizes across the batch dimension. Using the method tf.GradientTape.batch\_jacobian to approximate $\Gamma$ when the DNN that approximates $Z$ involves tf.keras.layers.BatchNormalization layers  returns something with the expected shape, but its contents have an unclear meaning, see~TensorFlow documentation\footnote{\url{https://www.tensorflow.org/guide/advanced\_autodiff\#batch\_jacobian}}, batch Jacobian section. Therefore, batch normalization is emitted not only in the DBDP scheme, but also in our scheme to ensure a fair comparison. For the SGD iterations, we use the Adam optimizer with a stepwise learning rate decay approach. We choose a batch size of $B=1024$ for each of $\kappa$ optimization steps. At the discrete time point $t_{N-1}$, we consider $\Kf = 24000$ optimization steps, where the learning rate $\alpha$ is adjusted as follows 
\begin{equation*}
  \alpha_{\kappa}=\begin{cases}
    \num{1e-2}, & \text{for $1 \leq \kappa \leq 2000$},\\
    \num{3e-3}, & \text{for $2000 < \kappa \leq 4000$},\\
    \num{1e-3}, & \text{for $4000 < \kappa \leq 8000$},\\
    \num{3e-4}, & \text{for $8000 < \kappa \leq 12000$},\\
    \num{1e-4}, & \text{for $12000 < \kappa \leq 16000$},\\
    \num{3e-5}, & \text{for $16000 < \kappa \leq 20000$},\\
    \num{1e-5}, & \text{for $20000 < \kappa \leq \Kf$}.\\
    \end{cases}
\end{equation*}
For the next discrete time points (i.e., $t_{N-2}, \ldots, t_0$), we make use of the transfer learning approach, and reduce the number of optimization steps to $\Kf = 10000$, and use the following learning rates
\begin{equation*}
  \alpha_{\kappa}=\begin{cases}
    \num{1e-3}, & \text{for $1 \leq \kappa \leq 2000$},\\
    \num{3e-4}, & \text{for $2000 < \kappa \leq 4000$},\\
    \num{1e-4}, & \text{for $4000 < \kappa \leq 6000$},\\
    \num{3e-5}, & \text{for $6000 < \kappa \leq 8000$},\\
    \num{1e-5}, & \text{for $8000 < \kappa \leq \Kf$}.\\
    \end{cases}
\end{equation*}
The gradient of the driver function $f$ w.r.t each variable $(x, y, z)$ and the function $g$ w.r.t to variable $x$ are calculated by using AD, namely tf.GradientTape in TensorFlow. For the gradient of the function representing $Z_t$ (when available) in~\eqref{eq3} w.r.t to variable $x$, tf.GradientTape.batch\_jacobian is used. Note that we consider a uniform time discretization $\Delta$ of $[0, T]$.

We define the following mean squared errors (MSEs) as performance metrics for a sample with the size $B$:
\begin{equation*}
    \tilde{\varepsilon}^{y}_n := \frac{1}{B}\sum_{j=1}^B \left\vert Y_{n, j} - Y_{n, j}^{\Delta, \hat{\theta}} \right\vert^2, \quad \tilde{\varepsilon}^{z}_n := \frac{1}{B}\sum_{j=1}^B \left\vert Z_{n, j} - Z_{n, j}^{\Delta, \hat{\theta}} \right\vert^2, \quad \tilde{\varepsilon}^{\gamma}_n := \frac{1}{B}\sum_{j=1}^B \left\vert \Gamma_{n, j} - \Gamma_{n, j}^{\Delta, \hat{\theta}} \right\vert^2,
\end{equation*}
for each process. To account for the stochasticity of the underlying Brownian motion and the Adam optimizer, we conduct $Q = 10$ independent runs (training's) of the algorithms and define, e.g.,
\begin{equation*}
    \overline{{\tilde{\varepsilon}}}^{y}_n := \frac{1}{Q} \sum_{q=1}^Q \tilde{\epsilon}^{y}_{n,q},
\end{equation*}
as the mean MSE for the process $Y$, and similarly for the other processes. Note that as a relative measure of the MSE, we consider, e.g.,
\begin{equation*}
    \tilde{\varepsilon}^{y, r}_n := \frac{1}{B}\sum_{j=1}^B \frac{\left\vert Y_{n, j} - Y_{n, j}^{\Delta, \hat{\theta}} \right\vert^2}{\left\vert Y_{n, j} \right\vert^2},
\end{equation*}
for the process $Y$, and similarly for the other processes. We choose a testing sample of size $B = 1024.$ The computation time (runtime) for one run of the algorithms is defined as $\tau$, and the average computation time over $Q=10$ runs as $\overline{{\tau}}: = \frac{1}{Q} \sum_{q=1}^Q \tau_q.$

\subsection{The Black-Scholes BSDE}
\label{subsec62}
We start with a linear BSDE - the Black-Scholes BSDE - which is used for pricing of European options. 
\begin{example}
The high-dimenisonal Black-Scholes BSDE reads~\cite{zhang2013sparse}
\begin{equation*}
    \begin{split}
    \left\{
        \begin{array}{rcl}
             dX_t^k &=&  a_k X_t^k\,dt + b_k X_t^k \,dW_t^k, \\
             X_0^k &=& x_0^k, \quad k = 1,\ldots, d,\\  
           -dY_t &=& - \left( R Y_t + \sum_{k=1}^d  \frac{ a_k - R + \delta_k}{b_k} Z_t^k\right)\,dt- Z_t \,dW_t,\\ 
   		   	   Y_T &=& \left(\prod_{k=1}^d \left(X_T^k\right)^{c_k}-K\right)^+,
        \end{array}
    \right. \\ 
    \end{split}
\end{equation*}
\label{ex1}
\end{example}
where $c_k >0$ and $\sum_{k=1}^d c_k = 1$. Note that $a_k$ represents the expected return of the stock $X_t^k$, $b_k$ the volatility of the stock returns, $\delta_k$ is its dividend rate, and $x_0^k$ is the price of the stock at $t =0$. Moreover, $X_T$ is the price of the stocks at time $T$, which denotes the maturity of the option contract. The value $K$ represents the contract's strike price. Finally, $R$ corresponds to the risk-free interest rate. The analytic solution (the option price $Y_t$ and its delta hedging strategy $Z_t$) is given by
\begin{equation}
 \begin{split}
    \left\{
        \begin{array}{rcl}
            Y_t &=& u(t, X_t) = \exp\left(-\check{\delta} \left(T-t\right)\right) \prod_{k=1}^d \left(X_t^k\right)^{c_k} \Phi \left(\check{d}_1\right)-\exp\left(-R\left(T-t\right)\right) K \Phi \left(\check{d}_2\right),\\
           Z_t^k &=& \frac{\partial u}{\partial x_k} b_k X_t^k = c_k \exp\left(-\check{\delta} \left(T-t\right)\right) \prod_{k=1}^d \left(X_t^k\right)^{c_k} \Phi \left(\check{d}_1\right)b_k, \quad k = 1, \ldots, d,\\
           \check{d}_1 &=& \frac{\ln\left(\frac{\prod_{k=1}^d \left(X_t^k\right)^{c_k}}{K}\right) + \left( R-\check{\delta} + \frac{\check{b}^2}{2} \right) \left(T-t\right)}{\check{b} \sqrt{T-t}},\\
            \check{d}_2 &=& \check{d}_1 - \check{b}\sqrt{T-t},\\
            \check{b}^2 &=& \sum_{k=1}^d (b_k c_k)^2,\quad \check{\delta} = \sum_{k=1}^d c_k\left(\delta_k + \frac{b_k^2}{2} \right) - \frac{\check{b}^2}{2},               
        \end{array}
    \right. \\
    \end{split}
    \label{eq55}
\end{equation}
where $\Phi \left(\cdot\right)$ is the standard normal cumulative distribution function. The analytical solution $\Gamma_t = \nabla_x \left( \nabla_x u\left(t, X_t\right) b\left(t, X_t\right) \right)$ is calculated by using AD. As we mentioned in Section~\ref{sec5}, when dealing with a forward SDE represented by the GBM, we apply the ln-transformation to ensure that the theoretical analysis is applicable to our numerical experiments. We define $\check{X}_t :=\ln\left(X_t\right)$ and $\check{u}(t, \check{X}_t) := u(t, X_t)$. Using the Feynman-Kac formula, we write the Black-Scholes BSDE in the ln-domain
\begin{equation}
    \begin{split}
    \left\{
        \begin{array}{rcl}
             d\check{X}_t^k &=&  \left(a_k - \frac{1}{2} b_k^2\right)\,dt + b_k \,dW_t^k, \\
             \check{X}_0^k &=& \ln\left(x_0^k\right), \quad k = 1,\ldots, d,\\  
            -d\check{Y}_t &=& - \left( R \check{Y}_t + \sum_{k=1}^d  \frac{ a_k - R + \delta_k}{b_k} \check{Z}_t^k\right)\,dt- \check{Z}_t \,dW_t,\\ 
   		   	   \check{Y}_T &=& \left( \exp\left(\sum_{k=1}^dc_k\check{X}_T^k\right)-K\right)^+.
        \end{array}
    \right. \\ 
    \end{split}
    \label{eq56}
\end{equation}
The ln-transformation simplifies the Malliavin derivatives as $D_n X_n^k = b_k X_n^k$, $D_n X_{n+1}^k = b_k X_{n+1}^k$ and $D_n \check{X}_{n}^k = D_n \check{X}_{n+1}^k = b_k$ for $k=1,\ldots,d$. Note that $\left( \check{Y}_t, \check{Z}_t \right) = \left( Y_t, Z_t \right)$ since $\check{Y}_t = \check{u}(t, \check{X}_t) = u(t, X_t) = Y_t$ and $\check{Z}_t^k = \frac{\partial \check{u}}{\partial \check{x}_k} b_k = \frac{\partial u}{\partial x_k} b_k X_t^k =  Z_t^k$ for $k = 1, \ldots, d$. Hence, we can compare the approximated solution of~\eqref{eq56} in the ln-domain with the exact solution of Example~\ref{ex1} given in~\eqref{eq55}. In case of the process $\Gamma$, we have that $\check{\Gamma}_t^{k_1, k_2} \frac{1}{X_t^{k_2}} = \Gamma_t^{k_1, k_2}$ for $k_1, k_2= 1, \ldots, d$. In the following tests, for $k=1,\ldots, d$, we set $x_0^k = 100$, $a_k = 0.05$, $b_k = 0.2$, $R = 0.03$, $c_k = \frac{1}{d}$ and $\delta_k = 0$. Moreover, we set $K = 100$, $T = 1$ and $d \in \{1, 10, 50\}$. 

We start with $d=1$ as we can also visualize the exact and approximated values of each process over the discrete time domain. In Figure~\ref{fig1}, we display the exact and approximated value of the processes $\left( Y, Z, \Gamma\right)$ from the first run of DBDP and DLBDP schemes at arbitrary discrete time points $\left(t_2, t_{32}, t_{63} \right) = \left(0.0312, 0.5000, 0.9844 \right)$ using $N = 64$. Moreover, we show only $256$ out of $1024$ testing samples for better visualization. 
\begin{figure}[tbhp]
	\centering
	\subfloat[Process $Y$ at $t_{63} = 0.9844$.]{
        \pgfplotstableread{"Figures/Example1/d1/Y_t_0.9844.dat"}{\table}
        \begin{tikzpicture} 
            \begin{axis}[
                xmin = 60, xmax = 200,
                ymin = -1, ymax = 100,
                xtick distance = 40,
                ytick distance = 20,
                width = 0.33\textwidth,
                height = 0.35\textwidth,
                xlabel = {$X_{63}^{\Delta}$},
                xticklabel style={/pgf/number format/1000 sep=},
                legend cell align = {left},
                legend pos = north west,
                legend style={nodes={scale=0.6, transform shape}}]
                \addplot[only marks, red, mark = o, mark options={scale=0.5}] table [x = 0, y =1] {\table}; 
                \addplot[only marks, blue, mark = square, mark options={scale=0.5}] table [x =0, y = 2] {\table};    
                \addplot[only marks, green, mark = triangle, mark options={scale=0.5}] table [x =0, y = 3] {\table};
                \legend{
                    $Y_{63}$, 
                    $Y_{63}^{\Delta, \hat{\theta}}$-DBDP,
                    $Y_{63}^{\Delta, \hat{\theta}}$-DLBDP
                } 
            \end{axis}
            \end{tikzpicture}
		\label{fig1a}
	}
	\subfloat[Process $Z$ at $t_{63} = 0.9844$.]{
        \pgfplotstableread{"Figures/Example1/d1/Z_t_0.9844.dat"}{\table}
        \begin{tikzpicture} 
            \begin{axis}[
                xmin = 60, xmax = 200,
                ymin = -1, ymax = 42,
                xtick distance = 40,
                ytick distance = 10,
                width = 0.33\textwidth,
                height = 0.35\textwidth,
                xlabel = {$X_{63}^{\Delta}$},
                xticklabel style={/pgf/number format/1000 sep=},
                legend cell align = {left},
                legend pos = south east,
                legend style={nodes={scale=0.6, transform shape}}]
                \addplot[only marks, red, mark = o, mark options={scale=0.5}] table [x = 0, y =1] {\table}; 
                \addplot[only marks, blue, mark = square, mark options={scale=0.5}] table [x =0, y = 2] {\table};                
                \addplot[only marks, green, mark = triangle, mark options={scale=0.5}] table [x =0, y = 3] {\table};                
                \legend{
                    $Z_{63}$, 
                    $Z_{63}^{\Delta, \hat{\theta}}$-DBDP,
                    $Z_{63}^{\Delta, \hat{\theta}}$-DLBDP
                } 
            \end{axis}
            \end{tikzpicture}
		\label{fig1b}
	}
	\subfloat[Process $\Gamma$ at $t_{63} = 0.9844$.]{
        \pgfplotstableread{"Figures/Example1/d1/Gamma_t_0.9844.dat"}{\table}
        \begin{tikzpicture} 
            \begin{axis}[
                xmin = 60, xmax = 200,
                ymin = -0.1, ymax = 3.5,
                xtick distance = 40,
                ytick distance = 1,
                width = 0.33\textwidth,
                height = 0.35\textwidth,
                xlabel = {$X_{63}^{\Delta}$},
                xticklabel style={/pgf/number format/1000 sep=},
                legend cell align = {left},
                legend pos = north east,
                legend style={nodes={scale=0.6, transform shape}}]
                \addplot[only marks, red, mark = o, mark options={scale=0.5}] table [x = 0, y =1] {\table}; 
                \addplot[only marks, blue, mark = square, mark options={scale=0.5}] table [x =0, y = 2] {\table};                
                \addplot[only marks, green, mark = triangle, mark options={scale=0.5}] table [x =0, y = 3] {\table};                
                \legend{
                    $\Gamma_{63}$, 
                    $\Gamma_{63}^{\Delta, \hat{\theta}}$-DBDP,
                    $\Gamma_{63}^{\Delta, \hat{\theta}}$-DLBDP
                } 
            \end{axis}
            \end{tikzpicture}
		\label{fig1c}
	}
        \hfill
	\subfloat[Process $Y$ at $t_{32} = 0.5000$.]{
        \pgfplotstableread{"Figures/Example1/d1/Y_t_0.5000.dat"}{\table}
        \begin{tikzpicture} 
            \begin{axis}[
                xmin = 68, xmax = 160,
                ymin = -1, ymax = 60,
                xtick distance = 20,
                ytick distance = 10,
                width = 0.33\textwidth,
                height = 0.35\textwidth,
                xlabel = {$X_{32}^{\Delta}$},
                xticklabel style={/pgf/number format/1000 sep=},
                legend cell align = {left},
                legend pos = north west,
                legend style={nodes={scale=0.6, transform shape}}]
                \addplot[only marks, red, mark = o, mark options={scale=0.5}] table [x = 0, y =1] {\table}; 
                \addplot[only marks, blue, mark = square, mark options={scale=0.5}] table [x =0, y = 2] {\table};    
                \addplot[only marks, green, mark = triangle, mark options={scale=0.5}] table [x =0, y = 3] {\table};
                \legend{
                    $Y_{32}$, 
                    $Y_{32}^{\Delta, \hat{\theta}}$-DBDP,
                    $Y_{32}^{\Delta, \hat{\theta}}$-DLBDP
                } 
            \end{axis}
            \end{tikzpicture}
		\label{fig1d}
	}
	\subfloat[Process $Z$ at $t_{32} = 0.5000$.]{
        \pgfplotstableread{"Figures/Example1/d1/Z_t_0.5000.dat"}{\table}
        \begin{tikzpicture} 
            \begin{axis}[
                xmin = 68, xmax = 160,
                ymin = -1, ymax = 35,
                xtick distance = 40,
                ytick distance = 10,
                width = 0.33\textwidth,
                height = 0.35\textwidth,
                xlabel = {$X_{32}^{\Delta}$},
                xticklabel style={/pgf/number format/1000 sep=},
                legend cell align = {left},
                legend pos = south east,
                legend style={nodes={scale=0.6, transform shape}}]
                \addplot[only marks, red, mark = o, mark options={scale=0.5}] table [x = 0, y =1] {\table}; 
                \addplot[only marks, blue, mark = square, mark options={scale=0.5}] table [x =0, y = 2] {\table};                
                \addplot[only marks, green, mark = triangle, mark options={scale=0.5}] table [x =0, y = 3] {\table};                
                \legend{
                    $Z_{32}$, 
                    $Z_{32}^{\Delta, \hat{\theta}}$-DBDP,
                    $Z_{32}^{\Delta, \hat{\theta}}$-DLBDP
                } 
            \end{axis}
            \end{tikzpicture}
		\label{fig1e}
	}
	\subfloat[Process $\Gamma$ at $t_{32} = 0.5000$.]{
        \pgfplotstableread{"Figures/Example1/d1/Gamma_t_0.5000.dat"}{\table}
        \begin{tikzpicture} 
            \begin{axis}[
                xmin = 68, xmax = 160,
                ymin = 0, ymax = 0.78,
                xtick distance = 40,
                ytick distance = 0.2,
                width = 0.33\textwidth,
                height = 0.35\textwidth,
                xlabel = {$X_{32}^{\Delta}$},
                xticklabel style={/pgf/number format/1000 sep=},
                legend cell align = {left},
                legend pos = north east,
                legend style={nodes={scale=0.6, transform shape}}]
                \addplot[only marks, red, mark = o, mark options={scale=0.5}] table [x = 0, y =1] {\table}; 
                \addplot[only marks, blue, mark = square, mark options={scale=0.5}] table [x =0, y = 2] {\table};                
                \addplot[only marks, green, mark = triangle, mark options={scale=0.5}] table [x =0, y = 3] {\table};                
                \legend{
                    $\Gamma_{32}$, 
                    $\Gamma_{32}^{\Delta, \hat{\theta}}$-DBDP,
                    $\Gamma_{32}^{\Delta, \hat{\theta}}$-DLBDP
                } 
            \end{axis}
            \end{tikzpicture}
		\label{fig1f}
	}
        \hfill
	\subfloat[Process $Y$ at $t_{2} = 0.0312$.]{
        \pgfplotstableread{"Figures/Example1/d1/Y_t_0.0312.dat"}{\table}
        \begin{tikzpicture} 
            \begin{axis}[
                xmin = 93.5, xmax = 110,
                ymin = 6, ymax = 17,
                xtick distance = 5,
                ytick distance = 2,
                width = 0.33\textwidth,
                height = 0.35\textwidth,
                xlabel = {$X_{2}^{\Delta}$},
                xticklabel style={/pgf/number format/1000 sep=},
                legend cell align = {left},
                legend pos = north west,
                legend style={nodes={scale=0.6, transform shape}}]
                \addplot[only marks, red, mark = o, mark options={scale=0.5}] table [x = 0, y =1] {\table}; 
                \addplot[only marks, blue, mark = square, mark options={scale=0.5}] table [x =0, y = 2] {\table};    
                \addplot[only marks, green, mark = triangle, mark options={scale=0.5}] table [x =0, y = 3] {\table};
                \legend{
                    $Y_{2}$, 
                    $Y_{2}^{\Delta, \hat{\theta}}$-DBDP,
                    $Y_{2}^{\Delta, \hat{\theta}}$-DLBDP
                } 
            \end{axis}
            \end{tikzpicture}
		\label{fig1g}
	}
	\subfloat[Process $Z$ at $t_{2} = 0.0312$.]{
        \pgfplotstableread{"Figures/Example1/d1/Z_t_0.0312.dat"}{\table}
        \begin{tikzpicture} 
            \begin{axis}[
                xmin = 93.5, xmax = 110,
                ymin = 9, ymax = 17,
                xtick distance = 5,
                ytick distance = 2,
                width = 0.33\textwidth,
                height = 0.35\textwidth,
                xlabel = {$X_{2}^{\Delta}$},
                xticklabel style={/pgf/number format/1000 sep=},
                legend cell align = {left},
                legend pos = south east,
                legend style={nodes={scale=0.6, transform shape}}]
                \addplot[only marks, red, mark = o, mark options={scale=0.5}] table [x = 0, y =1] {\table}; 
                \addplot[only marks, blue, mark = square, mark options={scale=0.5}] table [x =0, y = 2] {\table};                
                \addplot[only marks, green, mark = triangle, mark options={scale=0.5}] table [x =0, y = 3] {\table};                
                \legend{
                    $Z_{2}$, 
                    $Z_{2}^{\Delta, \hat{\theta}}$-DBDP,
                    $Z_{2}^{\Delta, \hat{\theta}}$-DLBDP
                } 
            \end{axis}
            \end{tikzpicture}
		\label{fig1h}
	}
	\subfloat[Process $\Gamma$ at $t_{2} = 0.0312$.]{
        \pgfplotstableread{"Figures/Example1/d1/Gamma_t_0.0312.dat"}{\table}
        \begin{tikzpicture} 
            \begin{axis}[
                xmin = 93.5, xmax = 110,
                ymin = 0.36, ymax = 0.56,
                xtick distance = 5,
                ytick distance = 0.05,
                width = 0.33\textwidth,
                height = 0.35\textwidth,
                xlabel = {$X_{2}^{\Delta}$},
                xticklabel style={/pgf/number format/1000 sep=},
                legend cell align = {left},
                legend pos = south west,
                legend style={nodes={scale=0.6, transform shape}}]
                \addplot[only marks, red, mark = o, mark options={scale=0.5}] table [x = 0, y =1] {\table}; 
                \addplot[only marks, blue, mark = square, mark options={scale=0.5}] table [x =0, y = 2] {\table};                
                \addplot[only marks, green, mark = triangle, mark options={scale=0.5}] table [x =0, y = 3] {\table};                
                \legend{
                    $\Gamma_{2}$, 
                    $\Gamma_{2}^{\Delta, \hat{\theta}}$-DBDP,
                    $\Gamma_{2}^{\Delta, \hat{\theta}}$-DLBDP
                } 
            \end{axis}
            \end{tikzpicture}
		\label{fig1i}
	}
    \caption{Exact and approximated values of the processes $\left(Y, Z, \Gamma \right)$ from the first run of DBDP and DLBDP schemes at discrete time points $\left(t_2, t_{32}, t_{63} \right) = \left(0.0312, 0.5000, 0.9844 \right)$ using $256$ samples of the testing sample in Example~\ref{ex1}, for $d=1$ and $N = 64$.}
    \label{fig1}
\end{figure}
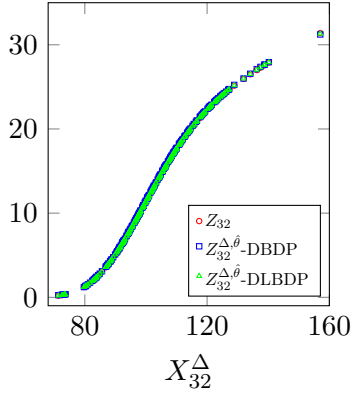
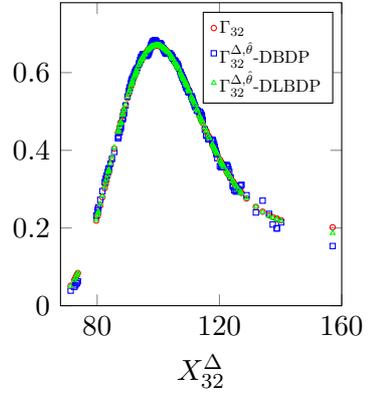
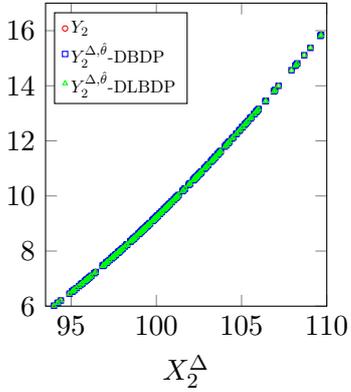
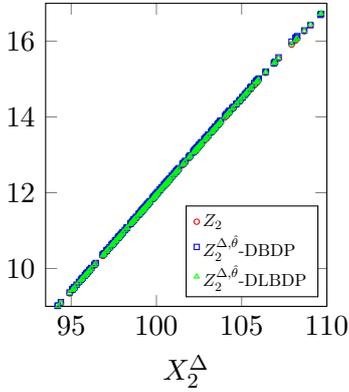
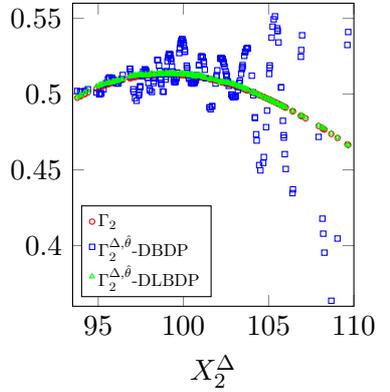
Our scheme outperforms the DBDP scheme in approximating the process $\Gamma$, particularly as we approach $t_0$. It is difficult to observe any improvement from our scheme for the processes $Y$ and $Z$ in Figure~\ref{fig1}. Therefore, to provide a clearer comparison using the entire testing sample across the discrete domain $\Delta$, we visualize in Figure~\ref{fig2} the mean MSE values for each process for $d \in \{1, 10, 50\}$. The STD of the MSE values is given in the shaded area.
\begin{figure}[tbhp]
	\centering
	\subfloat[Process $Y$, $d=1$.]{
        \pgfplotstableread{"Figures/Example1/d1/vareps_y_over_t.dat"}{\table}        
        \begin{tikzpicture} 
            \begin{axis}[
                xmin = 0, xmax = 1,
                ymin = 5e-5, ymax = 4e-3,
                xtick distance = 0.2,
                ymode=log, 
                grid = both,
                width = 0.33\textwidth,
                height = 0.35\textwidth,
                xlabel = {$t_n$},
                legend cell align = {left},
                legend pos = north west,
                legend style={nodes={scale=0.5, transform shape}}]
                \addplot[smooth, ultra thick, blue, dashed] table [x = 0, y =1] {\table}; 
                \addplot[smooth, ultra thick, green, solid] table [x = 0, y =4] {\table}; 
                \addplot [name path=upper,draw=none] table[x=0,y=2] {\table};
                \addplot [name path=lower,draw=none] table[x=0,y=3] {\table};
                \addplot [fill=blue!40] fill between[of=upper and lower];           
                \addplot [name path=upper,draw=none] table[x=0,y=5] {\table};
                \addplot [name path=lower,draw=none] table[x=0,y=6] {\table};
                \addplot [fill=green!40] fill between[of=upper and lower];                
                \legend{
                   $\overline{{\tilde{\varepsilon}}}^{y}_n$-DBDP,
                   $\overline{{\tilde{\varepsilon}}}^{y}_n$-DLBDP
                }
            \end{axis}
            \end{tikzpicture}
		\label{fig2a}
	} 
	\subfloat[Process $Z$, $d=1$.]{ \pgfplotstableread{"Figures/Example1/d1/vareps_z_over_t.dat"}{\table}
        \begin{tikzpicture} 
            \begin{axis}[
                xmin = 0, xmax = 1,
                ymin = 1e-4, ymax = 8e-2,
                xtick distance = 0.2,
                ymode=log,
                grid = both,
                width = 0.33\textwidth,
                height = 0.35\textwidth,
                xlabel = {$t_n$},
                legend cell align = {left},
                legend pos = north west,
                legend style={nodes={scale=0.5, transform shape}}]
                \addplot[smooth, ultra thick, blue, dashed] table [x = 0, y =1] {\table}; 
                \addplot[smooth, ultra thick, green, solid] table [x = 0, y =4] {\table}; 
                \addplot [name path=upper,draw=none] table[x=0,y=2] {\table};
                \addplot [name path=lower,draw=none] table[x=0,y=3] {\table};
                \addplot [fill=blue!40] fill between[of=upper and lower];           
                \addplot [name path=upper,draw=none] table[x=0,y=5] {\table};
                \addplot [name path=lower,draw=none] table[x=0,y=6] {\table};
                \addplot [fill=green!40] fill between[of=upper and lower];                
                \legend{
                   $\overline{{\tilde{\varepsilon}}}^{z}_n$-DBDP,
                   $\overline{{\tilde{\varepsilon}}}^{z}_n$-DLBDP
                }
            \end{axis}
            \end{tikzpicture}
		\label{fig2b}
	} 
        \subfloat[Process $\Gamma$, $d=1$.]{ \pgfplotstableread{"Figures/Example1/d1/vareps_gamma_over_t.dat"}{\table}
        \begin{tikzpicture} 
            \begin{axis}[
                xmin = 0, xmax = 1,
                ymin = 2e-7, ymax = 1,
                xtick distance = 0.2,
                ymode=log,
                grid = both,
                width = 0.33\textwidth,
                height = 0.35\textwidth,
                xlabel = {$t_n$},
                legend cell align = {left},
                legend pos = north east,
                legend style={nodes={scale=0.5, transform shape}}]
                \addplot[smooth, ultra thick, blue, dashed] table [x = 0, y =1] {\table}; 
                \addplot[smooth, ultra thick, green, solid] table [x = 0, y =4] {\table}; 
                \addplot [name path=upper,draw=none] table[x=0,y=2] {\table};
                \addplot [name path=lower,draw=none] table[x=0,y=3] {\table};
                \addplot [fill=blue!40] fill between[of=upper and lower];           
                \addplot [name path=upper,draw=none] table[x=0,y=5] {\table};
                \addplot [name path=lower,draw=none] table[x=0,y=6] {\table};
                \addplot [fill=green!40] fill between[of=upper and lower];                
                \legend{
                   $\overline{{\tilde{\varepsilon}}}^{\gamma}_n$-DBDP,
                   $\overline{{\tilde{\varepsilon}}}^{\gamma}_n$-DLBDP
                }
            \end{axis}
            \end{tikzpicture}
		\label{fig2c}
	} 
        \hfill
        \subfloat[Process $Y$, $d=10$.]{
        \pgfplotstableread{"Figures/Example1/d10/vareps_y_over_t.dat"}{\table}        
        \begin{tikzpicture} 
            \begin{axis}[
                xmin = 0, xmax = 1,
                ymin = 1e-5, ymax = 1e-4,
                xtick distance = 0.2,
                ymode=log,
                grid = both,
                width = 0.33\textwidth,
                height = 0.35\textwidth,
                xlabel = {$t_n$},
                legend cell align = {left},
                legend pos = north east,
                legend style={nodes={scale=0.5, transform shape}}]
                \addplot[smooth, ultra thick, blue, dashed] table [x = 0, y =1] {\table}; 
                \addplot[smooth, ultra thick, green, solid] table [x = 0, y =4] {\table}; 
                \addplot [name path=upper,draw=none] table[x=0,y=2] {\table};
                \addplot [name path=lower,draw=none] table[x=0,y=3] {\table};
                \addplot [fill=blue!40] fill between[of=upper and lower];           
                \addplot [name path=upper,draw=none] table[x=0,y=5] {\table};
                \addplot [name path=lower,draw=none] table[x=0,y=6] {\table};
                \addplot [fill=green!40] fill between[of=upper and lower];                
                \legend{
                   $\overline{{\tilde{\varepsilon}}}^{y}_n$-DBDP,
                   $\overline{{\tilde{\varepsilon}}}^{y}_n$-DLBDP
                }
            \end{axis}
            \end{tikzpicture}
		\label{fig2d}
	} 
	\subfloat[Process $Z$, $d=10$.]{ \pgfplotstableread{"Figures/Example1/d10/vareps_z_over_t.dat"}{\table}
        \begin{tikzpicture} 
            \begin{axis}[
                xmin = 0, xmax = 1,
                ymin = 3e-5, ymax = 2e-3,
                xtick distance = 0.2,
                ymode=log,
                grid = both,
                width = 0.33\textwidth,
                height = 0.35\textwidth,
                xlabel = {$t_n$},
                legend cell align = {left},
                legend pos = north west,
                legend style={nodes={scale=0.5, transform shape}}]
                \addplot[smooth, ultra thick, blue, dashed] table [x = 0, y =1] {\table}; 
                \addplot[smooth, ultra thick, green, solid] table [x = 0, y =4] {\table}; 
                \addplot [name path=upper,draw=none] table[x=0,y=2] {\table};
                \addplot [name path=lower,draw=none] table[x=0,y=3] {\table};
                \addplot [fill=blue!40] fill between[of=upper and lower];           
                \addplot [name path=upper,draw=none] table[x=0,y=5] {\table};
                \addplot [name path=lower,draw=none] table[x=0,y=6] {\table};
                \addplot [fill=green!40] fill between[of=upper and lower];                
                \legend{
                   $\overline{{\tilde{\varepsilon}}}^{z}_n$-DBDP,
                   $\overline{{\tilde{\varepsilon}}}^{z}_n$-DLBDP
                }
            \end{axis}
            \end{tikzpicture}
		\label{fig2e}
	} 
        \subfloat[Process $\Gamma$, $d=10$.]{ \pgfplotstableread{"Figures/Example1/d10/vareps_gamma_over_t.dat"}{\table}
        \begin{tikzpicture} 
            \begin{axis}[
                xmin = 0, xmax = 1,
                ymin = 5e-8, ymax = 5e-2,
                xtick distance = 0.2,
                ymode=log,
                grid = both,
                width = 0.33\textwidth,
                height = 0.35\textwidth,
                xlabel = {$t_n$},
                legend cell align = {left},
                legend pos = north east,
                legend style={nodes={scale=0.5, transform shape}}]
                \addplot[smooth, ultra thick, blue, dashed] table [x = 0, y =1] {\table}; 
                \addplot[smooth, ultra thick, green, solid] table [x = 0, y =4] {\table}; 
                \addplot [name path=upper,draw=none] table[x=0,y=2] {\table};
                \addplot [name path=lower,draw=none] table[x=0,y=3] {\table};
                \addplot [fill=blue!40] fill between[of=upper and lower];           
                \addplot [name path=upper,draw=none] table[x=0,y=5] {\table};
                \addplot [name path=lower,draw=none] table[x=0,y=6] {\table};
                \addplot [fill=green!40] fill between[of=upper and lower];                
                \legend{
                   $\overline{{\tilde{\varepsilon}}}^{\gamma}_n$-DBDP,
                   $\overline{{\tilde{\varepsilon}}}^{\gamma}_n$-DLBDP
                }
            \end{axis}
            \end{tikzpicture}
		\label{fig2f}
	} 
        \hfill
        \subfloat[Process $Y$, $d=50$.]{
        \pgfplotstableread{"Figures/Example1/d50/vareps_y_over_t.dat"}{\table}        
        \begin{tikzpicture} 
            \begin{axis}[
                xmin = 0, xmax = 1,
                ymin = 4e-6, ymax = 1e-4,
                xtick distance = 0.2,
                ymode=log,
                grid = both,
                width = 0.33\textwidth,
                height = 0.35\textwidth,
                xlabel = {$t_n$},
                legend cell align = {left},
                legend pos = north east,
                legend style={nodes={scale=0.5, transform shape}}]
                \addplot[smooth, ultra thick, blue, dashed] table [x = 0, y =1] {\table}; 
                \addplot[smooth, ultra thick, green, solid] table [x = 0, y =4] {\table}; 
                \addplot [name path=upper,draw=none] table[x=0,y=2] {\table};
                \addplot [name path=lower,draw=none] table[x=0,y=3] {\table};
                \addplot [fill=blue!40] fill between[of=upper and lower];           
                \addplot [name path=upper,draw=none] table[x=0,y=5] {\table};
                \addplot [name path=lower,draw=none] table[x=0,y=6] {\table};
                \addplot [fill=green!40] fill between[of=upper and lower];                
                \legend{
                   $\overline{{\tilde{\varepsilon}}}^{y}_n$-DBDP,
                   $\overline{{\tilde{\varepsilon}}}^{y}_n$-DLBDP
                }
            \end{axis}
            \end{tikzpicture}
		\label{fig2g}
	} 
	\subfloat[Process $Z$, $d=50$.]{ \pgfplotstableread{"Figures/Example1/d50/vareps_z_over_t.dat"}{\table}
        \begin{tikzpicture} 
            \begin{axis}[
                xmin = 0, xmax = 1,
                ymin = 4e-5, ymax = 1e-3,
                xtick distance = 0.2,
                ymode=log,
                grid = both,
                width = 0.33\textwidth,
                height = 0.35\textwidth,
                xlabel = {$t_n$},
                legend cell align = {left},
                legend pos = north west,
                legend style={nodes={scale=0.5, transform shape}}]
                \addplot[smooth, ultra thick, blue, dashed] table [x = 0, y =1] {\table}; 
                \addplot[smooth, ultra thick, green, solid] table [x = 0, y =4] {\table}; 
                \addplot [name path=upper,draw=none] table[x=0,y=2] {\table};
                \addplot [name path=lower,draw=none] table[x=0,y=3] {\table};
                \addplot [fill=blue!40] fill between[of=upper and lower];           
                \addplot [name path=upper,draw=none] table[x=0,y=5] {\table};
                \addplot [name path=lower,draw=none] table[x=0,y=6] {\table};
                \addplot [fill=green!40] fill between[of=upper and lower];                
                \legend{
                   $\overline{{\tilde{\varepsilon}}}^{z}_n$-DBDP,
                   $\overline{{\tilde{\varepsilon}}}^{z}_n$-DLBDP
                }
            \end{axis}
            \end{tikzpicture}
		\label{fig2h}
	} 
        \subfloat[Process $\Gamma$, $d=50$.]{ \pgfplotstableread{"Figures/Example1/d50/vareps_gamma_over_t.dat"}{\table}
        \begin{tikzpicture} 
            \begin{axis}[
                xmin = 0, xmax = 1,
                ymin = 9e-8, ymax = 1e-2,
                xtick distance = 0.2,
                ymode=log,
                grid = both,
                width = 0.33\textwidth,
                height = 0.35\textwidth,
                xlabel = {$t_n$},
                legend cell align = {left},
                legend pos = north east,
                legend style={nodes={scale=0.5, transform shape}}]
                \addplot[smooth, ultra thick, blue, dashed] table [x = 0, y =1] {\table}; 
                \addplot[smooth, ultra thick, green, solid] table [x = 0, y =4] {\table}; 
                \addplot [name path=upper,draw=none] table[x=0,y=2] {\table};
                \addplot [name path=lower,draw=none] table[x=0,y=3] {\table};
                \addplot [fill=blue!40] fill between[of=upper and lower];           
                \addplot [name path=upper,draw=none] table[x=0,y=5] {\table};
                \addplot [name path=lower,draw=none] table[x=0,y=6] {\table};
                \addplot [fill=green!40] fill between[of=upper and lower];                
                \legend{
                   $\overline{{\tilde{\varepsilon}}}^{\gamma}_n$-DBDP,
                   $\overline{{\tilde{\varepsilon}}}^{\gamma}_n$-DLBDP
                }
            \end{axis}
            \end{tikzpicture}
		\label{fig2i}
	} 
    \caption{Mean MSE values of the processes $\left(Y, Z, \Gamma \right)$ from DBDP and DLBDP schemes over the discrete time points $\{t_n\}_{n=0}^{N-1}$ using the testing sample in Example~\ref{ex1}, for $d \in \{1, 10, 50\}$ and $N = 64$. The STD of MSE values is given in the shaded area.}
\label{fig2}
\end{figure}
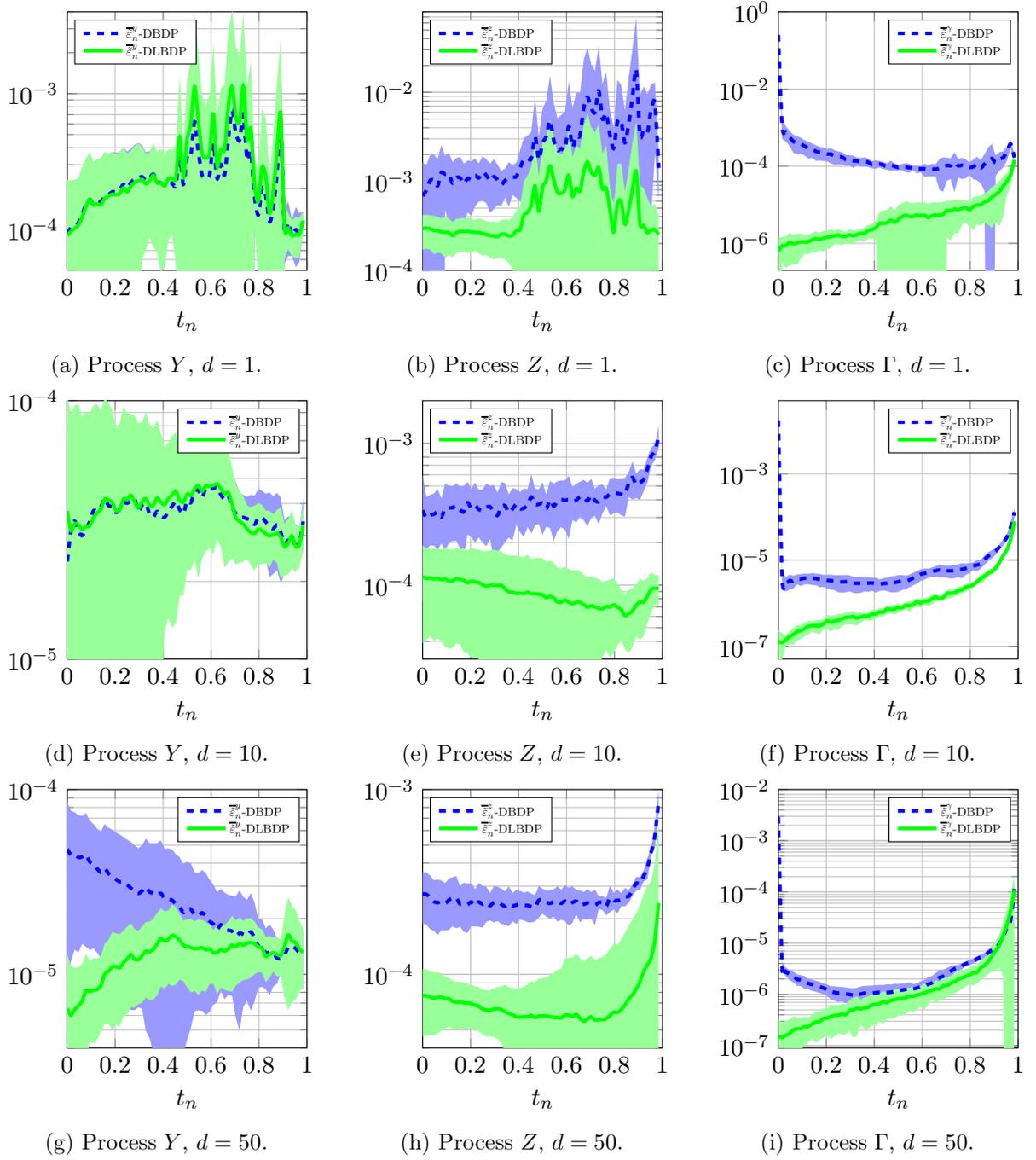
For the case of $d=1$, Figure~\ref{fig2c} 
 clearly shows a substantial improvement in approximating the process $\Gamma$ across the discrete time points $\{t_n\}_{n=0}^{N-1}$ achieved by our scheme compared to the DBDP scheme. Furthermore, Figure~\ref{fig2b} demonstrates that the DLBDP scheme also outperforms in approximating the process $Z$. However, there is no improvement achieved with our scheme for the process $Y$, as shown in Figure~\ref{fig2a}. As the dimension increases to $d=10$ and $d=50$, our scheme further exhibits a higher accuracy for approximating the processes $\left(Z, \Gamma \right)$. Moreover, an improvement in approximating the process $Y$ is evident for $d=
50$ from the DLBDP scheme compared to DBDP scheme, as displayed in Figure~\ref{fig2g}.

Next, we report in Table~\ref{tab1} the mean relative MSE of each process at $t_0$ while varying $N$ for $d \in \{1, 10, 50 \}$, along with the average computation time from both the DBDP and DLBDP schemes. The STD of the relative MSE values at $t_0$ is given in the brackets.
\begin{table}[tbhp]
    \begin{subtable}[h]{\textwidth}		
    \centering
      \resizebox{\textwidth}{!}{\begin{tabular}{| c | c | c | c | c |}
      \hline
        \multirow{3}{*}{Metric} & N = 2 & N = 8 & N = 32 & N = 64\\ \
        & DBDP & DBDP & DBDP & DBDP\\ 
        & DLBDP & DLBDP & DLBDP & DLBDP\\ \hline 
        \multirow{2}{*}{$\overline{{\tilde{\varepsilon}}}^{y, r}_0$} & $\num{1.10e-05}$ $(\num{1.33e-05})$ & $\num{3.88e-06}$ $(\num{1.89e-06})$ & $\num{3.13e-06}$ $(\num{4.34e-06})$ & $\num{1.05e-06}$ $(\num{1.43e-06})$ \\
        & $\num{1.04e-05}$ $(\num{1.11e-05})$ & $\num{3.40e-06}$ $(\num{3.89e-06})$ & $\num{3.26e-06}$ $(\num{4.43e-06})$ & $\num{1.02e-06}$ $(\num{1.60e-06})$ \\ \hline    
        \multirow{2}{*}{$\overline{{\tilde{\varepsilon}}}^{z, r}_0$} & $\num{3.22e-03}$ $(\num{3.64e-04})$ & $\num{2.04e-04}$ $(\num{3.16e-05})$ & $\num{1.87e-05}$ $(\num{8.02e-06})$ & $\num{4.77e-06}$ $(\num{6.59e-06})$ \\
        & $\num{9.55e-04}$ $(\num{1.29e-04})$ & $\num{7.42e-05}$ $(\num{1.20e-05})$ & $\num{5.76e-06}$ $(\num{1.75e-06})$ & $\num{2.05e-06}$ $(\num{8.60e-07})$ \\ \hline    
        \multirow{2}{*}{$\overline{{\tilde{\varepsilon}}}^{\gamma, r}_0$} & $\num{1.15e+00}$ $(\num{2.18e-02})$ & $\num{9.92e-01}$ $(\num{2.44e-03})$ & $\num{9.89e-01}$ $(\num{3.57e-03})$ & $\num{9.93e-01}$ $(\num{4.45e-03})$ \\
        & $\mathbf{\num{8.10e-04}}$ $(\mathbf{\num{6.60e-05}})$ & $\mathbf{\num{7.35e-05}}$ $(\mathbf{\num{2.26e-05}})$ & $\mathbf{\num{4.93e-06}}$ $(\mathbf{\num{4.82e-06}})$ & $\mathbf{\num{2.39e-06}}$ $(\mathbf{\num{3.41e-06}})$ \\ \hline    
        \multirow{2}{*}{$\overline{\tau}$} & $\num{9.49e+02}$ & $\num{2.74e+03}$ & $\num{1.05e+04}$ & $\num{2.24e+04}$ \\
         & $\num{7.44e+02}$ & $\num{2.11e+03}$ & $\num{8.33e+03}$ & $\num{1.80e+04}$ \\ \hline   
       \end{tabular}}
       \caption{$d=1.$}
        \label{tab1a} 
    \end{subtable}
     \\
    \begin{subtable}[h]{\textwidth}
      \resizebox{\textwidth}{!}{\begin{tabular}{| c | c | c | c | c |}
      \hline
        \multirow{3}{*}{Metric} & N = 2 & N = 8 & N = 32 & N = 64\\ \
        & DBDP & DBDP & DBDP & DBDP\\ 
        & DLBDP & DLBDP & DLBDP & DLBDP\\ \hline
        \multirow{2}{*}{$\overline{{\tilde{\varepsilon}}}^{y, r}_0$} & $\num{4.11e-04}$ $(\num{1.04e-04})$ & $\num{2.04e-05}$ $(\num{1.27e-05})$ & $\num{4.77e-06}$ $(\num{5.61e-06})$ & $\num{2.50e-06}$ $(\num{4.50e-06})$ \\
        & $\num{4.22e-05}$ $(\num{3.30e-05})$ & $\num{1.08e-05}$ $(\num{1.11e-05})$ & $\num{4.42e-06}$ $(\num{4.17e-06})$ & $\num{3.92e-06}$ $(\num{7.07e-06})$ \\ \hline    
        \multirow{2}{*}{$\overline{{\tilde{\varepsilon}}}^{z, r}_0$} & $\num{1.77e-02}$ $(\num{5.89e-04})$ & $\num{1.10e-03}$ $(\num{1.46e-04})$ & $\num{7.75e-05}$ $(\num{1.98e-05})$ & $\num{2.53e-05}$ $(\num{1.97e-05})$ \\
        & $\num{5.65e-03}$ $(\num{1.81e-04})$ & $\num{4.15e-04}$ $(\num{3.86e-05})$ & $\num{2.35e-05}$ $(\num{9.87e-06})$ & $\num{8.61e-06}$ $(\num{5.70e-06})$ \\ \hline    
        \multirow{2}{*}{$\overline{{\tilde{\varepsilon}}}^{\gamma, r}_0$} & $\num{1.00e+00}$ $(\num{3.23e-03})$ & $\num{1.00e+00}$ $(\num{4.33e-04})$ & $\num{1.00e+00}$ $(\num{1.19e-03})$ & $\num{1.00e+00}$ $(\num{2.62e-03})$ \\
        & $\mathbf{\num{6.89e-04}}$ $(\mathbf{\num{6.98e-05}})$ & $\mathbf{\num{8.66e-06}}$ $(\mathbf{\num{3.51e-06}})$ & $\mathbf{\num{7.33e-06}}$ $(\mathbf{\num{3.97e-06}})$ & $\mathbf{\num{7.28e-06}}$ $(\mathbf{\num{5.97e-06}})$ \\ \hline    
        \multirow{2}{*}{$\overline{\tau}$} & $\num{1.02e+03}$ & $\num{3.22e+03}$ & $\num{1.74e+04}$ & $\num{5.10e+04}$ \\
         & $\num{8.11e+02}$ & $\num{2.68e+03}$ & $\num{1.54e+04}$ & $\num{4.61e+04}$ \\ \hline   
       \end{tabular}}
    \caption{$d=10$.}
    \label{tab1b} 
    \end{subtable}
     \\
    \begin{subtable}[h]{\textwidth}
      \resizebox{\textwidth}{!}{\begin{tabular}{| c | c | c | c | c |}
      \hline
        \multirow{3}{*}{Metric} & N = 2 & N = 8 & N = 32 & N = 64\\ \
        & DBDP & DBDP & DBDP & DBDP\\ 
        & DLBDP & DLBDP & DLBDP & DLBDP\\ \hline
        \multirow{2}{*}{$\overline{{\tilde{\varepsilon}}}^{y, r}_0$} & $\num{5.40e-03}$ $(\num{3.12e-04})$ & $\num{4.30e-04}$ $(\num{9.66e-05})$ & $\num{4.77e-05}$ $(\num{3.65e-05})$ & $\num{1.68e-05}$ $(\num{1.35e-05})$ \\
        & $\num{2.01e-05}$ $(\num{1.35e-05})$ & $\num{7.54e-06}$ $(\num{7.04e-06})$ & $\num{5.58e-06}$ $(\num{5.87e-06})$ & $\num{2.32e-06}$ $(\num{2.21e-06})$ \\ \hline    
        \multirow{2}{*}{$\overline{{\tilde{\varepsilon}}}^{z, r}_0$} & $\num{5.74e-02}$ $(\num{1.41e-03})$ & $\num{4.17e-03}$ $(\num{3.41e-04})$ & $\num{2.69e-04}$ $(\num{6.50e-05})$ & $\num{8.41e-05}$ $(\num{2.57e-05})$ \\
        & $\num{2.28e-02}$ $(\num{3.59e-04})$ & $\num{1.50e-03}$ $(\num{5.84e-05})$ & $\num{1.01e-04}$ $(\num{2.35e-05})$ & $\num{2.37e-05}$ $(\num{9.35e-06})$ \\ \hline    
        \multirow{2}{*}{$\overline{{\tilde{\varepsilon}}}^{\gamma, r}_0$} & $\num{1.00e+00}$ $(\num{5.22e-05})$ & $\num{1.00e+00}$ $(\num{2.07e-04})$ & $\num{1.00e+00}$ $(\num{1.35e-04})$ & $\num{1.00e+00}$ $(\num{2.12e-04})$ \\
        & $\mathbf{\num{6.17e-02}}$ $(\mathbf{\num{1.85e-03}})$ & $\mathbf{\num{1.34e-03}}$ $(\mathbf{\num{2.33e-04}})$ & $\mathbf{\num{7.06e-05}}$ $(\mathbf{\num{6.49e-05}})$ & $\mathbf{\num{5.24e-05}}$ $(\mathbf{\num{5.97e-05}})$ \\ \hline    
        \multirow{2}{*}{$\overline{\tau}$} & $\num{1.22e+03}$ & $\num{5.46e+03}$ & $\num{5.15e+04}$ & $\num{1.82e+05}$ \\
         & $\num{1.03e+03}$ & $\num{4.88e+03}$ & $\num{4.99e+04}$ & $\num{1.78e+05}$ \\ \hline   
       \end{tabular}}
    \caption{$d=50$.}
    \label{tab1c} 
    \end{subtable}
    \caption{Mean relative MSE values of $\left(Y_0, Z_0, \Gamma_0 \right)$ from DBDP and DLBDP schemes and their average runtimes in Example~\ref{ex1} for $d \in \{1, 10, 50\}$ and $N \in \{2, 8, 32, 64\}$. The STD of the relative MSE values at $t_0$ is given in the brackets.}
    \label{tab1} 
\end{table}
The mean relative MSE of $\left(Y_0, Z_0\right)$ decreases as $N$ increases for each dimension in both schemes. This trend is also observed for $\Gamma_0$ in our scheme, but not in the DBDP scheme, which actually diverges. Note that the mean relative MSE values start to flatten out for $N =64$, indicating that the overall contribution of the approximation error from the DNNs increases for higher $N$ and becomes larger than the discretization error. This is consistent with the error analysis in Section~\ref{sec5} (see~\cite{hure2020deep}, Theorem 4.1 for the DBDP scheme). Moreover, our scheme consistently yields the smallest mean relative MSE for each process, especially as the dimension increases. The average computation time of the DLBDP algorithm is shorter compared to that of the DBDP algorithm. This is because the AD used to approximate the process $\Gamma$ in the DBDP scheme incurs higher computational costs than its approximation via a DNN in the DLBDP scheme. Note that we compare the computational time of both schemes including the computation of $\Gamma$ at each optimization step. In~\cite{negyesi2024one} it is mentioned that the runtime of their algorithm is roughly double of the DBDP one, as it requires solving two optimization problems per discrete time step. Hence, one can reasonably infer that our algorithm may be at least twice as fast as the one proposed in~\cite{negyesi2024one}.

To train the algorithms, we set a high number of optimization steps (and a high number of hidden neurons) as described in Section~\ref{subsec61} such that the same hyperparameters are used for each example. However, the computation time of the algorithms can be reduced, e.g., by reducing the  number of optimization steps. This can be seen in Figure~\ref{fig3}, where we display the mean loss and MSE values of each process for both the algorithms using a validation sample $B=1024$, at discrete time points $\left(t_{32}, t_{63} \right)$ in case of $d=50$ and using $N =64$. The mean loss is defined as $\overline{\mathbf{L}}_n^{\Delta}\left( \hat{\theta}_n \right) := \frac{1}{Q} \sum_{q=1}^{Q} \mathbf{L}_{n, q}^{\Delta}\left( \hat{\theta}_n \right)$. The STD of the loss and MSE values is given in the shaded area.
\begin{figure}[tbhp]
	\centering
	\subfloat[Loss, $t_{63} = 0.9844$.]{
        \pgfplotstableread{"Figures/Example1/d50/L_valid_t_0.9844.dat"}{\table}        
        \begin{tikzpicture} 
            \begin{axis}[
                xmin = 0, xmax = 24000,
                ymin = 5e-4, ymax = 1e-1,
                xtick distance = 4000,
                ymode=log, 
                grid = both,
                width = 0.48\textwidth,
                height = 0.3\textwidth,
                xlabel = {$\Kf$},
                legend cell align = {left},
                legend pos = north east,
                legend style={nodes={scale=0.5, transform shape}}]
                \addplot[smooth, ultra thick, blue, dashed] table [x = 0, y =1] {\table}; 
                \addplot[smooth, ultra thick, green, solid] table [x = 0, y =4] {\table}; 
                \addplot [name path=upper,draw=none] table[x=0,y=2] {\table};
                \addplot [name path=lower,draw=none] table[x=0,y=3] {\table};
                \addplot [fill=blue!40] fill between[of=upper and lower];           
                \addplot [name path=upper,draw=none] table[x=0,y=5] {\table};
                \addplot [name path=lower,draw=none] table[x=0,y=6] {\table};
                \addplot [fill=green!40] fill between[of=upper and lower];                
                \legend{
                   $\overline{\mathbf{L}}_n^{\Delta}\left( \hat{\theta}_n \right)$-DBDP,
                   $\overline{\mathbf{L}}_n^{\Delta}\left( \hat{\theta}_n \right)$-DLBDP
                }
            \end{axis}
            \end{tikzpicture}
		\label{fig3a}
	} 
	\subfloat[Loss, $t_{32}=0.5000$.]{ \pgfplotstableread{"Figures/Example1/d50/L_valid_t_0.5000.dat"}{\table}
        \begin{tikzpicture} 
            \begin{axis}[
                xmin = 0, xmax = 10000,
                ymin = 3e-6, ymax = 1e-2,
                xtick distance = 2000,
                ymode=log,
                grid = both,
                width = 0.48\textwidth,
                height = 0.3\textwidth,
                xlabel = {$\Kf$},
                legend cell align = {left},
                legend pos = north east,
                legend style={nodes={scale=0.5, transform shape}}]
                \addplot[smooth, ultra thick, blue, dashed] table [x = 0, y =1] {\table}; 
                \addplot[smooth, ultra thick, green, solid] table [x = 0, y =4] {\table}; 
                \addplot [name path=upper,draw=none] table[x=0,y=2] {\table};
                \addplot [name path=lower,draw=none] table[x=0,y=3] {\table};
                \addplot [fill=blue!40] fill between[of=upper and lower];           
                \addplot [name path=upper,draw=none] table[x=0,y=5] {\table};
                \addplot [name path=lower,draw=none] table[x=0,y=6] {\table};
                \addplot [fill=green!40] fill between[of=upper and lower];                
                \legend{
                   $\overline{\mathbf{L}}_n^{\Delta}\left( \hat{\theta}_n \right)$-DBDP,
                   $\overline{\mathbf{L}}_n^{\Delta}\left( \hat{\theta}_n \right)$-DLBDP
                }
            \end{axis}
            \end{tikzpicture}
		\label{fig3b}
	} 
    \hfill
    \subfloat[Process $Y$, $t_{63}=0.9844$.]{
        \pgfplotstableread{"Figures/Example1/d50/moments_vareps_y_valid_t_0.9844.dat"}{\table}        
        \begin{tikzpicture} 
            \begin{axis}[
                xmin = 0, xmax = 24000,
                ymin = 8e-6, ymax = 1e-1,
                xtick distance = 4000,
                ymode=log, 
                grid = both,
                width = 0.48\textwidth,
                height = 0.3\textwidth,
                xlabel = {$\Kf$},
                legend cell align = {left},
                legend pos = north east,
                legend style={nodes={scale=0.5, transform shape}}]
                \addplot[smooth, ultra thick, blue, dashed] table [x = 0, y =1] {\table}; 
                \addplot[smooth, ultra thick, green, solid] table [x = 0, y =4] {\table}; 
                \addplot [name path=upper,draw=none] table[x=0,y=2] {\table};
                \addplot [name path=lower,draw=none] table[x=0,y=3] {\table};
                \addplot [fill=blue!40] fill between[of=upper and lower];           
                \addplot [name path=upper,draw=none] table[x=0,y=5] {\table};
                \addplot [name path=lower,draw=none] table[x=0,y=6] {\table};
                \addplot [fill=green!40] fill between[of=upper and lower];                
                \legend{
                   $\overline{{\tilde{\varepsilon}}}^{y}_n$-DBDP,
                   $\overline{{\tilde{\varepsilon}}}^{y}_n$-DLBDP
                }
            \end{axis}
            \end{tikzpicture}
		\label{fig3c}
	} 
	\subfloat[Process $Y$, $t_{32}=0.5000$.]{ \pgfplotstableread{"Figures/Example1/d50/moments_vareps_y_valid_t_0.5000.dat"}{\table}
        \begin{tikzpicture} 
            \begin{axis}[
                xmin = 0, xmax = 10000,
                ymin = 4e-6, ymax = 1e-3,
                xtick distance = 2000,
                ymode=log,
                grid = both,
                width = 0.48\textwidth,
                height = 0.3\textwidth,
                xlabel = {$\Kf$},
                legend cell align = {left},
                legend pos = north east,
                legend style={nodes={scale=0.5, transform shape}}]
                \addplot[smooth, ultra thick, blue, dashed] table [x = 0, y =1] {\table}; 
                \addplot[smooth, ultra thick, green, solid] table [x = 0, y =4] {\table}; 
                \addplot [name path=upper,draw=none] table[x=0,y=2] {\table};
                \addplot [name path=lower,draw=none] table[x=0,y=3] {\table};
                \addplot [fill=blue!40] fill between[of=upper and lower];           
                \addplot [name path=upper,draw=none] table[x=0,y=5] {\table};
                \addplot [name path=lower,draw=none] table[x=0,y=6] {\table};
                \addplot [fill=green!40] fill between[of=upper and lower];                
                \legend{
                   $\overline{{\tilde{\varepsilon}}}^{y}_n$-DBDP,
                   $\overline{{\tilde{\varepsilon}}}^{y}_n$-DLBDP
                }
            \end{axis}
            \end{tikzpicture}
		\label{fig3d}
	} 
    \hfill
    \subfloat[Process $Z$, $t_{63}=0.9844$.]{
        \pgfplotstableread{"Figures/Example1/d50/moments_vareps_z_valid_t_0.9844.dat"}{\table}        
        \begin{tikzpicture} 
            \begin{axis}[
                xmin = 0, xmax = 24000,
                ymin = 1e-4, ymax = 1e+0,
                xtick distance = 4000,
                ymode=log, 
                grid = both,
                width = 0.48\textwidth,
                height = 0.3\textwidth,
                xlabel = {$\Kf$},
                legend cell align = {left},
                legend pos = north east,
                legend style={nodes={scale=0.5, transform shape}}]
                \addplot[smooth, ultra thick, blue, dashed] table [x = 0, y =1] {\table}; 
                \addplot[smooth, ultra thick, green, solid] table [x = 0, y =4] {\table}; 
                \addplot [name path=upper,draw=none] table[x=0,y=2] {\table};
                \addplot [name path=lower,draw=none] table[x=0,y=3] {\table};
                \addplot [fill=blue!40] fill between[of=upper and lower];           
                \addplot [name path=upper,draw=none] table[x=0,y=5] {\table};
                \addplot [name path=lower,draw=none] table[x=0,y=6] {\table};
                \addplot [fill=green!40] fill between[of=upper and lower];                
                \legend{
                   $\overline{{\tilde{\varepsilon}}}^{z}_n$-DBDP,
                   $\overline{{\tilde{\varepsilon}}}^{z}_n$-DLBDP
                }
            \end{axis}
            \end{tikzpicture}
		\label{fig3e}
	} 
	\subfloat[Process $Z$, $t_{32}=0.5000$.]{ \pgfplotstableread{"Figures/Example1/d50/moments_vareps_z_valid_t_0.5000.dat"}{\table}
        \begin{tikzpicture} 
            \begin{axis}[
                xmin = 0, xmax = 10000,
                ymin = 2e-5, ymax = 5e-3,
                xtick distance = 2000,
                ymode=log,
                grid = both,
                width = 0.48\textwidth,
                height = 0.3\textwidth,
                xlabel = {$\Kf$},
                legend cell align = {left},
                legend pos = north east,
                legend style={nodes={scale=0.5, transform shape}}]
                \addplot[smooth, ultra thick, blue, dashed] table [x = 0, y =1] {\table}; 
                \addplot[smooth, ultra thick, green, solid] table [x = 0, y =4] {\table}; 
                \addplot [name path=upper,draw=none] table[x=0,y=2] {\table};
                \addplot [name path=lower,draw=none] table[x=0,y=3] {\table};
                \addplot [fill=blue!40] fill between[of=upper and lower];           
                \addplot [name path=upper,draw=none] table[x=0,y=5] {\table};
                \addplot [name path=lower,draw=none] table[x=0,y=6] {\table};
                \addplot [fill=green!40] fill between[of=upper and lower];                
                \legend{
                   $\overline{{\tilde{\varepsilon}}}^{z}_n$-DBDP,
                   $\overline{{\tilde{\varepsilon}}}^{z}_n$-DLBDP
                }
            \end{axis}
            \end{tikzpicture}
		\label{fig3f}
	} 
    \hfill
    \subfloat[Process $\Gamma$, $t_{63}=0.9844$.]{
        \pgfplotstableread{"Figures/Example1/d50/moments_vareps_gamma_valid_t_0.9844.dat"}{\table}        
        \begin{tikzpicture} 
            \begin{axis}[
                xmin = 0, xmax = 24000,
                ymin = 4e-5, ymax = 5e-2,
                xtick distance = 4000,
                ymode=log, 
                grid = both,
                width = 0.48\textwidth,
                height = 0.3\textwidth,
                xlabel = {$\Kf$},
                legend cell align = {left},
                legend pos = north east,
                legend style={nodes={scale=0.5, transform shape}}]
                \addplot[smooth, ultra thick, blue, dashed] table [x = 0, y =1] {\table}; 
                \addplot[smooth, ultra thick, green, solid] table [x = 0, y =4] {\table}; 
                \addplot [name path=upper,draw=none] table[x=0,y=2] {\table};
                \addplot [name path=lower,draw=none] table[x=0,y=3] {\table};
                \addplot [fill=blue!40] fill between[of=upper and lower];           
                \addplot [name path=upper,draw=none] table[x=0,y=5] {\table};
                \addplot [name path=lower,draw=none] table[x=0,y=6] {\table};
                \addplot [fill=green!40] fill between[of=upper and lower];                
                \legend{
                   $\overline{{\tilde{\varepsilon}}}^{\gamma}_n$-DBDP,
                   $\overline{{\tilde{\varepsilon}}}^{\gamma}_n$-DLBDP
                }
            \end{axis}
            \end{tikzpicture}
		\label{fig3g}
	} 
	\subfloat[Process $\Gamma$, $t_{32}=0.5000$.]{ \pgfplotstableread{"Figures/Example1/d50/moments_vareps_gamma_valid_t_0.5000.dat"}{\table}
        \begin{tikzpicture} 
            \begin{axis}[
                xmin = 0, xmax = 10000,
                ymin = 2e-7, ymax = 1e-4,
                xtick distance = 2000,
                ymode=log,
                grid = both,
                width = 0.48\textwidth,
                height = 0.3\textwidth,
                xlabel = {$\Kf$},
                legend cell align = {left},
                legend pos = north east,
                legend style={nodes={scale=0.5, transform shape}}]
                \addplot[smooth, ultra thick, blue, dashed] table [x = 0, y =1] {\table}; 
                \addplot[smooth, ultra thick, green, solid] table [x = 0, y =4] {\table}; 
                \addplot [name path=upper,draw=none] table[x=0,y=2] {\table};
                \addplot [name path=lower,draw=none] table[x=0,y=3] {\table};
                \addplot [fill=blue!40] fill between[of=upper and lower];           
                \addplot [name path=upper,draw=none] table[x=0,y=5] {\table};
                \addplot [name path=lower,draw=none] table[x=0,y=6] {\table};
                \addplot [fill=green!40] fill between[of=upper and lower];                
                \legend{
                   $\overline{{\tilde{\varepsilon}}}^{\gamma}_n$-DBDP,
                   $\overline{{\tilde{\varepsilon}}}^{\gamma}_n$-DLBDP
                }
            \end{axis}
            \end{tikzpicture}
		\label{fig3h}
	} 
    \caption{Mean loss and MSE values of the process $\left(Y, Z, \Gamma \right)$ from DBDP and DLBDP schemes at discrete time points $\left(t_{32}, t_{63} \right) = \left(0.5000, 0.9844\right)$ using the validation sample in Example~\ref{ex1}, for $d=50$ and $N=64$. The STD of the loss and MSE values is given in the shaded area.}
\label{fig3}
\end{figure}
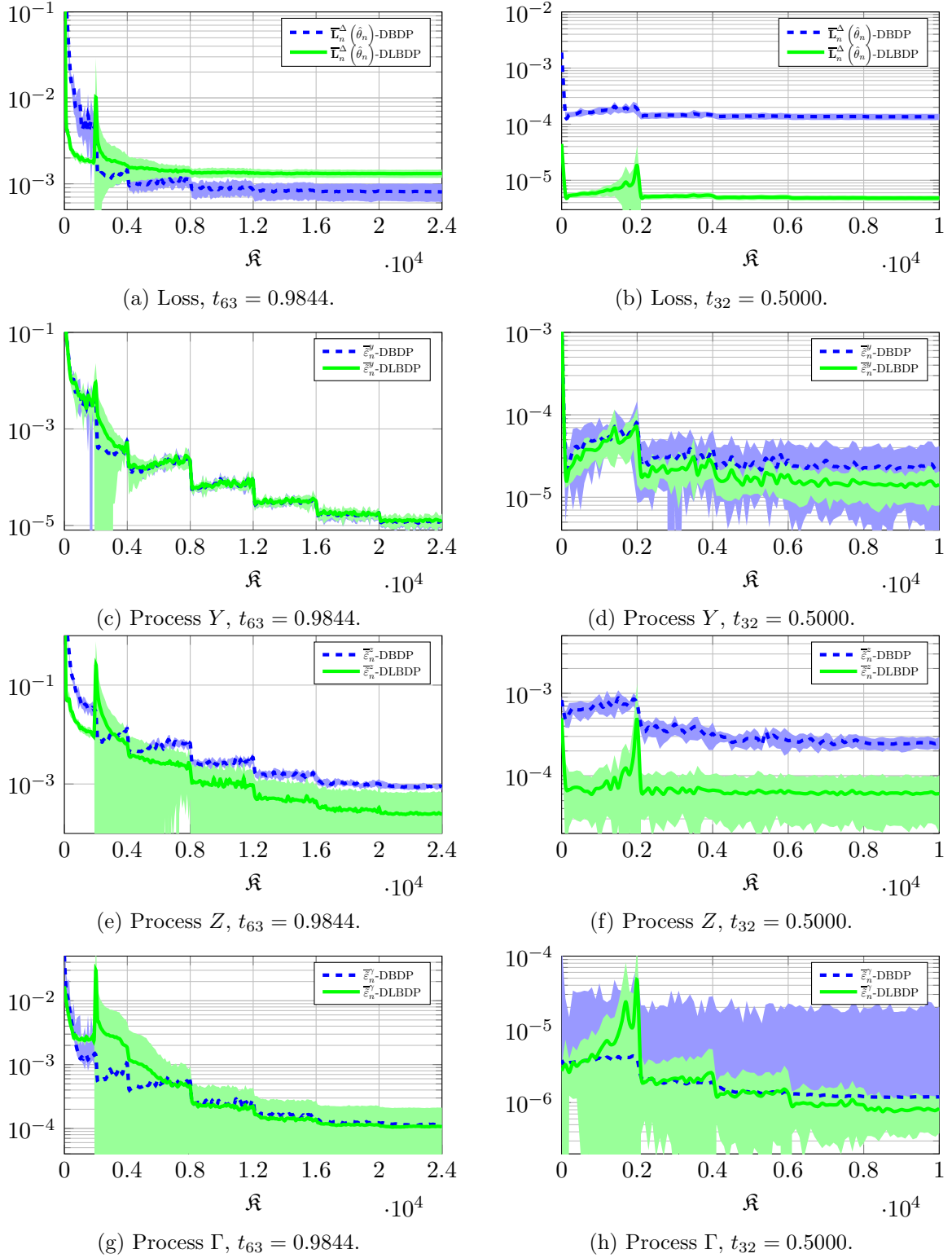
By choosing for instance $\Kf = 16000$ at $t_{63}$ and $\Kf = 5000$ at other discrete time points, the runtime of the algorithms is substantially reduced with almost an insignificant loss of accuracy.

\subsection{Option pricing with different interest rates }
\label{subsec63}
We now consider a pricing problem involving a European option in a financial market where the different interest rates for borrowing and lending are different. This model, originally introduced in~\cite{bergman1995option}, and has been addressed in e.g.,~\cite{weinan2017deep,weinan2019multilevel,teng2021review, teng2022gradient} is represented by a nonlinear BSDE.
\begin{example}
The high-dimensional nonlinear BSDE for pricing European options with different interest rates reads
\begin{equation*}
    \begin{split}
        \left\{
            \begin{array}{rcl}
                dX_t & = & a X_t\,dt + b X_t\, dW_t,\\
                X_0 & = & x_0,\\
                -dY_t & = & \left(-R_1Y_t - \frac{a - R_1}{b} \sum_{k=1}^{d}Z_t^k+ \left( R_2 - R_1\right) \max \left( \frac{1}{b} \sum_{k=1}^{d}Z_t^k - Y_t , 0 \right)\right)\,dt\\
                & & -Z_t \,dW_t,\\  
   	   	    	Y_T & = & g(X_T),
            \end{array}
        \right.
    \end{split}
\end{equation*}
\label{ex2}
\end{example}
where $R_1$ and $R_2$ are the interest rates for lending and borrowing, respectively, and $g(x)$ is the payoff function. Note that instead of solving the above BSDE directly, we solve the transformed BSDE in the ln-domain.

In the case of $d=1$, we consider an European call option with $g(X_T)=\left(X_T-K\right)^+$. This setting agrees with the setting in~\cite{gobet2005regression} (Section 6.3.1), where it is noted that solving the above nonlinear BSDE is the same as solving the linear BSDE in Example~\ref{ex1} with $R = R_2$. This is a good example to compare the approximation of the process $\left(Y, Z, \Gamma \right)$ in case of a nonlinear BSDE from both algorithms with the exact solution (given in~\eqref{eq55}) on the entire discrete time domain. We set $T=0.5$, $K = 100$, $x_0 = 100$, $a = 0.06$, $b = 0.2$, $R_1 = 0.04$ and $R_2 = 0.06$. In Figure~\ref{fig4}, we display the mean MSE values for each process over discrete domain $\Delta$ using the testing sample and $N=64$, where the STD of the MSE values is visualized in the shaded area.
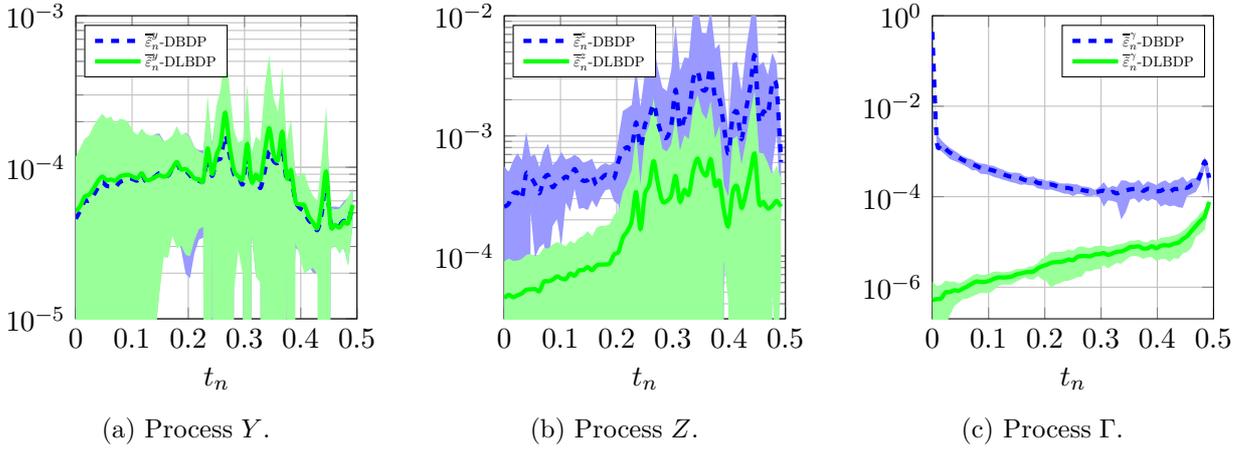
\begin{figure}[tbhp]
	\centering
	\subfloat[Process $Y$.]{
        \pgfplotstableread{"Figures/Example2/d1/vareps_y_over_t.dat"}{\table}        
        \begin{tikzpicture} 
            \begin{axis}[
                xmin = 0, xmax = 0.5,
                ymin = 1e-5, ymax = 1e-3,
                xtick distance = 0.1,
                ymode=log, 
                grid = both,
                width = 0.33\textwidth,
                height = 0.35\textwidth,
                xlabel = {$t_n$},
                legend cell align = {left},
                legend pos = north west,
                legend style={nodes={scale=0.5, transform shape}}]
                \addplot[smooth, ultra thick, blue, dashed] table [x = 0, y =1] {\table}; 
                \addplot[smooth, ultra thick, green, solid] table [x = 0, y =4] {\table}; 
                \addplot [name path=upper,draw=none] table[x=0,y=2] {\table};
                \addplot [name path=lower,draw=none] table[x=0,y=3] {\table};
                \addplot [fill=blue!40] fill between[of=upper and lower];           
                \addplot [name path=upper,draw=none] table[x=0,y=5] {\table};
                \addplot [name path=lower,draw=none] table[x=0,y=6] {\table};
                \addplot [fill=green!40] fill between[of=upper and lower];                
                \legend{
                   $\overline{{\tilde{\varepsilon}}}^{y}_n$-DBDP,
                   $\overline{{\tilde{\varepsilon}}}^{y}_n$-DLBDP
                }
            \end{axis}
            \end{tikzpicture}
		\label{fig4a}
	} 
	\subfloat[Process $Z$.]{ \pgfplotstableread{"Figures/Example2/d1/vareps_z_over_t.dat"}{\table}
        \begin{tikzpicture} 
            \begin{axis}[
                xmin = 0, xmax = 0.5,
                ymin = 3e-5, ymax = 1e-2,
                xtick distance = 0.1,
                ymode=log,
                grid = both,
                width = 0.33\textwidth,
                height = 0.35\textwidth,
                xlabel = {$t_n$},
                legend cell align = {left},
                legend pos = north west,
                legend style={nodes={scale=0.5, transform shape}}]
                \addplot[smooth, ultra thick, blue, dashed] table [x = 0, y =1] {\table}; 
                \addplot[smooth, ultra thick, green, solid] table [x = 0, y =4] {\table}; 
                \addplot [name path=upper,draw=none] table[x=0,y=2] {\table};
                \addplot [name path=lower,draw=none] table[x=0,y=3] {\table};
                \addplot [fill=blue!40] fill between[of=upper and lower];           
                \addplot [name path=upper,draw=none] table[x=0,y=5] {\table};
                \addplot [name path=lower,draw=none] table[x=0,y=6] {\table};
                \addplot [fill=green!40] fill between[of=upper and lower];                
                \legend{
                   $\overline{{\tilde{\varepsilon}}}^{z}_n$-DBDP,
                   $\overline{{\tilde{\varepsilon}}}^{z}_n$-DLBDP
                }
            \end{axis}
            \end{tikzpicture}
		\label{fig4b}
	} 
        \subfloat[Process $\Gamma$.]{ \pgfplotstableread{"Figures/Example2/d1/vareps_gamma_over_t.dat"}{\table}
        \begin{tikzpicture} 
            \begin{axis}[
                xmin = 0, xmax = 0.5,
                ymin = 2e-7, ymax = 1,
                xtick distance = 0.1,
                ymode=log,
                grid = both,
                width = 0.33\textwidth,
                height = 0.35\textwidth,
                xlabel = {$t_n$},
                legend cell align = {left},
                legend pos = north east,
                legend style={nodes={scale=0.5, transform shape}}]
                \addplot[smooth, ultra thick, blue, dashed] table [x = 0, y =1] {\table}; 
                \addplot[smooth, ultra thick, green, solid] table [x = 0, y =4] {\table}; 
                \addplot [name path=upper,draw=none] table[x=0,y=2] {\table};
                \addplot [name path=lower,draw=none] table[x=0,y=3] {\table};
                \addplot [fill=blue!40] fill between[of=upper and lower];           
                \addplot [name path=upper,draw=none] table[x=0,y=5] {\table};
                \addplot [name path=lower,draw=none] table[x=0,y=6] {\table};
                \addplot [fill=green!40] fill between[of=upper and lower];                
                \legend{
                   $\overline{{\tilde{\varepsilon}}}^{\gamma}_n$-DBDP,
                   $\overline{{\tilde{\varepsilon}}}^{\gamma}_n$-DLBDP
                }
            \end{axis}
            \end{tikzpicture}
		\label{fig4c}
	} 
    \caption{Mean MSE values of the processes $\left(Y, Z, \Gamma \right)$ from DBDP and DLBDP schemes over the discrete time points $\{t_n\}_{n=0}^{N-1}$ using the testing sample in Example~\ref{ex2}, for $d=1$ and $N = 64$. The STD of MSE values is given in the shaded area.}
\label{fig4}
\end{figure}
We see that our scheme outperforms the DBDP scheme in approximating the processes $\left(Z, \Gamma \right)$ on the entire discrete domain $\Delta$, similarly as in Example~\ref{ex1}. In Table~\ref{tab2}, we report the mean relative MSE values at $t_0$ for each process and the algorithm average runtime using $N \in \{ 2, 8, 32, 64\}$.  The STD of the relative MSE values at $t_0$ is given in the brackets.
\begin{table}[htb!]
{\footnotesize
\begin{center}
  \begin{tabular}{| c | c | c | c | c |}
  \hline
    \multirow{3}{*}{Metric} & N = 2 & N = 8 & N = 32 & N = 64\\ \
    & DBDP & DBDP & DBDP & DBDP\\ 
    & DLBDP & DLBDP & DLBDP & DLBDP\\ \hline
        \multirow{2}{*}{$\overline{{\tilde{\varepsilon}}}^{y, r}_0$} & $\num{6.23e-06}$ $(\num{6.59e-06})$ & $\num{2.62e-06}$ $(\num{2.65e-06})$ & $\num{2.93e-06}$ $(\num{3.37e-06})$ & $\num{8.93e-07}$ $(\num{8.21e-07})$ \\
        & $\num{6.04e-06}$ $(\num{6.47e-06})$ & $\num{2.67e-06}$ $(\num{2.68e-06})$ & $\num{3.07e-06}$ $(\num{3.62e-06})$ & $\num{1.00e-06}$ $(\num{1.29e-06})$ \\ \hline    
        \multirow{2}{*}{$\overline{{\tilde{\varepsilon}}}^{z, r}_0$} & $\num{2.01e-04}$ $(\num{7.06e-05})$ & $\num{1.53e-05}$ $(\num{8.48e-06})$ & $\num{2.95e-06}$ $(\num{3.21e-06})$ & $\num{1.72e-06}$ $(\num{2.02e-06})$ \\
        & $\num{2.83e-05}$ $(\num{1.78e-05})$ & $\num{1.42e-06}$ $(\num{8.28e-07})$ & $\num{6.81e-07}$ $(\num{7.68e-07})$ & $\num{2.97e-07}$ $(\num{3.05e-07})$ \\ \hline    
        \multirow{2}{*}{$\overline{{\tilde{\varepsilon}}}^{\gamma, r}_0$} & $\num{1.09e+00}$ $(\num{1.63e-02})$ & $\num{9.96e-01}$ $(\num{7.92e-04})$ & $\num{9.92e-01}$ $(\num{7.82e-03})$ & $\num{9.96e-01}$ $(\num{2.26e-03})$ \\
        & $\mathbf{\num{3.22e-04}}$ $(\mathbf{\num{3.61e-05}})$ & $\mathbf{\num{2.56e-05}}$ $(\mathbf{\num{1.01e-05}})$ & $\mathbf{\num{1.26e-06}}$ $(\mathbf{\num{1.27e-06}})$ & $\mathbf{\num{1.18e-06}}$ $(\mathbf{\num{1.86e-06}})$ \\ \hline    
        \multirow{2}{*}{$\overline{\tau}$} & $\num{8.08e+02}$ & $\num{2.33e+03}$ & $\num{9.09e+03}$ & $\num{1.96e+04}$ \\
         & $\num{5.81e+02}$ & $\num{1.72e+03}$ & $\num{6.91e+03}$ & $\num{1.53e+04}$ \\ \hline   
   \end{tabular}
  \end{center}
\caption{Mean relative MSE values of $\left(Y_0, Z_0, \Gamma_0 \right)$ from DBDP and DLBDP schemes and their average runtimes in Example~\ref{ex2} for $d=1$ and $N \in \{2, 8, 32, 64\}$. The STD of the relative MSE values at $t_0$ is given in the brackets.}
\label{tab2} 
}
\end{table}
The same conclusion can be drawn that the DLBDP scheme yields convergent results for $N \in \{2, 8, 32, 64\}$ for each process (whereas DBDP diverges for $\Gamma_0$) and outperforms the DBDP scheme. Additionally, our scheme exhibits smaller runtimes, similarly to Example~\ref{ex1}.

Next, we test both schemes in the case of $d=50$, using the payoff function 
$$Y_T = \max \left(\max_{k = 1, \ldots, d} (X_{T}^k - K_1, 0\right)-2\max \left(\max_{k = 1, \ldots, d} (X_{T}^k - K_2, 0\right),$$
where $K_1=120$ and $K_2 = 150$. The benchmark value is $Y_0 \doteq 17.9743$, which is computed using the multilevel Monte Carlo approach~\cite{weinan2019multilevel} with 7 Picard iterations and $Q=10$ independent runs. For $N \in \{2, 8, 32, 64\}$, we show in Table~\ref{tab3} the approximation for $Y_0$ (the reference results for $Z_0$ are not available) from both algorithms and their average runtime. More precisely, we report the mean approximation of $Y_0$ defined as $\overline{Y}_0^{\Delta, \hat{\theta}} := \frac{1}{Q} \sum_{q=1}^Q Y_{0,q}^{\Delta, \hat{\theta}}$, the mean relative MSE and their STD given in the brackets.
\begin{table}[htb!]
{\footnotesize
\begin{center}
  \begin{tabular}{| c | c | c | c | c |}
  \hline
    \multirow{3}{*}{Metric} & N = 2 & N = 8 & N = 32 & N = 64\\ \
    & DBDP & DBDP & DBDP & DBDP\\ 
    & DLBDP & DLBDP & DLBDP & DLBDP\\ \hline
        $Y_0$~\cite{weinan2019multilevel} & \multicolumn{4}{c|}{$17.9743$} \\ \hline
        \multirow{2}{*}{$\overline{Y}_0^{\Delta, \hat{\theta}}$} & $17.5165$ $(\num{5.79e-01})$ & $17.5915$ $(\num{8.68e-01})$ & $18.1236$ $(\num{6.73e-01})$ & $17.7246$ $(\num{5.55e-01})$ \\
        & $17.6329$ $(\num{3.65e-01})$ & $17.4472$ $(\num{1.03e+00})$ & $17.9746$ $(\num{1.04e-01})$ & $17.7875$ $(\num{2.70e-01})$\\ \hline    
        \multirow{2}{*}{$\overline{{\tilde{\varepsilon}}}^{y, r}_0$} & $\num{1.69e-03}$ $(\num{3.35e-03})$ & $\num{2.79e-03}$ $(\num{7.40e-03})$ & $\num{1.47e-03}$ $(\num{1.32e-03})$ & $\num{1.14e-03}$ $(\num{1.80e-03})$ \\
        & $\num{7.74e-04}$ $(\num{1.50e-03})$ & $\num{4.12e-03}$ $(\num{1.12e-02})$ & $\num{3.33e-05}$ $(\num{4.38e-05})$ & $\num{3.33e-04}$ $(\num{6.55e-04})$ \\ \hline    
        \multirow{2}{*}{$\overline{\tau}$} & $\num{5.56e+02}$ & $\num{3.54e+03}$ & $\num{4.31e+04}$ & $\num{1.67e+05}$ \\
         & $\num{4.59e+02}$ & $\num{3.23e+03}$ & $\num{4.17e+04}$ & $\num{1.65e+05}$ \\ \hline   
   \end{tabular}
  \end{center}
\caption{Mean approximation of $Y_0$, its mean relative MSE from DBDP and DLBDP schemes and their average runtimes in Example~\ref{ex2} for $d=50$ and $N \in \{2, 8, 32, 64\}$. The STD of the approximations of $Y_0$ and its relative MSE values are given in the brackets.}
\label{tab3} 
}
\end{table}
We observe that our scheme provides higher accurate approximations of $Y_0$ for the $50$-dimensional nonlinear BSDE in Example~\ref{ex2} compared to DBDP scheme, resulting in smaller relative MSE value, for shorter computation time. Note that the mean relative MSE can be further reduced by increasing the number of hidden neurons or layers provided that the optimization error is sufficiently small.

\subsection{The Hamilton-Jacobi-Bellman equation}
\label{subsec64}
The next example is a Hamilton-Jacobi-Bellman (HJB) equation which admits a semi-explicit solution~\cite{weinan2017deep}. In finance, specifically portfolio optimization, solving the HJB equation provides insights into the optimal investment strategy that maximizes expected utility of the investors terminal wealth. Hence, the process $Y$ is related to the wealth of the portfolio and the process $Z$ the holding on each asset.
\begin{example}
The high-dimensional HJB BSDE reads
\begin{equation*}
    \begin{split}
        \left\{
            \begin{array}{rcl}
                dX_t & = & b \, dW_t,\\
                X_0 & = & x_0,\\
                -dY_t & = & -\sum_{k=1}^{d}\left(\frac{Z_t^k}{b}\right)^2\,dt-Z_t \,dW_t,\\  
                Y_T & = & g(X_T).
            \end{array}
        \right.
    \end{split}
\end{equation*}
\label{ex3}
\end{example}
This BSDE admits the semi-explicit solution as given in~\cite{weinan2017deep}
$$
Y_t = u(t, X_t) = -\ln\left( \mathbb{E} \left[ \exp\left( -g( X_t +  b\left( W_T - W_t \right) ) \right) \right] \right).
$$
The semi-explicit solution of $\left( Z_t, \Gamma_t \right)$ is calculated using relations~\eqref{eq17}. Note that it is quite time consuming to approximate highly accurate pathwise reference solutions $\left( Y_t, Z_t, \Gamma_t \right)$ for $t \in [0, T]$. Hence, we only calculate a reference solution at $t_0$. We set $T=0.5$, $d=50$, $X_0 = \mathbf{1}_d$, $b = \sqrt{0.2}$ and $g(x) = \ln\left( \frac{1}{2}\left( 1 + \left\vert x \right\vert^2 \right) \right)$. Using $10^7$ Borwnian motion samples and $50$ independent runs, we calculate the mean approximations of $\left(Y_0, Z_0, \Gamma_0 \right)$ and use as reference values to test the accuracy of the DBDP and DLBDP schemes. In Table~\ref{tab4}, we report the relative MSE values at $t_0$ for each process, the corresponding STD (given in the brackets) and the algorithm average runtime using $N \in \{ 2, 8, 32, 64\}$. 
\begin{table}[htb!]
{\footnotesize
\begin{center}
  \begin{tabular}{| c | c | c | c | c |}
  \hline
    \multirow{3}{*}{Metric} & N = 2 & N = 8 & N = 32 & N = 64\\ \
    & DBDP & DBDP & DBDP & DBDP\\ 
    & DLBDP & DLBDP & DLBDP & DLBDP\\ \hline
        \multirow{2}{*}{$\overline{{\tilde{\varepsilon}}}^{y, r}_0$} & $\num{5.68e-07}$ $(\num{3.14e-07})$ & $\num{2.23e-07}$ $(\num{1.78e-07})$ & $\num{2.22e-07}$ $(\num{1.61e-07})$ & $\num{2.36e-07}$ $(\num{2.83e-07})$ \\
        & $\num{3.77e-07}$ $(\num{3.17e-07})$ & $\num{2.12e-07}$ $(\num{2.05e-07})$ & $\num{8.26e-08}$ $(\num{1.16e-07})$ & $\num{1.43e-07}$ $(\num{1.47e-07})$ \\ \hline    
        \multirow{2}{*}{$\overline{{\tilde{\varepsilon}}}^{z, r}_0$} & $\num{2.57e-04}$ $(\num{6.23e-05})$ & $\num{5.61e-04}$ $(\num{1.66e-04})$ & $\num{6.57e-04}$ $(\num{1.63e-04})$ & $\num{7.39e-04}$ $(\num{1.33e-04})$ \\
        & $\num{6.37e-05}$ $(\num{8.55e-06})$ & $\num{1.04e-04}$ $(\num{1.90e-05})$ & $\num{8.60e-05}$ $(\num{1.74e-05})$ & $\num{9.39e-05}$ $(\num{1.89e-05})$ \\ \hline    
        \multirow{2}{*}{$\overline{{\tilde{\varepsilon}}}^{\gamma, r}_0$} & $\num{9.34e-01}$ $(\num{5.67e-03})$ & $\num{9.70e-01}$ $(\num{1.39e-02})$ & $\num{8.13e-01}$ $(\num{9.51e-03})$ & $\num{8.42e-01}$ $(\num{4.27e-03})$ \\
        & $\mathbf{\num{3.04e-04}}$ $(\mathbf{\num{1.30e-05}})$ & $\mathbf{\num{8.35e-04}}$ $(\mathbf{\num{2.91e-05}})$ & $\mathbf{\num{1.13e-03}}$ $(\mathbf{\num{3.24e-05}})$ & $\mathbf{\num{1.52e-03}}$ $(\mathbf{\num{4.08e-05}})$ \\ \hline    
        \multirow{2}{*}{$\overline{\tau}$} & $\num{4.75e+02}$ & $\num{2.77e+03}$ & $\num{3.03e+04}$ & $\num{1.14e+05}$ \\
         & $\num{3.65e+02}$ & $\num{2.43e+03}$ & $\num{2.88e+04}$ & $\num{1.15e+05}$ \\ \hline   
   \end{tabular}
  \end{center}
\caption{Mean relative MSE values of $\left(Y_0, Z_0, \Gamma_0 \right)$ from DBDP and DLBDP schemes and their average runtimes in Example~\ref{ex3} for $N \in \{2, 8, 32, 64\}$. The STD of the relative MSE values at $t_0$ is given in the brackets.}
\label{tab4} 
}
\end{table}
Our scheme consistently outperforms the DBDP scheme in approximating each process, particularly for $\Gamma_0$, where the DBDP scheme yields high mean relative MSE values.

\subsection{The Black-Scholes extended with local volatility}
\label{subsec65}
Our final example is taken from~\cite{ruijter2016fourier} in order to demonstrate the effectiveness of our scheme in case of a time dependent diffusion function. Consider an European call option as in Example~\ref{ex1}, where each underlying asset follows a GBM with time-dependent drift and diffusion. 
\begin{example}
The high-dimenisonal Black-Scholes BSDE with local volatility reads~\cite{ruijter2016fourier}
\begin{equation*}
    \begin{split}
    \left\{
        \begin{array}{rcl}
             dX_t &=&  a(t) X_t\,dt + b(t) X_t \,dW_t, \\
             X_0 &=& x_0,\\  
           -dY_t &=& - \left( R Y_t + \sum_{k=1}^d  \frac{ a(t) - R + \delta}{b(t)} Z_t^k\right)\,dt- Z_t \,dW_t,\\ 
   		   	   Y_T &=& \left(\prod_{k=1}^d \left(X_T^k\right)^{c_k}-K\right)^+,
        \end{array}
    \right. \\ 
    \end{split}
\end{equation*}
\label{ex4}
\end{example}
 where for $a(t)$ and $b(t)$ we choose the following periodic functions
 \begin{align*}
     a(t) & = a_0 + a_1\sin\left( \frac{2\pi}{C_1} t\right) + a_2\sin\left( \frac{2\pi}{C_2} t\right),\\
    b(t) & = b_0 + b_1\sin\left( \frac{2\pi}{C_1} t\right) + b_2\sin\left( \frac{2\pi}{C_2} t\right).
 \end{align*}
The exact solution of this local volatility model is given by the Black-Scholes formula with volatility parameter $ \bar{b} =  \sqrt{\frac{1}{T-t} \int_{t}^T b(s)^2 ds}$. More precisely, the exact solution is given by~\eqref{eq55} with
$$\check{b} = \sum_{k=1}^d (\bar{b} c_k)^2,\quad \check{\delta} = \sum_{k=1}^d c_k\left(\delta_k + \frac{\bar{b}^2}{2} \right) - \frac{\check{b}^2}{2}, \quad Z_t^k = \frac{\partial u}{\partial x_k} b(t) X_t^k.$$
We apply the ln-transformation in this example, which is similar as in the case of Example~\ref{ex1}. Moreover, we set $T = 0.25$, $d=50$ and the other following parameter values
\begin{align*}
    & X_0 = 100, K = 100, R = 0.1, a_0 = 0.2, a_1 = 0.1, a_2 = 0.02,\\
    & c_k = \frac{1}{d}, \delta = 0, C_1 = 1, C_2 = 0.25, b_0 = 0.25, b_1 = 0.125, b_2 = 0.025. 
\end{align*}
Using $N=32$, the mean MSE values for each process over discrete domain $\Delta$ are visualized in Figure~\ref{fig5} for the testing sample. The STD of the MSE values is displayed in the shaded area.
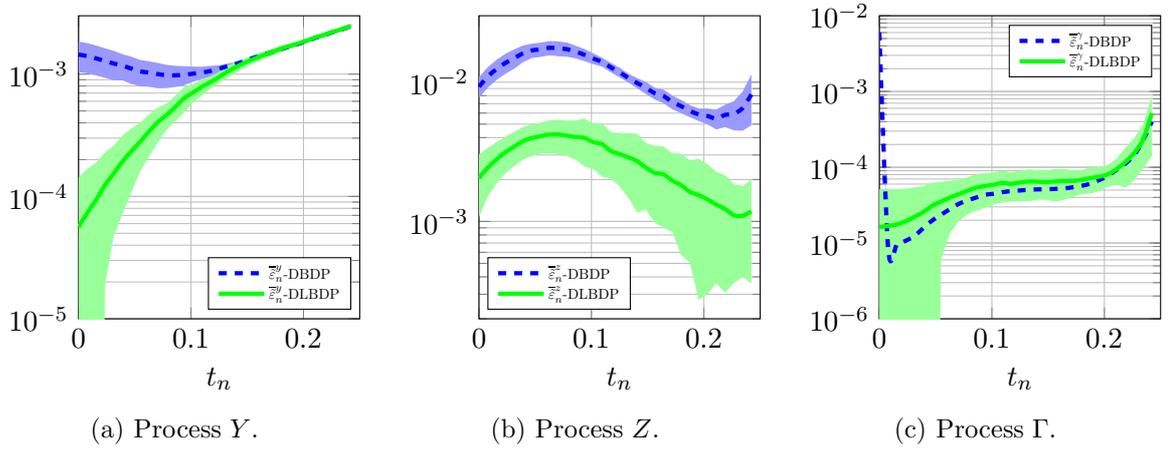
\begin{figure}[tbhp]
	\centering
	\subfloat[Process $Y$.]{
        \pgfplotstableread{"Figures/Example4/d50/vareps_y_over_t.dat"}{\table}        
        \begin{tikzpicture} 
            \begin{axis}[
                xmin = 0, xmax = 0.25,
                ymin = 1e-5, ymax = 3e-3,
                xtick distance = 0.1,
                ymode=log, 
                grid = both,
                width = 0.33\textwidth,
                height = 0.35\textwidth,
                xlabel = {$t_n$},
                legend cell align = {left},
                legend pos = south east,
                legend style={nodes={scale=0.5, transform shape}}]
                \addplot[smooth, ultra thick, blue, dashed] table [x = 0, y =1] {\table}; 
                \addplot[smooth, ultra thick, green, solid] table [x = 0, y =4] {\table}; 
                \addplot [name path=upper,draw=none] table[x=0,y=2] {\table};
                \addplot [name path=lower,draw=none] table[x=0,y=3] {\table};
                \addplot [fill=blue!40] fill between[of=upper and lower];           
                \addplot [name path=upper,draw=none] table[x=0,y=5] {\table};
                \addplot [name path=lower,draw=none] table[x=0,y=6] {\table};
                \addplot [fill=green!40] fill between[of=upper and lower];                
                \legend{
                   $\overline{{\tilde{\varepsilon}}}^{y}_n$-DBDP,
                   $\overline{{\tilde{\varepsilon}}}^{y}_n$-DLBDP
                }
            \end{axis}
            \end{tikzpicture}
		\label{fig5a}
	} 
	\subfloat[Process $Z$.]{ \pgfplotstableread{"Figures/Example4/d50/vareps_z_over_t.dat"}{\table}
        \begin{tikzpicture} 
            \begin{axis}[
                xmin = 0, xmax = 0.25,
                ymin = 2e-4, ymax = 3e-2,
                xtick distance = 0.1,
                ymode=log,
                grid = both,
                width = 0.33\textwidth,
                height = 0.35\textwidth,
                xlabel = {$t_n$},
                legend cell align = {left},
                legend pos = south west,
                legend style={nodes={scale=0.5, transform shape}}]
                \addplot[smooth, ultra thick, blue, dashed] table [x = 0, y =1] {\table}; 
                \addplot[smooth, ultra thick, green, solid] table [x = 0, y =4] {\table}; 
                \addplot [name path=upper,draw=none] table[x=0,y=2] {\table};
                \addplot [name path=lower,draw=none] table[x=0,y=3] {\table};
                \addplot [fill=blue!40] fill between[of=upper and lower];           
                \addplot [name path=upper,draw=none] table[x=0,y=5] {\table};
                \addplot [name path=lower,draw=none] table[x=0,y=6] {\table};
                \addplot [fill=green!40] fill between[of=upper and lower];                
                \legend{
                   $\overline{{\tilde{\varepsilon}}}^{z}_n$-DBDP,
                   $\overline{{\tilde{\varepsilon}}}^{z}_n$-DLBDP
                }
            \end{axis}
            \end{tikzpicture}
		\label{fig5b}
	} 
        \subfloat[Process $\Gamma$.]{ \pgfplotstableread{"Figures/Example4/d50/vareps_gamma_over_t.dat"}{\table}
        \begin{tikzpicture} 
            \begin{axis}[
                xmin = 0, xmax = 0.25,
                ymin = 1e-6, ymax = 1e-2,
                xtick distance = 0.1,
                ymode=log,
                grid = both,
                width = 0.33\textwidth,
                height = 0.35\textwidth,
                xlabel = {$t_n$},
                legend cell align = {left},
                legend pos = north east,
                legend style={nodes={scale=0.5, transform shape}}]
                \addplot[smooth, ultra thick, blue, dashed] table [x = 0, y =1] {\table}; 
                \addplot[smooth, ultra thick, green, solid] table [x = 0, y =4] {\table}; 
                \addplot [name path=upper,draw=none] table[x=0,y=2] {\table};
                \addplot [name path=lower,draw=none] table[x=0,y=3] {\table};
                \addplot [fill=blue!40] fill between[of=upper and lower];           
                \addplot [name path=upper,draw=none] table[x=0,y=5] {\table};
                \addplot [name path=lower,draw=none] table[x=0,y=6] {\table};
                \addplot [fill=green!40] fill between[of=upper and lower];                
                \legend{
                   $\overline{{\tilde{\varepsilon}}}^{\gamma}_n$-DBDP,
                   $\overline{{\tilde{\varepsilon}}}^{\gamma}_n$-DLBDP
                }
            \end{axis}
            \end{tikzpicture}
		\label{fig5c}
	} 
    \caption{Mean MSE values of the processes $\left(Y, Z, \Gamma \right)$ from DBDP and DLBDP schemes over the discrete time points $\{t_n\}_{n=0}^{N-1}$ using the testing sample in Example~\ref{ex4} for $N = 32$. The STD of MSE values is given in the shaded area.}
\label{fig5}
\end{figure}
Compared to previous examples, we notice significant improvements not only in approximating the process $Z$ but also the process $Y$ through our approach. In the case of the process $\Gamma$, such improvements are evident only near $t_0$. In Table~\ref{tab5}, we report the mean relative MSE values at $t_0$ for each process from both schemes, using $N \in \{2, 8, 16, 32\}$. The corresponding STD is given in the brackets. The average runtime of the algorithms are also included.
\begin{table}[htb!]
{\footnotesize
\begin{center}
  \begin{tabular}{| c | c | c | c | c |}
  \hline
    \multirow{3}{*}{Metric} & N = 2 & N = 8 & N = 16 & N = 32\\ \
    & DBDP & DBDP & DBDP & DBDP\\ 
    & DLBDP & DLBDP & DLBDP & DLBDP\\ \hline
        \multirow{2}{*}{$\overline{{\tilde{\varepsilon}}}^{y, r}_0$} & $\num{9.00e-03}$ $(\num{4.33e-04})$ & $\num{5.93e-03}$ $(\num{5.42e-04})$ & $\num{2.05e-03}$ $(\num{2.23e-04})$ & $\num{5.67e-04}$ $(\num{1.59e-04})$ \\
        & $\num{1.89e-04}$ $(\num{5.52e-05})$ & $\num{6.84e-06}$ $(\num{1.01e-05})$ & $\num{3.41e-05}$ $(\num{5.64e-05})$ & $\num{2.22e-05}$ $(\num{3.39e-05})$ \\ \hline    
        \multirow{2}{*}{$\overline{{\tilde{\varepsilon}}}^{z, r}_0$} & $\num{1.77e-01}$ $(\num{1.79e-03})$ & $\num{2.76e-02}$ $(\num{1.48e-03})$ & $\num{6.73e-03}$ $(\num{6.17e-04})$ & $\num{1.58e-03}$ $(\num{2.68e-04})$ \\
        & $\num{1.14e-01}$ $(\num{9.54e-04})$ & $\num{9.58e-03}$ $(\num{5.82e-04})$ & $\num{1.60e-03}$ $(\num{5.51e-04})$ & $\num{3.50e-04}$ $(\num{1.68e-04})$ \\ \hline    
        \multirow{2}{*}{$\overline{{\tilde{\varepsilon}}}^{\gamma, r}_0$} & $\num{1.00e+00}$ $(\num{1.52e-03})$ & $\num{1.00e+00}$ $(\num{1.38e-04})$ & $\num{1.00e+00}$ $(\num{2.16e-04})$ & $\num{1.00e+00}$ $(\num{4.81e-04})$ \\
        & $\mathbf{\num{4.72e-01}}$ $(\mathbf{\num{4.22e-03}})$ & $\mathbf{\num{1.73e-03}}$ $(\mathbf{\num{8.26e-04}})$ & $\mathbf{\num{3.71e-03}}$ $(\mathbf{\num{6.75e-03}})$ & $\mathbf{\num{2.72e-03}}$ $(\mathbf{\num{5.84e-03}})$ \\ \hline    
        \multirow{2}{*}{$\overline{\tau}$} & $\num{1.39e+03}$ & $\num{7.13e+03}$ & $\num{2.18e+04}$ & $\num{8.32e+04}$ \\
         & $\num{1.23e+03}$ & $\num{6.65e+03}$ & $\num{2.09e+04}$ & $\num{8.18e+04}$ \\ \hline   
   \end{tabular}
  \end{center}
\caption{Mean relative MSE values of $\left(Y_0, Z_0, \Gamma_0 \right)$ from DBDP and DLBDP schemes and their average runtimes in Example~\ref{ex4} for $N \in \{2, 8, 16, 32\}$. The STD of the relative MSE values at $t_0$ is given in the brackets.}
\label{tab5} 
}
\end{table}
Our scheme gives the smallest relative MSE values for each process for lower runtime. In this example, the improvement in approximating $Y_0$ is more evident than in previous examples.

\section{Conclusions}
\label{sec7}
In this work, we introduce a novel backward scheme which utilizes the differential deep learning approach to solve high-dimensional nonlinear BSDEs. By applying Malliavin calculus, we transform the BSDEs into a differential deep learning problem. This transformation results in a system of BSDEs that requires the estimation of the solution, its gradient, and the Hessian matrix, given by the triple of processes $\left(Y, Z, \Gamma \right)$ in the BSDE system. To approximate this solution triple, we discretize the integrals within the system using the Euler-Maruyama method and parameterize their discrete version using DNNs. The DNN parameters are iteratively optimized backwardly at each time step by minimizing a differential learning type loss function, constructed as a weighted sum of the dynamics of the discretized BSDE system. An error analysis is conducted to demonstrate the convergence of the proposed algorithm. Our formulation provides additional information to the SGD method to give more accurate approximations compared to deep learning-based approaches, as our loss function includes not only the dynamics of the process $Y$ but also $Z$. The introduced differential deep learning-based approach can be used to other deep learning based schemes, e.g., ~\cite{weinan2017deep,raissi2024forward,kapllani2024deep}. The proficiency of our algorithm in terms of accuracy and computational efficiency is demonstrated through numerous numerical experiments involving pricing and hedging nonlinear options up to $50$ dimensions. The proposed algorithm holds promise for applications in pricing and hedging financial derivatives in high-dimensional settings.

\bibliographystyle{siamplain}
\bibliography{bibfile}

\begin{thebibliography}{10}

\bibitem{abbas2022pathwise}
{\sc L.~Abbas-Turki, S.~Cr{\'e}pey, B.~Saadeddine, and W.~Sabbagh}, {\em {Pathwise XVAs: The Direct Scheme}},  (2022), \url{https://perso.lpsm.paris/~crepey/papers/MABSDE.pdf}.

\bibitem{andersson2023convergence}
{\sc K.~Andersson, A.~Andersson, and C.~W. Oosterlee}, {\em {Convergence of a Robust Deep FBSDE Method for Stochastic Control}}, SIAM J. Sci. Comput., 45 (2023), pp.~A226--A255, \url{https://doi.org/10.1137/22M1478057}.

\bibitem{ANKIRCHNER_2010}
{\sc S.~Ankirchner, C.~Blanchet-Scalliet, and A.~Eyraud-Loisel}, {\em {Credit risk premia and quadratic BSDEs with a single jump}}, Int. J. Theor. Appl. Finance., 13 (2010), pp.~1103--1129, \url{https://doi.org/10.1142/s0219024910006133}.

\bibitem{beck2021deep}
{\sc C.~Beck, S.~Becker, P.~Cheridito, A.~Jentzen, and A.~Neufeld}, {\em {Deep Splitting Method for Parabolic PDEs}}, SIAM J. Sci. Comput., 43 (2021), pp.~A3135--A3154, \url{https://doi.org/10.1137/19M1297919}.

\bibitem{bender2008}
{\sc C.~Bender and J.~Zhang}, {\em {Time discretization and Markovian iteration for coupled FBSDEs}}, Ann. Appl. Probab., 18 (2008), pp.~143--177, \url{https://doi.org/10.1214/07-aap448}.

\bibitem{bergman1995option}
{\sc Y.~Z. Bergman}, {\em {Option Pricing with Differential Interest Rates}}, Rev. Financ. Stud., 8 (1995), pp.~475--500, \url{https://doi.org/10.1093/rfs/8.2.475}.

\bibitem{bouchard2004discrete}
{\sc B.~Bouchard and N.~Touzi}, {\em {Discrete-time approximation and Monte-Carlo simulation of backward stochastic differential equations}}, Stoch. Process Their Appl., 111 (2004), pp.~175--206, \url{https://doi.org/10.1016/j.spa.2004.01.001}.

\bibitem{chassagneux2023learning}
{\sc J.-F. Chassagneux, J.~Chen, N.~Frikha, and C.~Zhou}, {\em {A learning scheme by sparse grids and Picard approximations for semilinear parabolic PDEs}}, IMA J. Numer. Anal., 43 (2023), pp.~3109--3168, \url{https://doi.org/10.1093/imanum/drac066}.

\bibitem{chen2021deep}
{\sc Y.~Chen and J.~W. Wan}, {\em {Deep neural network framework based on backward stochastic differential equations for pricing and hedging American options in high dimensions}}, Quant. Finance, 21 (2021), pp.~45--67, \url{https://doi.org/10.1080/14697688.2020.1788219}.

\bibitem{cheridito2014bsdes}
{\sc P.~Cheridito and K.~Nam}, {\em {BSDEs with terminal conditions that have bounded Malliavin derivative}}, J. Funct. Anal., 266 (2014), pp.~1257--1285, \url{https://doi.org/10.1016/j.jfa.2013.12.004}.

\bibitem{crisan2012solving}
{\sc D.~Crisan and K.~Manolarakis}, {\em {Solving Backward Stochastic Differential Equations Using the Cubature Method: Application to Nonlinear Pricing}}, {SIAM} J. Financial Math., 3 (2012), pp.~534--571, \url{https://doi.org/10.1137/090765766}.

\bibitem{cybenko1989approximation}
{\sc G.~Cybenko}, {\em Approximation by superpositions of a sigmoidal function}, Math. Control Signal Systems, 2 (1989), pp.~303--314, \url{https://doi.org/10.1007/BF02551274}.

\bibitem{delarue2006forward}
{\sc F.~Delarue and S.~Menozzi}, {\em {A forward–backward stochastic algorithm for quasi-linear PDEs}}, Ann. Appl. Probab., 16 (2006), pp.~140 -- 184, \url{https://doi.org/10.1214/105051605000000674}.

\bibitem{weinan2017deep}
{\sc W.~E, J.~Han, and A.~Jentzen}, {\em {Deep Learning-Based Numerical Methods for High-Dimensional Parabolic Partial Differential Equations and Backward Stochastic Differential Equations}}, Commun. Math. Stat., 5 (2017), pp.~349--380, \url{https://doi.org/10.1007/s40304-017-0117-6}.

\bibitem{weinan2019multilevel}
{\sc W.~E, M.~Hutzenthaler, A.~Jentzen, and T.~Kruse}, {\em {On Multilevel Picard Numerical Approximations for High-Dimensional Nonlinear Parabolic Partial Differential Equations and High-Dimensional Nonlinear Backward Stochastic Differential Equations}}, J. Sci. Comput., 79 (2019), pp.~1534--1571, \url{https://doi.org/10.1007/s10915-018-00903-0}.

\bibitem{Eyraud_Loisel_2005}
{\sc A.~Eyraud-Loisel}, {\em {Backward stochastic differential equations with enlarged filtration: Option hedging of an insider trader in a financial market with jumps}}, Stoch. Process Their Appl., 115 (2005), pp.~1745--1763, \url{https://doi.org/10.1016/j.spa.2005.05.006}.

\bibitem{Fahim_2011}
{\sc A.~Fahim, N.~Touzi, and X.~Warin}, {\em {A probabilistic numerical method for fully nonlinear parabolic {PDEs}}}, Ann. Appl. Probab., 21 (2011), pp.~1322--1364, \url{https://doi.org/10.1214/10-aap723}.

\bibitem{fu2017efficient}
{\sc Y.~Fu, W.~Zhao, and T.~Zhou}, {\em {Efficient spectral sparse grid approximations for solving multi-dimensional forward backward SDEs}}, Discrete Contin. Dyn. Syst. - B, 22 (2017), pp.~3439--3458, \url{https://doi.org/10.3934/dcdsb.2017174}.

\bibitem{fujii2019asymptotic}
{\sc M.~Fujii, A.~Takahashi, and M.~Takahashi}, {\em {Asymptotic Expansion as Prior Knowledge in Deep Learning Method for High dimensional BSDEs}}, Asia-Pac. Financ. Mark., 26 (2019), pp.~391--408, \url{https://doi.org/10.1007/s10690-019-09271-7}.

\bibitem{germain2022approximation}
{\sc M.~Germain, H.~Pham, and X.~Warin}, {\em {Approximation Error Analysis of Some Deep Backward Schemes for Nonlinear PDEs}}, SIAM J. Sci. Comput., 44 (2022), pp.~A28--A56, \url{https://doi.org/10.1137/20M1355355}.

\bibitem{gnoatto2022deep}
{\sc A.~Gnoatto, M.~Patacca, and A.~Picarelli}, {\em {A deep solver for BSDEs with jumps}}, 2022, \url{https://arxiv.org/abs/2211.04349}.

\bibitem{gnoatto2023deep}
{\sc A.~Gnoatto, A.~Picarelli, and C.~Reisinger}, {\em {Deep xVA Solver: A Neural Network--Based Counterparty Credit Risk Management Framework}}, {SIAM} J. Financial Math., 14 (2023), pp.~314--352, \url{https://doi.org/10.1137/21M1457606}.

\bibitem{gobet2010solving}
{\sc E.~Gobet and C.~Labart}, {\em {Solving BSDE with Adaptive Control Variate}}, {SIAM} J. Numer. Anal., 48 (2010), pp.~257--277, \url{https://doi.org/10.1137/090755060}.

\bibitem{gobet2005regression}
{\sc E.~Gobet, J.-P. Lemor, and X.~Warin}, {\em {A regression-based Monte Carlo method to solve backward stochastic differential equations}}, Ann. Appl. Probab., 15 (2005), pp.~2172--2202, \url{https://doi.org/10.1214/105051605000000412}.

\bibitem{gobet2016stratified}
{\sc E.~Gobet, J.~G. L{\'{o}}pez-Salas, P.~Turkedjiev, and C.~V{\'{a}}zquez}, {\em {Stratified Regression Monte-Carlo Scheme for Semilinear PDEs and BSDEs with Large Scale Parallelization on GPUs}}, SIAM J. Sci. Comput., 38 (2016), pp.~C652--C677, \url{https://doi.org/10.1137/16m106371x}.

\bibitem{han2018solving}
{\sc J.~Han, A.~Jentzen, and W.~E}, {\em Solving high-dimensional partial differential equations using deep learning}, Proc. Natl. Acad. Sci. U.S.A., 115 (2018), pp.~8505--8510, \url{https://doi.org/10.1073/pnas.1718942115}.

\bibitem{han2020convergence}
{\sc J.~Han and J.~Long}, {\em {Convergence of the deep BSDE method for coupled FBSDEs}}, Probab. Uncertain. Quant. Risk, 5 (2020), \url{https://doi.org/10.1186/s41546-020-00047-w}.

\bibitem{hornik1989multilayer}
{\sc K.~Hornik, M.~Stinchcombe, and H.~White}, {\em Multilayer feedforward networks are universal approximators}, Neural Netw., 2 (1989), pp.~359--366, \url{https://doi.org/10.1016/0893-6080(89)90020-8}.

\bibitem{huge2020differential}
{\sc B.~Huge and A.~Savine}, {\em {Differential Machine Learning}}, 2020, \url{https://arxiv.org/abs/2005.02347}.

\bibitem{hure2020deep}
{\sc C.~Hur{\'{e}}, H.~Pham, and X.~Warin}, {\em {Deep backward schemes for high-dimensional nonlinear PDEs}}, Math. Comput., 89 (2020), pp.~1547--1579, \url{https://doi.org/10.1090/mcom/3514}.

\bibitem{imkeller2010path}
{\sc P.~Imkeller and G.~D. Reis}, {\em {Path regularity and explicit convergence rate for BSDE with truncated quadratic growth}}, Stoch. Process Their Appl., 120 (2010), pp.~348--379, \url{https://doi.org/10.1016/j.spa.2009.11.004}.

\bibitem{ji2020three}
{\sc S.~Ji, S.~Peng, Y.~Peng, and X.~Zhang}, {\em {Three Algorithms for Solving High-Dimensional Fully Coupled FBSDEs Through Deep Learning}}, IEEE Intell. Syst., 35 (2020), pp.~71--84, \url{https://doi.org/10.1109/MIS.2020.2971597}.

\bibitem{ji2021novel}
{\sc S.~Ji, S.~Peng, Y.~Peng, and X.~Zhang}, {\em {A novel control method for solving high-dimensional Hamiltonian systems through deep neural networks}}, 2021, \url{https://arxiv.org/abs/2111.02636}.

\bibitem{ji2022deep}
{\sc S.~Ji, S.~Peng, Y.~Peng, and X.~Zhang}, {\em {A deep learning method for solving stochastic optimal control problems driven by fully-coupled FBSDEs}}, 2022, \url{https://arxiv.org/abs/2204.05796}.

\bibitem{jiang2021convergence}
{\sc Y.~Jiang and J.~Li}, {\em {Convergence of the Deep BSDE method for FBSDEs with non-Lipschitz coefficients}}, Probab. Uncertain. Quant. Risk, 6 (2021), pp.~391--408, \url{https://doi.org/10.3934/puqr.2021019}.

\bibitem{kapllani2022effect}
{\sc L.~Kapllani}, {\em {The Effect of the Number of Neural Networks on Deep Learning Schemes for Solving High Dimensional Nonlinear Backward Stochastic Differential Equations}}, in M. Ehrhardt, M. Günther (eds) Progress in Industrial Mathematics at ECMI 2021. ECMI 2021. Mathematics in Industry(), vol.~39, Springer, Cham, 2022, \url{https://doi.org/10.1007/978-3-031-11818-0_10}.

\bibitem{kapllani2022multistep}
{\sc L.~Kapllani and L.~Teng}, {\em {Multistep schemes for solving backward stochastic differential equations on GPU}}, J. Math. Ind., 12 (2022), \url{https://doi.org/10.1186/s13362-021-00118-3}.

\bibitem{kapllani2024deep}
{\sc L.~Kapllani and L.~Teng}, {\em {Deep learning algorithms for solving high-dimensional nonlinear backward stochastic differential equations}}, Discrete Contin. Dyn. Syst. - B, 29 (2024), pp.~1695--1729, \url{https://doi.org/10.3934/dcdsb.2023151}.

\bibitem{kapllaniuncertainty}
{\sc L.~Kapllani, L.~Teng, and M.~Rottmann}, {\em {Uncertainty quantification for deep learning-based schemes for solving high-dimensional backward stochastic differential equations}}, 2023, \url{https://arxiv.org/abs/2310.03393}.

\bibitem{El1997}
{\sc N.~E. Karoui, S.~Peng, and M.~C. Quenez}, {\em {Backward Stochastic Differential Equations in Finance}}, Math. Financ., 7 (1997), pp.~1--71, \url{https://doi.org/10.1111/1467-9965.00022}.

\bibitem{kloeden2013numerical}
{\sc P.~E. Kloeden and E.~Platen}, {\em {Numerical Solution of Stochastic Differential Equations}}, Springer, 2013, \url{https://doi.org/10.1007/978-3-662-12616-5}.

\bibitem{kremsner2020deep}
{\sc S.~Kremsner, A.~Steinicke, and M.~Sz{\"o}lgyenyi}, {\em {A Deep Neural Network Algorithm for Semilinear Elliptic PDEs with Applications in Insurance Mathematics}}, Risks, 8 (2020), p.~136, \url{https://doi.org/10.3390/risks8040136}.

\bibitem{labart2011parallel}
{\sc C.~Labart and J.~Lelong}, {\em {A Parallel Algorithm for solving BSDEs-Application to the pricing and hedging of American options}}, 2011, \url{https://arxiv.org/abs/1102.4666v1}.

\bibitem{lefebvre2023differential}
{\sc W.~Lefebvre, G.~Loeper, and H.~Pham}, {\em {Differential learning methods for solving fully nonlinear PDEs}}, Digit. Finance, 5 (2023), pp.~183--229, \url{https://doi.org/10.1007/s42521-023-00077-x}.

\bibitem{lemor2006rate}
{\sc J.-P. Lemor, E.~Gobet, and X.~Warin}, {\em Rate of convergence of an empirical regression method for solving generalized backward stochastic differential equations}, Bernoulli, 12 (2006), pp.~889--916, \url{https://doi.org/10.3150/bj/1161614951}.

\bibitem{liang2021deep}
{\sc J.~Liang, Z.~Xu, and P.~Li}, {\em Deep learning-based least squares forward-backward stochastic differential equation solver for high-dimensional derivative pricing}, Quant. Finance, 21 (2021), pp.~1309--1323, \url{https://doi.org/10.1080/14697688.2021.1881149}.

\bibitem{ma2008numerical}
{\sc J.~Ma, J.~Shen, and Y.~Zhao}, {\em {On Numerical Approximations of Forward-Backward Stochastic Differential Equations}}, {SIAM} J. Numer. Anal., 46 (2008), pp.~2636--2661, \url{https://doi.org/10.1137/06067393x}.

\bibitem{negyesi2024one}
{\sc B.~Negyesi, K.~Andersson, and C.~W. Oosterlee}, {\em {The One Step Malliavin scheme: new discretization of BSDEs implemented with deep learning regressions}}, IMA J. Numer. Anal.,  (2024), p.~drad092, \url{https://doi.org/10.1093/imanum/drad092}.

\bibitem{negyesi2024generalized}
{\sc B.~Negyesi, Z.~Huang, and C.~W. Oosterlee}, {\em {Generalized convergence of the deep BSDE method: a step towards fully-coupled FBSDEs and applications in stochastic control}}, 2024, \url{https://arxiv.org/abs/2403.18552}.

\bibitem{nualart2006malliavin}
{\sc D.~Nualart}, {\em {The Malliavin Calculus and Related Topics}}, vol.~1995, Springer, 2006, \url{https://doi.org/10.1007/3-540-28329-3}.

\bibitem{Pardoux1990}
{\sc E.~Pardoux and S.~Peng}, {\em Adapted solution of a backward stochastic differential equation}, Syst. Control. Lett., 14 (1990), pp.~55--61, \url{https://doi.org/10.1016/0167-6911(90)90082-6}.

\bibitem{pereira2019learning}
{\sc M.~Pereira, Z.~Wang, I.~Exarchos, and E.~A. Theodorou}, {\em {Learning Deep Stochastic Optimal Control Policies using Forward-Backward SDEs}}, 2021, \url{https://arxiv.org/abs/1902.03986}.

\bibitem{pham2021neural}
{\sc H.~Pham, X.~Warin, and M.~Germain}, {\em {Neural networks-based backward scheme for fully nonlinear PDEs}}, SN Partial Differ. Equ. Appl., 2 (2021), \url{https://doi.org/10.1007/s42985-020-00062-8}.

\bibitem{raissi2024forward}
{\sc M.~Raissi}, {\em {Forward--backward stochastic neural networks: deep learning of high-dimensional partial differential equations}}, in Peter Carr Gedenkschrift: Research Advances in Mathematical Finance, World Scientific, 2024, pp.~637--655, \url{https://doi.org/10.1142/9789811280306_0018}.

\bibitem{ruijter2016fourier}
{\sc M.~Ruijter and C.~Oosterlee}, {\em {Numerical Fourier method and second-order Taylor scheme for backward SDEs in finance}}, Appl. Numer. Math., 103 (2016), pp.~1--26, \url{https://doi.org/10.1016/j.apnum.2015.12.003}.

\bibitem{ruijter2015fourier}
{\sc M.~J. Ruijter and C.~W. Oosterlee}, {\em {A Fourier Cosine Method for an Efficient Computation of Solutions to BSDEs}}, {SIAM} J. Sci. Comput., 37 (2015), pp.~A859--A889, \url{https://doi.org/10.1137/130913183}.

\bibitem{takahashi2022new}
{\sc A.~Takahashi, Y.~Tsuchida, and T.~Yamada}, {\em {A new efficient approximation scheme for solving high-dimensional semilinear PDEs: Control variate method for Deep BSDE solver}}, J. Comput. Phys., 454 (2022), p.~110956, \url{https://doi.org/10.1016/j.jcp.2022.110956}.

\bibitem{teng2021review}
{\sc L.~Teng}, {\em {A Review of Tree-Based Approaches to Solving Forward--Backward Stochastic Differential Equations}}, J. Comput. Finance, 25 (2021), \url{https://doi.org/10.21314/JCF.2021.010}.

\bibitem{teng2022gradient}
{\sc L.~Teng}, {\em {Gradient boosting-based numerical methods for high-dimensional backward stochastic differential equations}}, Appl. Math. Comput., 426 (2022), p.~127119, \url{https://doi.org/10.1016/j.amc.2022.127119}.

\bibitem{teng2020multi}
{\sc L.~Teng, A.~Lapitckii, and M.~G{\"u}nther}, {\em {A multi-step scheme based on cubic spline for solving backward stochastic differential equations}}, Appl. Numer. Math., 150 (2020), pp.~117--138, \url{https://doi.org/10.1016/j.apnum.2019.09.016}.

\bibitem{teng2021high}
{\sc L.~Teng and W.~Zhao}, {\em {High-order combined multi-step scheme for solving forward backward stochastic differential equations}}, J. Sci. Comput., 87 (2021), \url{https://doi.org/10.1007/s10915-021-01505-z}.

\bibitem{wang2018deep}
{\sc H.~Wang, H.~Chen, A.~Sudjianto, R.~Liu, and Q.~Shen}, {\em {Deep Learning-Based BSDE Solver for Libor Market Model with Application to Bermudan Swaption Pricing and Hedging}}, 2018, \url{https://arxiv.org/abs/1807.06622}.

\bibitem{zhang2013sparse}
{\sc G.~Zhang}, {\em {A Sparse-Grid Method for Multi-Dimensional Backward Stochastic Differential Equations}}, J. Comput. Math., 31 (2013), pp.~221--248, \url{https://doi.org/10.4208/jcm.1212-m4014}.

\bibitem{zhang2004numerical}
{\sc J.~Zhang}, {\em {A numerical scheme for BSDEs}}, Ann. Appl. Probab., 14 (2004), pp.~459--488, \url{https://doi.org/10.1214/aoap/1075828058}.

\bibitem{zhang2017backward}
{\sc J.~Zhang}, {\em {Backward Stochastic Differential Equations}}, Springer, 2017, \url{https://doi.org/10.1007/978-1-4939-7256-2}.

\bibitem{zhao2006new}
{\sc W.~Zhao, L.~Chen, and S.~Peng}, {\em {A New Kind of Accurate Numerical Method for Backward Stochastic Differential Equations}}, {SIAM} J. Sci. Comput., 28 (2006), pp.~1563--1581, \url{https://doi.org/10.1137/05063341x}.

\bibitem{zhao2014new}
{\sc W.~Zhao, Y.~Fu, and T.~Zhou}, {\em {New Kinds of High-Order Multistep Schemes for Coupled Forward Backward Stochastic Differential Equations}}, {SIAM} J. Sci. Comput., 36 (2014), pp.~A1731--A1751, \url{https://doi.org/10.1137/130941274}.

\bibitem{zhao2010stable}
{\sc W.~Zhao, G.~Zhang, and L.~Ju}, {\em {A Stable Multistep Scheme for Solving Backward Stochastic Differential Equations}}, {SIAM} J. Numer. Anal., 48 (2010), pp.~1369--1394, \url{https://doi.org/10.1137/09076979x}.

\end{thebibliography}

\end{document}